\documentclass{article}

\title{\bf Intuitionistic monotone modal logic via translation}
\author{Jim de Groot}
\date{}

\usepackage{mathptmx}
\usepackage[dvipsnames]{xcolor}
\usepackage[colorlinks=true,
            linktocpage=true,
            citecolor=RubineRed,
            linkcolor=RoyalBlue]{hyperref}
\usepackage[numbers]{natbib}
\usepackage{doi}
\usepackage{bussproofs}

\usepackage{amsmath}
\usepackage{amssymb}
\usepackage{amsthm}
  \theoremstyle{definition}
  \swapnumbers
    \newtheorem{para}{}[section]
    \newtheorem{definition}[para]{Definition}
    \newtheorem{example}[para]{Example}
    \newtheorem{remark}[para]{Remark}
  \theoremstyle{plain}
    \newtheorem{lemma}[para]{Lemma}
    
    \newtheorem{corollary}[para]{Corollary}
    \newtheorem{theorem}[para]{Theorem}
    \newtheorem{proposition}[para]{Proposition}

  \swapnumbers

\usepackage{titlesec}
\titleformat{\section}{\large\centering\bfseries}{\thesection}{1em}{}
\titleformat{\subsection}{\large\bfseries}{\thesubsection}{1em}{}

\usepackage[a4paper, margin=2.5cm]{geometry}
\usepackage{stmaryrd}
\usepackage{tikz-cd}
  \usetikzlibrary{arrows,arrows.meta}
  \tikzcdset{arrow style=tikz, diagrams={>=stealth'}}
\usepackage{booktabs}
\usepackage{enumerate}
\usepackage{multicol}

\usepackage{mathrsfs}
\usepackage[cal=txupr,scr=boondoxupr,scrscaled=1.050]{mathalfa}

\newcommand{\mc}[1]{\mathcal{#1}}
\newcommand{\mb}[1]{\mathbb{#1}}

\newcommand{\ms}[1]{\mathscr{#1}}
\newcommand{\mo}[1]{\mathcal{#1}}  % for models
\newcommand{\fomo}[1]{\mathfrak{#1}} % for first order models
\newcommand{\lan}[1]{\mathbf{#1}}   % for languages
\renewcommand{\log}[1]{\mathsf{#1}}   % for logics

\newcommand{\fun}[1]{\mathcal{#1}}

\renewcommand{\iff}{\quad\text{iff}\quad}
\renewcommand{\phi}{\varphi}
\DeclareMathOperator{\Prop}{Prop}

\newcommand{\llb}{\llbracket}
\newcommand{\rrb}{\rrbracket}
\DeclareMathOperator{\up}{up}
\DeclareMathOperator{\st}{st}

\newcommand{\Ax}{\mathscr{Ax}}
\newcommand{\REl}{\mathsf{El}}
\newcommand{\RAx}{\mathsf{Ax}}
\newcommand{\RMP}{\mathsf{MP}}

\newcommand{\ADualI}{\mathsf{I_{\Diamond\!\Box}}}
\newcommand{\RMonB}{\mathsf{Mon_{\Box}}}
\newcommand{\RMonD}{\mathsf{Mon_{\Diamond}}}
\newcommand{\RMonM}{\mathsf{Mon_{\Mon}}}
\newcommand{\ANega}{\mathsf{neg_a}}

\newcommand{\AInt}{\mathsf{I_{\Diamond}}}
\newcommand{\RStr}{\mathsf{str}}
\newcommand{\RNec}{\mathsf{Nec}}

\newcommand{\AKbox}{\mathsf{K_{\Box}}}
\newcommand{\AKdia}{\mathsf{K_{\Diamond}}}
\newcommand{\ACdia}{\mathsf{C_{\Diamond}}}
\newcommand{\ANdia}{\mathsf{N_{\Diamond}}}

\newcommand{\iM}{\log{iM}}

\newcommand{\IMf}{\log{IM}}
\newcommand{\IMfC}{\log{IM_{Calc}}}
\newcommand{\IK}{\log{IK}}

\newcommand{\M}{\log{M}}
\newcommand{\K}{\log{K}}
\newcommand{\IMC}{\log{GHC}}

\newcommand{\WM}{\log{WM}}

\newcommand{\Lbd}{\lan{L}_{\Box\!\Diamond}}

\newcommand{\pow}{\fun{P}}

\renewcommand{\preceq}{\preccurlyeq}

\newcommand{\nf}{\gamma} %% for Neighbourhood Function

\newcommand{\goodbox}{\hspace{.2ex}\text{%
  \tikz[baseline=-.6ex, rounded corners=.01ex, line width=.1ex]
    {\draw (-.6ex,-.6ex) rectangle (.6ex,.6ex);}}\kern.2ex}
\newcommand{\gooddiamond}{\hspace{.2ex}\text{%
  \tikz[baseline=-.6ex, rounded corners=.01ex, rotate=45, line width=.1ex]
    {\draw (-.5ex,-.5ex) rectangle (.5ex,.5ex);}}\kern.2ex}
\renewcommand{\Box}{\goodbox}
\renewcommand{\Diamond}{\gooddiamond}

\newcommand{\BoxN}{\Box_N}
\newcommand{\BoxI}{\Box_{\ni}}
\newcommand{\DiamondN}{\Diamond_N}
\newcommand{\DiamondI}{\Diamond_{\ni}}

\DeclareFontEncoding{LS2}{}{\noaccents@}
\DeclareFontSubstitution{LS2}{stix}{m}{n}
\DeclareSymbolFont{arrows3}       {LS2}{stixtt}   {m} {n}
\DeclareMathSymbol{\Star}{\mathord}{arrows3}{"A0}
\newcommand{\N}{\mathsf{N}}
\newcommand{\E}{\mathsf{E}}
\renewcommand{\P}{\mathsf{P}}
\newcommand{\n}{\mathsf{n}}
\newcommand{\s}{\mathsf{s}}

\newcommand{\I}{\mathcal{I}}

\newcommand{\cleq}{\mathbin{\subset\kern-7.4pt<}}

\DeclareMathOperator{\dom}{dom}
\DeclareMathOperator{\last}{last}
\DeclareMathOperator{\first}{first}
\DeclareMathOperator{\length}{length}

\newcommand{\bul}{\bullet}
\newcommand{\glue}{\smallfrown}

\newcommand{\FOM}{\mathsf{FOM}}
\newcommand{\IFOM}{\mathsf{IFOM}}

\newcommand{\coh}{\mathrm{coh}}
\newcommand{\Mon}{\triangledown}

\newcommand{\inm}{\mathrm{inm}}

\begin{document}

\maketitle

\begin{abstract}
\noindent
  We introduce a monotone modal analogue of the intuitionistic (normal) modal
  logic $\log{IK}$ using a translation into a suitable (intuitionistic) first-order logic.
  We axiomatise the logic and give a semantics by means of intuitionistic
  neighbourhood models, which contain neighbourhoods whose value can change
  when moving along the intuitionistic accessibility relation.
  We compare the resulting logic with other intuitionistic monotone modal
  logics and show how it can be embedded into a multimodal version of $\log{IK}$.
\end{abstract}

%\tableofcontents

%%%%%%%%%%%%%%%%%%%%%%%%%%%%%%%%%%%%%%%%%%%%%%%%%%%%%%%%%%%%%%%%%%%%%%%%%%%%%%%%%
\section{Introduction}

  Ever since Fitch first introduced a modal extension of intuitionistic
  logic~\cite{Fit48},
  many different systems of intuitionistic modal logic have been put forward
  and investigated~\cite{Bul65,Bul66,Fis77,Fis81,BozDos84,Ewa86,PloSti86,Sim94,WolZak97,WolZak98,WolZak99,Dav09}.
  The various approaches often differ on the relation between the box and diamond modality.
  On one side of this spectrum sit for example logics studied by 
  Sotirov~\cite[Section~4]{Sot84} and Wolter and Zakharyaschev~\cite[Section~2]{WolZak99},
  in which $\Box$ and $\Diamond$ are completely unrelated.
  At the other extremity, we have systems
  studied by Bull~\cite{Bul65b} and
  Bo\v{z}i\'{c} and Do\v{s}en~\cite[Section~11]{BozDos84},
  in which the diamond operator can be viewed as an abbreviation for $\neg\Box\neg$.
  
  There are two prominent ways to establish a more subtle interaction between
  the modalities: via a proof system and via a (standard) translation.
  These give rise to \emph{constructive} and \emph{intuitionistic} counterparts of
  classical modal logics.
  Constructive versions of the classical normal modal logic
  $\log{K}$ were introduced by Bellin, De Paiva and
  Ritter~\cite{BelPaiRit01} and by Wijesekera~\cite{Wij90},
  and are obtained by restricting the sequent calculus for $\log{K}$
  to single conclusions, or to zero or one conclusions.
  Besides their proof-theoretic perspectives, both logics are also
  motivated by applications in AI: $\log{CK}$ captures the notion of contex
  in knowledge representation and reasoning~\cite{Pai03,MenPai05}
  and $\log{WK}$ can be used to represent states with partial knowledge~\cite{WijNer05}.

  Intuitionistic $\log{K}$ ($\log{IK}$) was introduced by Fischer Servi~\cite{Fis77,Fis81},
  and further studied by Plotkin and Sterling~\cite{PloSti86}, Ewald~\cite{Ewa86}
  and Simpson~\cite{Sim94}. It can be viewed as the set of formulas whose
  standard translation is derivable in \emph{intuitionistic} first-order
  logic~\cite{Sim94}, or equivalently, whose translation into a suitable
  classical bimodal logic is derivable~\cite{Fis77}.
  The relation between $\log{CK}$, $\log{IK}$ and various logics in between
  is investigated in~\cite{DasMar23,GroShiClo24}.

  Besides intuitionistic analogues of $\log{K}$,
  extensions of intuitionistic logic with monotone modal operators have
  been studied as well. Monotone modal operators are modalities $\Mon$ for
  which $\phi \to \psi$ implies $\Mon\phi \to \Mon\psi$, or equivalently,
  that satisfy $\Mon(\phi \wedge \psi) \to \Mon\phi$.
  The study of these in an intuitionistic context dates back to
  Bull~\cite[Page~5]{Bul65}, who studied a monotone modality that
  satisfies the S4 axioms.
  Discarding the S4 axioms yields Goldblatt's intuitionistic logic
  with a geometric modality~\cite{Gol81,Gol93}, denoted here by $\iM$,
  which was further researched in e.g.~\cite{GroPat20,TabIemJal22,Gro22gt}.
  More recently, extensions of intuitionistic logic with two monotone modal operators were
  investigated by Dalmonte, Grellois and Olivetti~\cite{DalGreOli20,DalGreOli21}.
  Some of these can be viewed as \emph{constructive} monotone modal logics,
  in the sense that they can be obtained by restricting a sequent calculus
  for classical monotone modal logic~\cite{Dal22}.
  Surprisingly, it appears that intuitionistic monotone modal logics arising
  from a translation into a suitable intuitionistic first-order logic have
  not yet been studied. This is the gap we are aiming to fill in this paper.
  
  The first step towards this goal is to decide on a suitable first-order logic.
  Since monotone modal logic is usually interpreted in neighbourhood
  models and we need our first-order structures to mimic these,
  we use a two-sorted first-order logic $\FOM$, with the
  first sort representing worlds and the second representing neighbourhoods.
  These are related using predicates $\N$ and $\E$,
  where $w \N a$ and $a \E w$ are intuitively read as ``$a$ is a neighbourhood of $w$''
  and ``$w$ is an element of $a$,'' respectively.
  This way of viewing neighbourhood structures stems from~\cite{FluZie80},
  and has since been used regularly~\cite{Han03,CatGabSus09,HanKupPac09,Gro22-inl}.
  It can be shown that classical monotone modal logic is the set of formulas
  whose standard translation is derivable in $\FOM$.
  We then obtain an intuitionistic monotone modal logic as the set of formulas
  whose standard translation is derivable in $\IFOM$, the intuitionistic
  first-order logic with the same signature as $\FOM$.
  In fact, to give us more flexibility, we define the logic $\IMf$ as a collection
  of consecutions, i.e.~expressions of the form $\Gamma \vdash \phi$ where $\Gamma$ is a set
  of formulas and $\phi$ is a single formula.
  
  The next objective is to axiomatise $\IMf$.
  It turns out that $\IMf$ is stronger than its constructive counterpart,
  and that it can be obtained by extending the logic $\WM$ from~\cite{Dal22}
  with the interaction axiom
  $(\Box\top \to \Diamond\phi) \to \Diamond\phi$.
  This mirrors the fact that $\IK$ can be obtained by extending $\log{WK}$.
  In order to prove completeness of the axiomatisation, we use a semantic detour.
  We know that $\IFOM$-structures (trivially) provide a sound and complete semantics
  for~$\IMf$. Guided by these, we define \emph{intuitionistic neighbourhood models}
  and we show that these validate precisely the same sequents as the $\IFOM$-structures.
  Completeness then follows from a routine canonical model construction
  where we build an intuitionistic neighbourhood model.
  
  An intuitionistic neighbourhood model consist of a (partially ordered)
  intuitionistic Kripke frame, a valuation, and a set of
  \emph{intuitionistic neighbourhoods}: these are neighbourhoods that can
  change when moving along the intuitionistic accessibility relation.
  This reflects the universal quantification over the ``neighbourhood sort''
  of $\IFOM$, which relies on the shape of the neighbourhood in all accessible worlds.
  When viewing neighbourhoods as representing evidence, intuitionistic
  neighbourhoods embody the idea that a piece evidence may change when
  moving among possible worlds.
  Building the bridge between $\IFOM$-structures and intuitionistic
  neighbourhood models is one of the main technical challenges of the paper.
  
  In an effort to position the logic $\IMf$ into the logical landscape,
  we compare it to various other intuitionistic modal logics.
  First, we consider the embedding of $\iM$ into the diamond-free fragments and box-free
  fragments of $\WM$ and $\IMf$. This gives rise to various (non-)conservativity
  results, again mirroring the relation between constructive and
  intuitionistic $\log{K}$~\cite{DasMar23,GroShiClo24}.
  Second, we show that $\IMf$ can be embedded into a multimodal version of $\IK$,
  in the sense that a sequent is derivable in $\IMf$ if and only if
  its multimodal translation is derivable.
  This can be viewed as an intuitionistic analogue of the embedding of
  classical monotone modal logic into bimodal $\log{K}$,
  which is investigated in~\cite[Section~5.2]{Han03} and in more
  generality in~\cite{GasHer96,KraWol99}.

\paragraph{Structure of the paper}
  In Section~\ref{sec:2} we define the logic $\IMf$
  using the translation into the intuitionistic first-order logic $\IFOM$. 
  In Section~\ref{sec:semantics} we investigate various semantics for $\IMf$,
  and use a canonical model construction to prove completeness.
  Then, in Section~\ref{sec:relation}, we investigate the relation between
  $\iM$ and fragments of $\WM$ and $\IMf$,
  and we show how to embed $\IMf$ into multimodal $\IK$.
  We conclude in Section~\ref{sec:conc}.

%%%%%%%%%%%%%%%%%%%%%%%%%%%%%%%%%%%%%%%%%%%%%%%%%%%%%%%%%%%%%%%%%%%%%%%%%%%%%%%%%
\section{From classical to intuitionistic monotone modal logic}\label{sec:2}

  As outlined in the introduction, in order to obtain an intuitionistic counterpart
  of monotone modal logic, we start by taking a first-order perspective on classical
  monotone modal logic and then change the first-order logic to an
  intuitionistic one.
  
  Throughout this paper, we let $\Lbd$ denote the language generated by
  the grammar
  \begin{equation*}
    \phi ::= p_i \mid \bot \mid \phi \wedge \phi \mid \phi \vee \phi \mid \phi \to \phi \mid \Box\phi \mid \Diamond\phi,
  \end{equation*}
  where the $p_i$ range over some set $\Prop$ of proposition letters.
  Note that this can be the language underlying both classical an intuitionistic
  modal logics.

%================================================================================
\subsection{Classical monotone modal logic}

  Monotone modal logic $\log{M}$ is the extension of classical propositional
  logic with two dual monotone modalities $\Box$ and
  $\Diamond$~\cite{Che80,Han03,HanKup04}.
  That is, it arises from extending an axiomatisation of classical propositional logic
  with the duality axiom $\Box\phi \leftrightarrow \neg\Diamond\neg\phi$
  and the monotonicity rules
  \begin{equation*}
    \dfrac{\emptyset \vdash \phi \to \psi}
          {\Gamma \vdash \Box\phi \to \Box\psi} \; (\RMonB)
    \qquad\text{and}\qquad
    \dfrac{\emptyset \vdash \phi \to \psi}
          {\Gamma \vdash \Diamond\phi \to \Diamond\psi} \; (\RMonD).
  \end{equation*}
  The logic is usually interpreted in neighbourhood semantics.
  There are various ways to do so. Most importantly, the monotonicity of the
  modalities can either be built into the frames~\cite{Che80,Han03},
  or into their interpretation~\cite{Lew73,AreFig09}.
  We use the latter, as this aligns best with our first-order intentions.

\begin{definition}
  A \emph{neighbourhood frame} is a pair $(W, N)$ consisting of a nonempty
  set $W$ of worlds and a \emph{neighbourhood function} $\nf : W \to \pow\pow(W)$,
  where $\pow(W)$ denotes the powerset of $W$.
  A \emph{neighbourhood model} is a tuple $\mo{M} = (W, \nf, V)$ consisting of a
  neighbourhood frame $(W, \nf)$ and a valuation $V : \Prop \to \pow(W)$
  of the proposition letters.
  The interpretation of $\Lbd$-formulas at a world $w$ in a neighbourhood
  model $\mo{M} = (W, N, V)$ is given recursively by
  \begin{align*}
    \mo{M}, w \Vdash p
      &\iff w \in V(p) \\
    \mo{M}, w \Vdash \bot
      &\phantom{\iff}\text{never} \\
    \mo{M}, w \Vdash \phi \wedge \psi
      &\iff \mo{M}, w \Vdash \phi \text{ and } \mo{M}, w \Vdash \psi \\
    \mo{M}, w \Vdash \phi \vee \psi
      &\iff \mo{M}, w \Vdash \phi \text{ or } \mo{M}, w \Vdash \psi \\
    \mo{M}, w \Vdash \phi \to \psi
      &\iff \mo{M}, w \not\Vdash \phi \text{ or } \mo{M}, w \Vdash \psi \\
    \mo{M}, w \Vdash \Box\phi
      &\iff \text{there exists } a \in \nf(w)
            \text{ such that for all } v \in a
            \text{ we have } \mo{M}, v \Vdash \phi \\
    \mo{M}, w \Vdash \Diamond\phi
      &\iff \text{for all } a \in \nf(w)
            \text{ there exists } v \in a
            \text{ such that } \mo{M}, v \Vdash \phi
  \end{align*}
\end{definition}

  We use a two-sorted first-order correspondence language, with one sort representing
  the worlds and the other the neighbourhoods~\cite{FluZie80,HanKupPac09,Gro22-inl}.

\begin{definition}
  Let $\FOM$ be the two-sorted first-order logic with sorts $\s$ and $\n$,
  a binary predicate $\N$ between $\s$ and $\n$,
  a binary predicate $\E$ between $\n$ and $\s$,
  and a unary predicate $\P_i$ of type $\s$ for each $p_i \in \Prop$.
\end{definition}

  Then a $\FOM$-structure is a tuple $\fomo{M} = (\I(\s), \I(\n), \I(\N), \I(\E), \I(\P_i))$
  consisting of interpretations $\I(\s)$ and $\I(\n)$ of the sorts as nonempty sets,
  and interpretations
  $\I(\N) \subseteq \I(\s) \times \I(\n)$ and
  $\I(\E) \subseteq \I(\n) \times \I(\s)$
  and $\I(\P_i) \subseteq \I(\s)$ of the predicates.
  Every neighbourhood model $\mo{M} = (W, \nf, V)$ gives rise
  to a $\FOM$-structure $\mo{M}^{\circ}$ with interpretations:
  \begin{align*}
    \I(\s) &= W
      &(w, a) \in \I(\N) &\iff a \in \nf(w) \\
    \I(\n) &= \bigcup\{ \nf(w) \mid w \in W \}
      &(a, w) \in \I(\E) &\iff w \in a \\
      &
      &w \in \I(\P_i) &\iff w \in V(p_i)
  \end{align*}
  In the converse direction, a $\FOM$-structure 
  $\fomo{M} = (\I(\s), \I(\n), \I(\N), \I(\E), \I(\P_i))$
  yields a neighbourhood model $\mo{M}^{\bullet} = (\I(\s), \gamma, V)$,
  where $V(p_i) = \I(\P_i)$ and 
  \begin{equation*}
    \gamma(w) = \big\{ \{ y \in \I(\s) \mid a \E y \} \mid w \N a \big\}.
  \end{equation*}

\begin{definition}\label{def:st}
  The standard translation $\st_x : \Lbd \to \FOM$ is defined recursively by
  \begin{align*}\label{eq:st-mod}
    \st_x(p_i) &= P_ix \\
    \st_x(\bot) &= (x \neq x)
      &\st_x(\phi \to \psi) &= \st_x(\phi) \to \st_x(\psi) \\
    \st_x(\phi \wedge \psi) &= \st_x(\phi) \wedge \st_x(\psi)
      &\st_x(\Box\phi) &= (\exists a) (x \N a \wedge (\forall y) (a \E y \to \st_y(\phi))) \\
    \st_x(\phi \vee \psi) &= \st_x(\phi) \vee \st_x(\psi)
      &\st_x(\Diamond\phi) &= (\forall a)(x \N a \to (\exists y)( a \E y \wedge \st_y(\phi)))
  \end{align*}
\end{definition}
  
  Then we have:
  
\begin{proposition}
  For all $\phi \in \Lbd$, neighbourhood models $\mo{M}$ and
  $\FOM$-structures $\fomo{N}$:
  \begin{equation*}
    \mo{M}, w \Vdash \phi \iff \mo{M}^{\circ} \models \st_x(\phi)[w],
    \quad\quad
    \fomo{N}^{\bullet}, w \Vdash \phi \iff \fomo{N} \models \st_x(\phi)[w].
  \end{equation*}
\end{proposition}

  Therefore, we can view the classical monotone modal logic $\log{M}$
  as the set of $\Lbd$-formulas whose standard translation
  is derivable in $\FOM$.

%================================================================================
\subsection{Intuitionistic first-order logic}

  Let $\IFOM$ be the intuitionistic first-order logic of the same signature
  as $\FOM$.
  Then an $\IFOM$-structure is a tuple $\fomo{M} = (W, \leq, \I)$ consisting of
  a partially ordered set $(W, \leq)$ together
  with an interpretation function $\I$ that assigns to each world $w \in W$
  sets $\I(\s, w)$ and $\I(\n, w)$, and relations
  $\I(\N, w) \subseteq \I(\s, w) \times \I(\n, w)$ and
  $\I(\E, w) \subseteq \I(\n, w) \times \I(\s, w)$ and
  $\I(\P_i, w) \subseteq \I(\s, w)$,
  such that $w \leq w'$ implies $\I(\mathsf{X}, w) \subseteq \I(\mathsf{X}, w')$
  for all $\mathsf{X} \in \{ \s, \n, \N, \E, \P_i \}$.
  In other words, for each $w \in W$ we get a (classical) $\FOM$-structure
  $(\I(\s, w), \I(\n, w), \I(\E, w), \I(\N, w), \I(\P_i, w))$, and these
  increase with respect to inclusion as we go up along the intuitionistic
  accessibility relation.
  
  If $\alpha(x)$ is a first-order formula with free variable $x$,
  and $d$ is a domain element of some first-order structure $\fomo{M}$,
  then we denote by $\alpha(x)[d]$ the formula obtained from interpreting
  $x$ as $d$.

\begin{example}
  The following diagram depicts an $\IFOM$-structure:
    \begin{center}
    \begin{tikzpicture}[scale=.85,yscale=.85]
      %% boxes
        \draw[rounded corners=2mm,fill=blue!10] (1.6,.7) rectangle (4.4,.1)
              (4.6,.4) node[right]{\footnotesize{$\I(\s, w_1)$}};
        \draw[rounded corners=2mm,fill=red!10] (2,-.1) rectangle (4,-.7)
              (4.2,-.4) node[right]{\footnotesize{$\I(\n, w_1)$}};
        \draw[rounded corners=2mm,fill=blue!10] (1.1,2.7) rectangle (5.9,2.1)
              (6.1,2.4) node[right]{\footnotesize{$\I(\s, w_2)$}};
        \draw[rounded corners=2mm,fill=red!10] (1.5,1.9) rectangle (5.5,1.3)
              (5.7,1.6) node[right]{\footnotesize{$\I(\n, w_2)$}};
        \draw[rounded corners=2mm,fill=blue!10] (.6,4.7) rectangle (7.4,4.1)
              (7.6,4.4) node[right]{\footnotesize{$\I(\s, w_3)$}};
        \draw[rounded corners=2mm,fill=red!10] (1,3.9) rectangle (7,3.3)
              (7.2,3.6) node[right]{\footnotesize{$\I(\n, w_3)$}};
      %% nodes
        \node (w1) at (0,0) {$w_1$};
          \node (x11) at (2,.4) {$d_1$};
          \node (x12) at (4,.4) {$d_2$};
          \node (a11) at (3,-.4) {$a_1$};
        \node (w2) at (-.5,2) {$w_2$};
          \node (x21) at (1.5,2.4) {$d_1$};
          \node (x22) at (3.5,2.4) {$d_2$};
          \node (x23) at (5.5,2.4) {$d_3$};
          \node (a21) at (2.5,1.6) {$a_1$};
        \node (w3) at (-1,4) {$w_3$};
          \node (x31) at (1,4.4) {$d_1$};
          \node (x32) at (3,4.4) {$d_2$};
          \node (x33) at (5,4.4) {$d_3$};
          \node (x34) at (7,4.4) {$d_4$};
          \node (a31) at (2,3.6) {$a_1$};
          \node (a32) at (5,3.6) {$a_2$};
      %% edges
        \draw[-latex] (w1) to node[left]{\footnotesize{$\leq$}} (w2);
        \draw[-latex] (w2) to node[left]{\footnotesize{$\leq$}} (w3);
      %% relations w1
        \draw[-latex, bend right=25] (x11) to (a11);
        \draw[-latex, bend right=20] (a11) to (x12);
      %% relations w2
        \draw[-latex, bend right=25] (x21) to (a21);
        \draw[-latex, bend right=25] (x22) to (a21);
        \draw[-latex, bend right=20] (a21) to (x22);
        \draw[-latex, bend right=15] (a21) to (x23);
      %% relations w3
        \draw[-latex, bend right=25] (x31) to (a31);
        \draw[-latex, bend right=25] (x32) to (a31);
        \draw[-latex, bend right=20] (a31) to (x32);
        \draw[-latex, bend right=15] (a31) to (x33);
        \draw[-latex, bend right=10] (x32) to (a32);
        \draw[-latex, bend right=15] (a32) to (x34);
    \end{tikzpicture}
  \end{center}
  Suppose $\I(\P_i, w_2) = \I(\P_i, w_2) = \I(\P_i, w_3) = \{ d_2 \}$.
  Then $w_1 \models (\forall a)(d_1 \N a \to (\exists y)(a \E y \wedge \P_iy))$:
  the intuitionistic interpretation of the universal quantifier requires
  us to verify for each successor $v$ of $w_1$ that $x \N a$ implies the existence
  of some $y$ such that $a \E y$ and $\P_i y$.
  And this is indeed the case, because at each ``level'' $w_1, w_2, w_3$, 
  the state $d_1$ is only connected to neighbourhood $a_1$, and in each case
  $d_2$ is a suitable $y$.
  Since the given formula is the standard translation of
  $\Diamond p_i$ with $x$ interpreted as $d_1$, we find that
  $w_1 \models \st_x(\Diamond p_i)[d_1]$.
\end{example}

  The standard translation $\st_x$ from Definition~\ref{def:st} can be viewed
  as a translation from $\Lbd$ into $\IFOM$.
  This allows us to define the theorems of our intuitionistic monotone modal
  logic as the set $\{ \phi \in \Lbd \mid \IFOM \models \st_x(\phi) \}$.
  We extend this to $\Lbd$-consecutions, i.e.~expressions of the form $\Gamma \vdash \phi$
  such that $\Gamma \cup \{ \phi \} \subseteq \Lbd$, as follows.

\begin{definition}\label{def:imf}
  The logic $\IMf$ consists of all $\Lbd$-consecutions $\Gamma \vdash \phi$ such that
  for every $\IFOM$-structure $\fomo{M}$, every world $w$ in $\fomo{M}$
  and every $d \in \I(\s, w)$,
  $\fomo{M}, w \models \st_x(\Gamma)[d]$ implies $\fomo{M}, w \models \st_x(\phi)[d]$.
  We write $\Gamma \vdash_{\IMf} \phi$ if a consecution $\Gamma \vdash \phi$
  is in $\IMf$.
\end{definition}

  Intuitively, this states that $\Gamma \vdash_{\IMf} \phi$
  if $\st_x(\Gamma) \models \st_x(\phi)$ in all
  pointed $\IFOM$-structures, where the point indicates the interpretation of $x$.
  We can use the first-order structures for $\IFOM$ as models
  for $\IMf$. This is in analogy with IL-models for the intuitionistic
  modal logic $\IK$~\cite[Section~5.2]{Sim94}.

\begin{definition}\label{def:IM-structure}
  Let
  $\fomo{M} = (W, \leq, \I)$
  be an $\IFOM$-structure.
  Formulas from $\Lbd$ can be interpreted in pairs of the form $(w, x)$,
  where $w \in W$ and $x \in \I(\s, w)$, as follows:
  \begin{align*}
    \fomo{M}, w, x \models p_i &\iff x \in \I(\P_i, w) \\
    \fomo{M}, w, x \models \bot &\phantom{\iff}\text{never} \\
    \fomo{M}, w, x \models \phi \wedge \psi
      &\iff \fomo{M}, w, x \models \phi \text{ and } \fomo{M}, w, x \models \psi \\
    \fomo{M}, w, x \models \phi \vee \psi
      &\iff \fomo{M}, w, x \models \phi \text{ or } \fomo{M}, w, x \models \psi \\
    \fomo{M}, w, x \models \phi \to \psi
      &\iff \text{for all } w' \geq w, \;\; \fomo{M}, w', x \models \phi
            \text{ implies } \fomo{M}, w', x \models \psi \\
    \fomo{M}, w, x \models \Box\phi
      &\iff \text{there exists } a \in \I(\n, w) \text{ such that } (x, a) \in \I(\N, w)
            \text{ and such that} \\
      &\phantom{\iff} \text{for all } w' \geq w
           \text{ and all } x' \in \I(\s, w'), \; (a, x') \in \I(\E, w')
           \text{ implies } \fomo{M}, w', x' \models \phi \\
    \fomo{M}, w, x \models \Diamond\phi
      &\iff \text{for all } w' \geq w \text{ and all } a' \in \I(\n, w'),
            \text{ if } (x, a') \in \I(\N, w') \\
      &\phantom{\iff} 
            \text{ then there exists } y' \in \I(\s, w')
            \text{ s.t. } (a', y') \in \I(\E, w')
            \text{ and } \fomo{M}, w', y' \models \phi
  \end{align*}
  For $\Gamma \subseteq \Lbd$, we write $\fomo{M}, w, x \Vdash \Gamma$
  if $\fomo{M}, w, x \Vdash \psi$ for all $\psi \in \Gamma$.
  Furthermore, we write $\Gamma \models \phi$ if $\fomo{M}, w, x \Vdash \Gamma$
  implies $\fomo{M}, w, x \Vdash \phi$, for every $\IFOM$-structure
  $\fomo{M} = (W, \leq, \I)$ and
  every $w \in W$ and $x \in \I(\s, w)$.
\end{definition}

  Then by construction:

\begin{theorem}\label{thm:sc-triv}
  For all $\Gamma \cup \{ \phi \} \subseteq \Lbd$,
  we have $\Gamma \vdash_{\IMf} \phi$ if and only if $\Gamma \models \phi$.
\end{theorem}

%================================================================================
\subsection{Calculi for intuitionistic monotone modal logics}

  We define various intuitionistic monotone modal logics
  via generalised Hilbert calculi
  (see e.g.~\cite[Section~2.4]{TroSch00} or~\cite[Section~4.1]{Shi23}).
  This allows us to define a consequence relation (i.e.~a set consecutions
  $\Gamma \vdash \phi$)
  using axioms and rules, while at the same time being careful about
  the shape of $\Gamma$ allowed in the rules, avoiding confusions such as
  those pointed out in~\cite{GorShi20}.
  (For example, the premises of the modus ponens rule allow for any $\Gamma$,
  while those of the monotonicity rules require $\Gamma = \emptyset$.)

\begin{definition}\label{def:ghc}
  Let $\operatorname{Ax} \subseteq \Lbd$ be a set of axioms.
  Let $\Ax$ be the set of substitution instances of axioms in
  $\operatorname{Ax}$ together with all substitution instances of an
  axiomatisation of intuitionistic logic.
  Define $\IMC(\operatorname{Ax})$
  as the collection of consecutions derivable from the following rules:
  the \emph{element rule} and \emph{axiom rule}
  $$
    \dfrac{}{\Gamma \vdash \phi} \; (\REl),
    \qquad
    \dfrac{}{\Gamma \vdash \psi} \; (\RAx),
  $$
  where $\phi \in \Gamma$ and $\psi \in \Ax$,
  and \emph{modus ponens} and the \emph{monotonicity rules}
  $$
    \dfrac{\Gamma \vdash \phi \quad \Gamma \vdash \phi \to \psi}
          {\Gamma \vdash \psi} \; (\RMP),
    \qquad
    \dfrac{\emptyset \vdash \phi \to \psi}
          {\Gamma \vdash \Box\phi \to \Box\psi} \; (\RMonB)
    \qquad\text{and}\qquad
    \dfrac{\emptyset \vdash \phi \to \psi}
          {\Gamma \vdash \Diamond\phi \to \Diamond\psi} \; (\RMonD).
  $$
  We write $\Gamma \vdash_{\IMC(\operatorname{Ax})} \phi$
  if $\Gamma \vdash \phi$ can be derived in $\IMC(\operatorname{Ax})$.
\end{definition}

  It can be shown that any logic of the form $\IMC(\operatorname{Ax})$
  satisfies the deduction theorem.

\begin{theorem}
  Let $\Gamma \cup \{ \phi, \psi \} \subseteq \Lbd$.
  Then $\Gamma, \phi \vdash_{\IMC(\operatorname{Ax})} \psi$
  if and only if $\Gamma \vdash_{\IMC(\operatorname{Ax})} \phi \to \psi$.
\end{theorem}

  Taking $\Ax = \emptyset$ yields a monotone modal extension of intuitionistic
  logic where the modalities are completely unrelated. This can be viewed as
  a bimodal version of $\iM$, and as a monotone counterpart of
  the intuitionistic normal modal logics
  $\log{IK_0(\Box,\Diamond)}$~\cite[Section~4]{Sot84} and
  $\log{IntK_{\Box\!\Diamond}}$~\cite[Section~2]{WolZak99}.
  Using $\Ax = \{ (\Box p \wedge \Diamond\neg p) \to \bot \}$,
  we can define constructive monotone modal logic
  $\WM$~\cite{DalGreOli20,Dal22},
  the monotone counterpart of Wijesekera's constructive modal logic
  $\log{WK}$~\cite{Wij90,WijNer05}.

\begin{definition}\label{def:WM}
  $\WM := \IMC( \{ (\Box p \wedge \Diamond\neg p) \to \bot \})$.
\end{definition}

  Lastly, we define the calculus $\IMfC$.
  We will prove in Theorem~\ref{thm:IM-IMC} that this axiomatises
  the logic $\IMf$ from Definition~\ref{def:imf},
  thus axiomatising the intuitionistic monotone modal logic arising
  from the translation into the intuitionistic first-order logic $\IFOM$.

\begin{definition}
  The logic $\IMfC$ is defined as $\IMC(\operatorname{Ax})$, where $\operatorname{Ax}$
  consists of
  \begin{equation*}
    \text{($\ANega$)} \; (\Box p \wedge \Diamond\neg p) \to \bot
    \quad\text{and}\quad
    \text{($\AInt$)} \; (\Box \top \to \Diamond p) \to \Diamond p.
  \end{equation*}
\end{definition}

%%%%%%%%%%%%%%%%%%%%%%%%%%%%%%%%%%%%%%%%%%%%%%%%%%%%%%%%%%%%%%%%%%%%%%%%%%%%%%%%%
\section{Intuitionistic neighbourhood semantics}\label{sec:semantics}

  The interpretation of $\Box$ from Definition~\ref{def:IM-structure} relies
  on what worlds $v$ are in some neighbourhood $a$ (that is, on
  $a \E v$), quantified over intuitionistic successors of $w$.
  In particular, this elementhood can change when moving along the intuitionistic
  successors of $w$.
  In other words, we may have $w \leq w'$ and $v \in \I(\E, w')$ while $v \notin \I(\E, w)$.
  This suggests that when moving from a classical to an intuitionistic setting,
  we need to adapt our notion of a neighbourhood. Rather than a subset of 
  worlds, a neighbourhood can change when moving
  along the intuitionistic accessibility relation.
  This gives rise to the notion of an \emph{intuitionistic neighbourhood}.
  An intuitionistic neighbourhood model then becomes an intuitionistic Kripke
  frame together with a set of intuitionistic neighbourhoods and a valuation.
  
  Neighbourhoods of a world can be viewed as sets of evidence.
  Intuitionistic neighbourhoods then expresses the idea that evidence can
  change when moving along the intuitionistic accessibility relation.
  Necessity of $\phi$, i.e.~truth of $\Box\phi$, at a world $w$ then follows
  from finding evidence for $\phi$ at all successors of $w$. 
  Similarly, possibility of $\phi$ has to be shown at all successors.

  In this section we prove that such intuitionistic neighbourhood models
  provide another sound and complete semantics for $\IMf$
  by constructing truth-preserving translations between intuitionistic neighbourhood
  models and $\IFOM$-structures.
  While translating $\IFOM$-structures to intuitionistic neighbourhood models is
  straightforward (Section~\ref{subsec:inf}),
  the converse direction proceeds in three steps:
  given an intuitionistic neighbourhood model, we first turn it into
  a \emph{coherent} model (Section~\ref{subsec:coherent}),
  then into a \emph{Cartesian} one (Section~\ref{subsec:unravelling}),
  and we show that the coherent and Cartesian models are those
  isomorphic to intuitionistic neighbourhood models arising from $\IFOM$-structures(Section~\ref{subsec:cartesian}).
  
  Finally, in Section~\ref{subsec:IM-complete} we use a canonical model construction
  to prove that $\IMfC$ is complete with respect to the class of intuitionistic
  neighbourhood models, which entials that the generalised Hilbert calculus
  defining $\IMfC$ axiomatises $\IMf$.

%================================================================================
\subsection{Intuitionistic neighbourhood frames}\label{subsec:inf}

  If $(W, \leq)$ is an intuitionistic Kripke frame and $a \subseteq W$,
  then we define ${\uparrow}a = \{ v \in W \mid w \leq v \text{ for some } w \in a \}$.
  An \emph{upset} is a subset $a \subseteq W$ such that $a = {\uparrow}a$,
  and we denote the set of upsets of $(W, \leq)$ by $\up(W, \leq)$.
  In its simplest form, we can define an intuitionistic neighbourhood as follows.

\begin{definition}\label{def:in}
  Let $(W, \leq)$ be an intuitionistic Kripke frame.
  An \emph{intuitionistic neighbourhood} is a partial function
  $a : W \rightharpoonup \fun{P}W$ such that its domain
  $\dom(a) := \{ w \in W \mid a(w) \text{ is defined} \}$ is
  an upset of $(W, \leq)$.
\end{definition}

  Intuitionistic neighbourhood frames and models then arise from equipping
  an intuitionistic Kripke frame or model with a set of intuitionistic neighbourhoods.

\begin{definition}\label{def:inm}
  An \emph{intuitionistic neighbourhood frame} is a tuple
  $(W, \leq, N)$ consisting of an intuitionistic Kripke frame $(W, \leq)$
  and a collection $N$ of intuitionistic neighbourhoods.
  We say that $a \in N$ is a neighbourhood of $w \in W$ if
  $w \in \dom(a)$, and we write $N_w$ for the collection of neighbourhoods
  of $w$.
  
  An \emph{intuitionistic neighbourhood model} is a tuple
  $\mo{M} = (W, \leq, N, V)$ where $(W, \leq, N)$ is an intuitionistic
  neighbourhood frame and $V : \Prop \to \up(W, \leq)$ is a valuation
  of the proposition letters.
  The interpretation of $\Lbd$-formulas in a world $w$ of $\mo{M}$
  is defined recursively via:
  \begin{align*}
    \mo{M}, w \Vdash \bot &\phantom{\iff} \text{never} \\
    \mo{M}, w \Vdash p &\iff w \in V(p) \\
    \mo{M}, w \Vdash \phi \wedge \psi
      &\iff \mo{M}, w \Vdash \phi \text{ and } \mo{M}, w \Vdash \psi \\
    \mo{M}, w \Vdash \phi \vee \psi
      &\iff \mo{M}, w \Vdash \phi \text{ or } \mo{M}, w \Vdash \psi \\
    \mo{M}, w \Vdash \phi \to \psi
      &\iff \text{for all } v \geq w, \;
            \mo{M}, v \Vdash \phi \text{ implies } \mo{M}, v \Vdash \psi \\
    \mo{M}, w \Vdash \Box\phi
      &\iff \text{there exists } a \in N_w
            \text{ such that for all } w' \geq w, \;\;
            v \in a(w') \text{ implies } \mo{M}, v \Vdash \phi \\
    \mo{M}, w \Vdash \Diamond\phi
      &\iff \text{for all } w' \geq w
            \text{ and all } a \in N_{w'}
            \text{ there exists } v \in a(w')
            \text{ such that } \mo{M}, v \Vdash \phi
  \end{align*}
  If $\Gamma \subseteq \Lbd$ is a set of formulas, then we write
  $\mo{M}, w \Vdash \Gamma$ if $\mo{M}, w \Vdash \psi$ for all $\psi \in \Gamma$.
  We denote the \emph{truth set} of $\phi$ in $\mo{M}$ by
  $\llb \phi \rrb^{\mo{M}} := \{ w \in W \mid \mo{M}, w \Vdash \phi \}$.
  A model $\mo{M}$ \emph{validates} a consecution $\Gamma \vdash \phi$
  if any world that satisfies all formulas in $\Gamma$ also satisfies $\phi$,
  and a frame $\mo{F} = (W, \leq, N)$ validates $\Gamma \vdash \phi$
  if every model of the form $\mo{M} = (W, \leq, N, V)$ validates it.
  Lastly, we say that $\phi$ is a \emph{semantic consequence} of $\Gamma$
  over the class of intuitionistic neighbourhood models if every intuitionistic
  neighbourhood model validates $\Gamma \vdash \phi$, and we denote this
  by $\Gamma \Vdash^{\inm} \phi$.
\end{definition}

  The next lemma extends the usual persistence property for intuitionistic
  logic (see e.g.~\cite[Proposition~2.1]{ChaZak97}) to the modal setting.
  It can be proven by a routine induction on the structure of $\phi$.
  
\begin{lemma}
  Let $\mo{M} = (W, \leq, N, V)$ be an intuitionistic neighbourhood model.
  For all formulas $\phi \in \Lbd$ and all worlds $w, v \in W$,
  if $\mo{M}, w \Vdash \phi$ and $w \leq v$ then $\mo{M}, v \Vdash \phi$.
\end{lemma}

  While we do not need bounded morphisms,
  for future reference we define isomorphisms between models.

\begin{definition}
  Let $\mo{M} = (W, \leq, N, V)$ and $\mo{M}' = (W', \leq', N', V')$ be two
  intuitionistic Kripke models.
  An \emph{isomorphism} from $\mo{M}$ to $\mo{M}'$ consists of a pair
  $(\alpha, \nu)$ of bijections $\alpha : W \to W'$ and $\nu : N \to N'$
  such that for all $w, v \in W$, $a \in N$ and $u \in \dom(a)$:
  \begin{enumerate}\itemsep=0em
  \renewcommand{\labelenumi}{(\theenumi) }
    \renewcommand{\theenumi}{I$_{\arabic{enumi}}$}
    \item $w \leq v$ if and only if $\alpha(w) \leq' \alpha(v)$
    \item $w \in \dom(a)$ if and only if $\alpha(w) \in \dom(\nu(a))$
    \item $w \in a(u)$ if and only if $\alpha(w) \in (\nu(a))(\alpha(u))$
    \item $w \in V(p_i)$ if and only if $\alpha(w) \in V'(p_i)$
  \end{enumerate}
\end{definition}

\begin{lemma}\label{lem:iso}
  Let $\mo{M} = (W, \leq, N, V)$ and $\mo{M}' = (W', \leq', N', V')$ be two
  intuitionistic Kripke models
  and $(\alpha, \nu) : \mo{M} \to \mo{M}'$ an isomorphism.
  Then for all $w \in W$ and $\phi \in \Lbd$, $\mo{M}, w \Vdash \phi$
  if and only if $\mo{M}', \alpha(w) \Vdash \phi$.
\end{lemma}

  Every $\IFOM$-structure gives rise to an intuitionistic neighbourhood model.

\begin{definition}\label{def:IFOM-to-inm}
  Let $\fomo{M} = (W, \leq, \I)$ be
  an $\IFOM$-structure.
  Then we define the intuitionistic Kripke frame $(W^{\bul}, \leqq)$
  by
  \begin{align*}
    &W^{\bullet} = \{ \langle w, x \rangle \mid w \in W, x \in \I(\s,w) \},
    &\langle w, x \rangle \leqq \langle w', x' \rangle \iff w \leq w' \text{ and } x = x'
  \end{align*}
  For each $a \in \bigcup \{ \I(\n, w) \mid w \in W \}$, define an intuitionistic
  neighbourhood $a^{\bullet} : W^{\bullet} \rightharpoonup \fun{P}(W^{\bullet})$
  by setting
  $$
    a^{\bul}(\langle w, x \rangle)
      = \{ \langle w, y \rangle \in W^{\bullet} \mid (a, y) \in \I(\E, w) \}
  $$
  if $a \in \I(\n, w)$ and $(x, a) \in \I(\N, w)$,
  and leaving $a^{\bul}$ undefined otherwise.
  Let $N^{\bullet} = \{ a^{\bullet} \mid a \in \I(\n, w) \text{ for some } w \in W \}$.
  Then $(W^{\bullet}, \leqq, N^{\bullet})$ is an intuitionistic neighbourhood frame.
  Finally, define the valuation $V^{\bullet}$ by
  $$
    V^{\bullet}(p_i) = \{ \langle w, x \rangle \in W^{\bullet} \mid x \in \I(\P_i, w) \}.
  $$
  and let $\fomo{M}^{\bul} = (W^{\bul}, \leqq, N^{\bul}, V^{\bul})$ be the
  resulting model.
\end{definition}

\begin{lemma}
  If $\fomo{M}$ is a $\IFOM$-structure then $\fomo{M}^{\bul}$ is an
  intuitionistic neighbourhood model.
\end{lemma}
\begin{proof}
  Since $\leq$ is a partial order, so is $\leqq$.
  The domain of each $a^{\bul}$ is an upset of $(W^{\bul}, \leqq)$ because
  $w \leq w'$ implies $\I(\n, w) \subseteq \I(\n, w')$.
  To see that the valuation sends propositions to upsets,
  suppose $\langle w, x \rangle \in V^{\bul}(p_i)$ and
  $\langle w, x \rangle \leqq \langle w', x' \rangle$.
  Then $w \leq w'$ and $x = x'$ and it follows from the fact that
  $\I(\P_i, w) \subseteq \I(\P_i, w')$ that $\langle w', x' \rangle \in V^{\bul}(p_i)$.
\end{proof}

\begin{proposition}\label{prop:truth-str-to-inm}
  For every $\IFOM$-structures
  $\fomo{M} = (W, \leq, \I)$
  and all $w \in W$, $x \in \I(\s, w)$ and $\phi \in \Lbd$,
  \begin{equation*}
    \fomo{M}, w, x \models \phi \iff \fomo{M}^{\bul}, \langle w, x \rangle \Vdash \phi.
  \end{equation*}
\end{proposition}
\begin{proof}
  We use induction on the structure of $\phi$.
  If $\phi = \bot$ then the proposition follows from the fact that $\bot$
  is never satisfied. If $\phi = p_i \in \Prop$, then we have
  \begin{equation*}
    \fomo{M}, w, x \models p_i
      \iff x \in \I(\P_i, w)
      \iff \langle w, x \rangle \in V^{\bul}(p_i)
      \iff \fomo{M}^{\bul}, \langle w, x \rangle \Vdash p_i.
  \end{equation*}
  The inductive steps for $\phi = (\psi \wedge \chi)$ and $\phi = (\psi \vee \chi)$
  are straightforward. Suppose $\phi = \psi \to \chi$. Then
  \begin{align*}
    \fomo{M}, w, x \Vdash \psi \to \chi
      &\iff \text{for all } w' \geq w, \;\;
            \fomo{M}, w', x \Vdash \psi \text{ implies } \fomo{M}, w', x \Vdash \chi \\
      &\iff \text{for all } \langle w', x' \rangle \geqq \langle w, x \rangle, \;\;
            \fomo{M}^{\bul}, \langle w', x' \rangle \Vdash \psi
            \text{ implies } \fomo{M}^{\bul}, \langle w', x' \rangle \Vdash \chi \\
      &\iff \fomo{M}^{\bul}, \langle w, x \rangle \Vdash \psi \to \chi.
  \end{align*}
  Here the second ``iff'' uses the definition of $\geqq$ and the induction
  hypothesis.
  
  Next suppose $\phi = \Box\psi$.
  If $\fomo{M}, w, x \models \Box\psi$ then there exists an
  $a \in \I(\n, w)$ such that $(x, a) \in \I(\N, w)$ and
  for all $w' \geq w$ and $x' \in \I(\s, w')$ we have that $(a, x') \in \I(\E, w')$
  implies $\fomo{M}, w', x' \Vdash \phi$.
  This $a$ gives rise to a neighbourhood $a^{\bul}$ of
  $\fomo{M}^{\bul}$ witnessing the truth of
  $\fomo{M}^{\bul}, \langle w, x \rangle \Vdash \Box\psi$.
  Conversely, the existence of a suitable neighbourhood $a^{\bul}$
  witnessing that $\fomo{M}^{\bul}, \langle w, x \rangle \Vdash \Box\psi$
  yields $a \in \I(\n, w)$ satisfying the desired properties.
  The case $\phi = \Diamond\psi$ also follows from the fact that neighbourhoods
  of $\fomo{M}^{\bul}$ correspond precisely to elements of sort $\n$ in $\fomo{M}$.
\end{proof}

\begin{theorem}[Soundness]\label{thm:sound}
  If $\Gamma \vdash_{\IMfC} \phi$ then $\Gamma \Vdash^{\inm} \phi$
  and $\Gamma \models \phi$.
\end{theorem}
\begin{proof}
  It suffices to prove that $\Gamma \vdash_{\IMfC} \phi$ implies $\Gamma \Vdash^{\inm} \phi$.
  Proposition~\ref{prop:truth-str-to-inm} then implies $\Gamma \models \phi$.
  We use induction on a proof $\delta$ of $\Gamma \vdash_{\IMfC} \phi$.
  This gives rise to the following five cases.
  
  \medskip\noindent
  \textit{The last rule in $\delta$ is $(\REl)$.}
    In this case we have $\phi \in \Gamma$. Clearly, this implies that any world
    in any intuitionistic neighbourhood model that satisfies $\Gamma$ also
    satisfies $\phi$, so $\Gamma \Vdash^{\inm} \phi$.
  
  \medskip\noindent
  \textit{The last rule in $\delta$ is $(\RAx)$.}
    Since $\Gamma$ is arbitrary, we need to prove that the substitution
    instances of all axioms are valid on all intuitionistic neighbourhood models.
    For the axioms of intuitionistic logic, this is well known.
    
    To prove validity of $(\ANega)$, let $\mo{M} = (W, \leq, N, V)$ be any
    intuitionistic neighbourhood model
    and suppose $\mo{M}, w \Vdash \Box\phi \wedge \Diamond\neg\phi$.
    Then $\mo{M}, w \Vdash \Box\phi$, so there exists an intuitionistic
    neighbourhood $a \in N$ such that $w \in \dom(a)$ and
    $a(w') \subseteq \llb \phi \rrb^{\mo{M}}$ for all $w' \geq w$.
    In particular, $a(w) \subseteq \llb \phi \rrb^{\mo{M}}$.
    But we also have $\mo{M}, w \Vdash \Diamond\neg\phi$, which
    implies that we can find some $v \in a(w)$ such that $\mo{M}, v \Vdash \neg\phi$.
    So $\mo{M}, v \Vdash \phi$ and $\mo{M}, v \not\Vdash \phi$, a contradiction.
    We conclude that no world satisfies $\Box\phi \wedge \Diamond\neg\phi$,
    hence $(\Box\phi \wedge \Diamond\neg\phi) \to \bot$ is valid.
    Since $\mo{M}$ is arbitrary, this proves
    $\Gamma \Vdash^{\inm} (\Box\phi \wedge \Diamond\neg\phi) \to \bot$.
    
    We argue similarly for $(\AInt)$.
    Suppose $\mo{M}, w \Vdash \Box\top \to \Diamond\psi$.
    To see that $\mo{M}, w \Vdash \Diamond\psi$, let $w' \geq w$
    and suppose $w' \in \dom(a)$ for some $a \in N$. Then $a$ witnesses
    $\mo{M}, w' \Vdash \Box\top$. Since $w \leq w'$, the definition of the
    interpretation of implication gives $\mo{M}, w' \Vdash \Diamond\psi$.
    Again by definition, this means that there exists a $v' \in a(w')$ such
    that $\mo{M}, v' \Vdash \psi$. Therefore $\mo{M}, w \Vdash \Diamond\psi$.
    Since $\mo{M}$ and $w$ are arbitrary, every world of every
    intuitionistic neighbourhood model satisfies
    $(\Box\top \to \Diamond\psi) \to \Diamond\psi$,
    and hence $\Gamma \Vdash^{\inm} (\Box\top \to \Diamond\psi) \to \Diamond\psi$.

  \medskip\noindent
  \textit{The last rule in $\delta$ is $(\RMP)$.}
    By induction $\Gamma \Vdash^{\inm} \psi \to \phi$
    and $\Gamma \Vdash^{\inm} \psi$. So any world $w$ in any model $\mo{M}$
    that satisfies $\Gamma$, also satisfies $\psi \to \phi$ and $\phi$.
    By definition we then get $\mo{M}, w \Vdash \phi$.
    Therefore $\Gamma \Vdash^{\inm} \phi$.
  
  \medskip\noindent
  \textit{The last rule in $\delta$ is $(\RMonB)$ or $(\RMonD)$.}
    In this case $\phi$ must be of the form $\Star\psi \to \Star\chi$,
    where $\Star \in \{ \Box, \Diamond \}$,
    and by induction $\emptyset \Vdash^{\inm} \psi \to \chi$.
    Then we must have $\llb \psi \rrb^{\mo{M}} \subseteq \llb \chi \rrb^{\mo{M}}$
    for any model $\mo{M}$. Using this, it follows immediately from the
    definition of the interpretations of $\Box$ and $\Diamond$ that
    $\mo{M}, w \Vdash \Star\psi$ implies $\mo{M}, w \Vdash \Star\chi$
    for any intuitionistic neighbourhood model $\mo{M}$ and world $w$.
\end{proof}

%================================================================================
\subsection{Coherent intuitionistic neighbourhoods}\label{subsec:coherent}

  If we think of an intuitionistic Kripke frame as a collection of information
  states which gain increasingly more information, and of intuitionistic
  neighbourhoods as evidence, then one would expect that evidence does not
  get lost when progressing along the intuitionistic accessibility relation.
  If anything, evidence should get more specific. Thus, a condition
  we may like to put on our intuitionistic neighbourhoods is
  \begin{enumerate}
  \renewcommand{\labelenumi}{(\theenumi) }
    \renewcommand{\theenumi}{N$_{\arabic{enumi}}$}
    \item \label{it:in-2}
          if $w \leq w'$ and $v \in a(w)$ then there exists a $v' \in W$
          such that $v' \in a(w')$ such that $v \leq v'$.
  \end{enumerate}
  Reasoning in an ideal world, we could also hope that more specific evidence
  will eventually be found. That is,
  \begin{enumerate}
  \renewcommand{\labelenumi}{(\theenumi) }
    \renewcommand{\theenumi}{N$_{\arabic{enumi}}$}
    \setcounter{enumi}{1}
    \item \label{it:in-3}
          if $v \in a(w)$ and $v \leq v'$ then there exists a $w' \in W$
          such that $w \leq w'$ and $v' \in a(w')$.
  \end{enumerate}

\begin{definition}\label{def:coherent}
  Let $(W, \leq)$ be an intuitionistic Kripke frame. An intuitionistic
  neighbourhood $a : W \rightharpoonup \fun{P}(W)$ is called \emph{coherent}
  if it satisfies~\eqref{it:in-2} and~\eqref{it:in-3}.
  An intuitionistic neighbourhood frame or model is called coherent if all its
  intuitionistic neighbourhoods are coherent.
  We write $\Gamma \Vdash^{\coh} \phi$ if the consecution $\Gamma \vdash \phi$
  is valid in all coherent intuitionistic neighbourhood models.
\end{definition}

\begin{figure}[h!]
  \centering
    \begin{tikzpicture}[scale=.75]
      %% nodes
        \node (w)  at (0,0)   {$w$};
        \node (v)  at (2,0)   {$v$};
        \node (wp) at (0,1.5) {$w'$};
        \node (vp) at (2,1.5) {$v'$};
        \node (N2) at (1,-.75) {\eqref{it:in-2}};
      %% edges
        \draw[-latex] (w) to (wp);
        \draw[-Circle] (w) to node[above]{\footnotesize{$a$}}(v);
        \draw[dashed, -latex] (v) to (vp);
        \draw[dashed, -Circle] (wp) to node[above]{\footnotesize{$a$}}(vp);
      %% nodes
        \node (w)  at (4,0)   {$w$};
        \node (v)  at (6,0)   {$v$};
        \node (wp) at (4,1.5) {$w'$};
        \node (vp) at (6,1.5) {$v'$};
        \node (N2) at (5,-.75) {\eqref{it:in-3}};
      %% edges
        \draw[dashed, -latex] (w) to (wp);
        \draw[-Circle] (w) to node[above]{\footnotesize{$a$}}(v);
        \draw[-latex] (v) to (vp);
        \draw[dashed, -Circle] (wp) to node[above]{\footnotesize{$a$}}(vp);
    \end{tikzpicture}
  \caption{A depiction of the coherence conditions.
           Here the arrows indicate the intuitionistic
           accessibility relation and%
           $\smash{\protect\tikz[baseline=-.8mm]
            {\protect\node (w) at (0,0) {$w$};
             \protect\node (v) at (1,0) {$v$};
             \protect\draw[-Circle] (w) to node[above,yshift=-1pt]{\footnotesize{$a$}}(v);}}$%
           indicates that $v \in a(w)$.}
  \label{fig:in}
\end{figure}
  
  Both coherence conditions are satisfied in intuitionistic Kripke models arising
  from $\IFOM$-structures via Definition~\ref{def:IFOM-to-inm},
  and in fact that is our motivation to investigate them further.
  Moreover, it turns out that we can modify any intuitionistic neighbourhood
  model into one satisfying~\eqref{it:in-2} and~\eqref{it:in-3} without
  affecting the truth of formulas.
  Intuitively, this is achieved by introducing copies of neighbourhoods
  relative to the worlds in their domain in such a way that they satisfy~\eqref{it:in-2}.
  Second, we unravel maximal worlds,
  i.e.~the worlds that have no other worlds strictly above them,
  in order to satisfy~\eqref{it:in-3}.
  Interestingly, the canonical model construction we give in
  Section~\ref{subsec:IM-complete} also requires an unravelling of
  maximal points.

\begin{definition}
  Let $\mo{M} = (W, \leq, N, V)$ be an intuitionistic neighbourhood model.
  Let
  $$
    W^{\coh} := \{ (w, 0) \in W \times \mb{N} \mid w \in W, w \text{ not maximal} \}
      \cup \{ (w, n) \in W \times \mb{N} \mid w \in W, w \text{ is maximal} \}.
  $$
  Define $\leq^{\coh}$ on $W^{\coh}$ by $(w, n) \leq^{\coh} (v, m)$ if
  $w \leq v$ and $n \leq m$.
  
  For any $w \in W$ and $a \in N$ such that $w \in \dom(a)$,
  define $a(w)^{\coh} := \{ (u, k) \in W^{\coh} \mid u \in a(w) \}$.
  Then, for $a \in N$ and $(v, m) \in W^{\coh}$ we define the intuitionistic
  neighbourhood $a_{v, m}$ to be the partial function with domain
  $\{ (w, n) \in W^{\coh} \mid (v, m) \leq (w, n) \}$ given by
  \begin{equation*}
    a_{v,m}(w, n)
      = \begin{cases}
          a(w)^{\coh} & \text{if } (v, m) = (w, n) \\
          {\uparrow}(\textstyle\bigcup \{ a(u)^{\coh} \mid v \leq u \}) & \text{if } (v, m) <^{\coh} (w, n)
        \end{cases}
  \end{equation*}
  Let $N^{\coh} = \{ a_{v,m} \mid a \in N, (v, m) \in W^{\coh} \}$.
  Finally, define a valuation by $V^{\coh}(p_i) = \{ (w, n) \in W^{\coh} \mid w \in V(p_i) \}$ and let $\mo{M}^{\coh} := (W^{\coh}, \leq^{\coh}, N^{\coh}, V^{\coh})$.
\end{definition}

\begin{lemma}
  Let $\mo{M}$ be an intuitionistic neighbourhood model.
  Then $\mo{M}^{\coh}$ is a coherent intuitionistic neighbourhood model.
\end{lemma}
\begin{proof}
  We need to verify that $a_{v,m}$ is a coherent intuitionistic neighbourhood,
  for every $a \in N$ and $(v, m) \in W^{\coh}$.
  So let $a_{v,m} \in N^{\coh}$.
  Then its domain is an upset by definition.
  Item~\eqref{it:in-2} follows from the fact that if
  $(w, n) \in \dom(a_{v,m})$ and $(w, n) \leq^{\coh} (w', n')$, then
  $a_{v,m}(w, n) \subseteq a_{v,m}(w', n')$ by construction.
  For~\eqref{it:in-3}, suppose $(u,k) \in a_{v,m}(w,n)$ and $(u, k) \leq^{\coh} (u', k')$.
  By construction $(w, n)$ has a successor $(w', n')$ and
  $a_{v,m}(w', n')$ contains $(u, k)$ and is upwards closed, hence contains $(u', k')$.
\end{proof}

\begin{proposition}\label{prop:coherent}
  Let $\mo{M} = (W, \leq, N, V)$ be an intuitionistic neighbourhood model.
  Then for all $\phi \in \Lbd$ and all $(w, n) \in W^{\coh}$,
  \begin{equation*}
    \mo{M}^{\coh}, (w, n) \Vdash \phi \iff \mo{M}, w \Vdash \phi.
  \end{equation*}
\end{proposition}
\begin{proof}
  We prove the proposition by induction on the structure of $\phi$.
  The case $\phi = \bot$ holds because $\bot$ is never satisfied
  and if $\phi = p_i \in \Prop$ then the proposition follows from the
  definition of $V^{\coh}$. The inductive steps for conjunction, disjunction
  and implication are routine.
  
  \medskip\noindent
  \textit{Induction step for $\phi = \Box\psi$.}
    Suppose $\mo{M}^{\coh}, (w, n) \Vdash \Box\psi$ and let $a_{v,m}$ be
    the neighbourhood witnessing this. Then $a \in N$ and $v \in \dom(a)$.
    Since $(w, n) \in \dom(a_{v,m})$ we have $v \leq w$, so $w \in \dom(a)$.
    Let $w \leq w'$ and $u \in a(w')$. Then $(w, n) \leq^{\coh} (w', n)$
    and $(u, 0) \in a_{v,m}(w', n)$, hence $\mo{M}^{\coh}, (u, 0) \Vdash \psi$.
    By induction we get $\mo{M}, u \Vdash \psi$,
    and therefore the intuitionistic neighbourhood $a$ witnesses $\mo{M}, w \Vdash \Box\psi$.
    Conversely, if $\mo{M}, w \Vdash \Box\psi$ and $a$ witnesses this,
    then by construction $a_{w,n}$ is a neighbourhood which shows that
    $\mo{M}^{\coh}, (w, n) \Vdash \Box\psi$.
    
  \medskip\noindent
  \textit{Induction step for $\phi = \Diamond\psi$.}
    Suppose $\mo{M}, w \Vdash \Diamond\psi$.
    Let $(w', n') \in W^{\coh}$ and $a_{v,m} \in N^{\coh}$ be 
    such that $(w, n) \leq^{\coh} (w', n')$
    and $(w', n') \in \dom(a_{v,m})$. We need to find a world in
    $a_{v,m}(w', n')$ that satisfies $\psi$.
    By definition of $\leq^{\coh}$ and $a_{v,m}$ we get $w \leq w'$ and $v \leq w'$ and
    $w' \in \dom(a)$, so we can find some $u \in a(w')$ such that $\mo{M}, u \Vdash \psi$.
    By induction this means that $\mo{M}^{\coh}, (u, 0) \Vdash \psi$.
    Further, $(u, 0) \in a(w')^{\coh} \subseteq a_{v,m}(w', n')$, so we
    have found the desired world.
    
    For the converse we reason by contrapositive, so we assume
    $\mo{M}, w \not\Vdash \Diamond\psi$. Then there exists a
    $w' \geq w$ and a neighbourhood $a$ such that for all $u \in a(w')$ we
    have $\mo{M}, u \not\Vdash \psi$. The induction hypothesis then implies
    that for all $(u, k) \in a(w')^{\coh}$ we have
    $\mo{M}^{\coh}, (u, k) \not\Vdash \psi$.
    Furthermore, we have $(w', n) \in \dom(a_{w',n})$ and
    $a_{w',n}(w', n) = a(w')^{\coh}$, so using the fact that
    $(w, n) \leq^{\coh} (w', n)$ we find $\mo{M}^{\coh}, (w, n) \not\Vdash \Diamond\psi$.
\end{proof}

\begin{theorem}\label{thm:compl-coh}
  We have $\Gamma \Vdash^{\inm} \phi$ if and only if $\Gamma \Vdash^{\coh} \phi$.
\end{theorem}
\begin{proof}
  The direction from left to right follows from the fact that every
  coherent intuitionistic neighbourhood model is in particular an intuitionistic
  neighbourhood model. The direction from right to left follows from
  Proposition~\ref{prop:coherent}.
\end{proof}

\begin{remark}
  An interesting situation arises when we restrict $N$ in an intuitionistic
  neighbourhood frame $(W, \leq N)$ to having a single neighbourhood $a$.
  If $\dom(a) = W$, then the frame corresponds to a frame for Wijesekera's
  constructive logic $\log{WK}$~\cite{Wij90,WijNer05}, and the interpretation
  of the modal operators corresponds to that for $\log{WK}$.
  However, if we additionally assume that $a$ is coherent, then we obtain a frame or model
  for $\IK$~\cite{Sim94} (see also Appendix~\ref{app:IK}).
  So while the coherence condition does not entail any additional valid
  consecutions when working with intuitionistic neighbourhood semantics,
  it does when we restrict to a setting with a single neighbourhood,
  because we know that $\IK$ proves strictly more theorems than $\log{WK}$.
\end{remark}

%================================================================================
\subsection{Cartesian intuitionistic neighbourhood models}\label{subsec:cartesian}

\newcommand{\lsim}{\ensuremath{\mathbin{\raisebox{.9\depth}{\scalebox{1}[-1]{$\lesssim$}}}}}
\newcommand{\eqR}{\overline{R}}

\newcommand{\qel}{\ensuremath{\mathbin{\raisebox{.9\depth}{\scalebox{1}[-1]{$\leqslant$}}}}}
\newcommand{\qeg}{\ensuremath{\mathbin{\raisebox{.9\depth}{\scalebox{1}[-1]{$\geqslant$}}}}}
\newcommand{\ov}{\overline}

  We have seen that every first-order structure for
  $\IFOM$ gives rise to an intuitionistic neighbourhood model.
  In this section we characterise precisely the intuitionistic neighbourhood
  models arising from an $\IFOM$-structure (up to isomorphism).
  This gives a condition on the underlying frames.
  In analogy to the normal modal case (see~\cite[Section~3]{Fis81}
  and~\cite[Section~8.1.1]{Sim94}),
  we call such intuitionistic neighbourhood frames Cartesian.

\begin{definition}
  Let $\mo{M} = (W, \leq, N, V)$ be an intuitionistic neighbourhood model.
  \begin{enumerate}
    \item Define a binary relation $R$ on $W$ by letting $wRv$ if there exists a
          neighbourhood $a$ such that $v \in a(w)$.
          We write $R^{\sim}$ for the equivalence closure of $R$.
          Further, we write $\leq^{\sim}$ for the equivalence closure of $\leq$.
    \item $\mo{M}$ (or its underlying frame) is called
          \emph{$R^{\sim}$-Cartesian} if for all $w, v \in W$,
          $w \leq^{\sim} v R^{\sim} w$ implies $w = v$.
    \item $\mo{M}$ (or its underlying frame) is called
          \emph{N-Cartesian} or \emph{neighbourhood Cartesian}
          if $w R^{\sim} v$ and $w, v \in \dom(a)$ implies $a(w) = a(v)$,
          for all $w, v \in W$ and $a \in N$.
    \item We say that $\mo{M}$ is \emph{Cartesian} if it is both
          $R^{\sim}$-Cartesian and N-Cartesian.
  \end{enumerate}
\end{definition}

  We derive some basic properties of $R^{\sim}$ and of Cartesian models.

\begin{lemma}\label{lem:properties}
  Let $\mo{M} = (W, \leq, N, V)$ be a coherent and Cartesian
  intuitionistic neighbourhood model.
  \begin{enumerate}
    \item \label{it:properties-1}
          If $w R^{\sim} v \leq u$ then there exists an $s \in W$ such that
          $w \leq s R^{\sim} u$.
    \item \label{it:properties-2}
          If $\mo{M}$ is Cartesian, then $w R^{\sim}v \leq^{\sim} u$
          and $w R^{\sim}v' \leq^{\sim} u$ implies $v = v'$.
  \end{enumerate}
\end{lemma}
\begin{proof}
  The first item follows from repeated application of~\eqref{it:in-2} and~\eqref{it:in-3}.
  For the second item, suppose $w R^{\sim} v \leq^{\sim} u$
  and $w R^{\sim} v' \leq^{\sim} u$.
  Then we have $v \leq^{\sim} v' R^{\sim} v$, hence if $\mo{M}$ is
  Cartesian we find $v = v'$.
\end{proof}

\begin{figure}[h!]
  \centering
  \begin{tikzpicture}[scale=.75]
    %% figure 1
    %% nodes
      \node (w) at (0,0) {$w$};
      \node (v) at (0,1.5) {$v$};
      \node (ww) at (2,1.5) {$w$};
    %% arrows
      \draw[latex-latex, bend left=30] (w) to (v);
      \draw[Circle-Circle] (v) to (ww);
      \draw[dashed, double, bend right=29] (w) to (v);
    %% name
      \node at (1,-1) {$R^{\sim}$-Cartesian};
    %% figure 2
    %% edges
      \node (w) at (5,0) {$w$};
      \node (v) at (7,0) {$v$};
      \node (s) at (5,1.5) {$s$};
      \node (u) at (7,1.5) {$u$};
    %% arrows
      \draw[Circle-Circle] (w) to (v);
      \draw[-latex] (v) to (u);
      \draw[dashed, -latex] (w) to (s);
      \draw[dashed, Circle-Circle] (s) to (u);
    %% name
      \node at (6,-1) {Lemma~\ref{lem:properties}\eqref{it:properties-1}};
    %% figure 3
    %% nodes
      \node (w) at (10,.5) {$w$};
      \node (v) at (12,.75) {$v$};
      \node (vp) at (13,0) {$v'$};
      \node (u) at (12.5,2.125) {$u$};
    %% arrows
      \draw[Circle-Circle, bend left=5] (w) to (v);
      \draw[Circle-Circle, bend right=10] (w) to (vp);
      \draw[latex-latex, bend left=10] (v) to (u);
      \draw[latex-latex, bend right=10] (vp) to (u);
      \draw[dashed, double] (v) to (vp);
    %% name
      \node at (11.5, -1) {Lemma~\ref{lem:properties}\eqref{it:properties-2}};
  \end{tikzpicture}
  \caption{Depictions of the $R^{\sim}$-Cartesian property and Lemma~\ref{lem:properties}.
           In this diagram, and other diagrams in this section, we draw
           $\smash{\protect\tikz[baseline=-.8mm]
            {\protect\node (w) at (0,0) {$w$};
             \protect\node (v) at (1.25,0) {$v$};
             \protect\draw[Circle-Circle] (w) 
                to node[above,yshift=-1pt]{\footnotesize{$a$}}(v);}}$%
             if $wR^{\sim}v$, and
           $\smash{\protect\tikz[baseline=-.8mm]
            {\protect\node (w) at (0,0) {$w$};
             \protect\node (v) at (1.25,0) {$v$};
             \protect\draw[latex-latex] (w) to (v);}}$%
           if $w \leq^{\sim} v$.}
\end{figure}

  We already observed that intuitionistic neighbourhood models coming
  from $\IFOM$-structures are coherent. We can easily verify that they
  are also Cartesian.

\begin{lemma}\label{lem:Cartesian}
  Let $\fomo{M}$ be an $\IFOM$-structure.
  Then $\fomo{M}^{\bul}$ is a Cartesian intuitionistic neighbourhood model.
\end{lemma}
\begin{proof}
  Let $\langle w, x \rangle$ and $\langle v, y \rangle$ be elements
  of $\fomo{M}^{\bul}$.
  By definition, $\langle w, x \rangle \leqq \langle v, y \rangle$
  implies $x = y$. As a consequence,
  $\langle w, x \rangle \leqq^{\sim} \langle v, y \rangle$ also implies $x = y$.
  On the other hand, if $a \in \I(\n, w)$ then
  $\langle v, y \rangle \in a^{\bul}(\langle w, x \rangle)$ implies $w = v$.
  This entails that
  $\langle w, x \rangle R^{\sim} \langle v, y \rangle$ implies $w = v$.

  Combining the two observations above immediately gives that
  $\langle w, x \rangle \leqq^{\sim} \langle v, y \rangle R^{\sim} \langle w, x \rangle$
  implies $w = v$ and $x = y$, so that $\langle w, x \rangle = \langle v, y \rangle$
  and $\fomo{M}^{\bul}$ is $R^{\sim}$-Cartesian.
  Second, suppose $\langle w, x \rangle R^{\sim} \langle v, y \rangle$.
  Then $w = v$, so we assume $\langle w, x \rangle R^{\sim} \langle w, y \rangle$.
  If both $\langle w, x \rangle \in \dom(a^{\bul})$
  and $\langle w, y \rangle \in \dom(a^{\bul})$,
  then we have $(x, a) \in \I(\N, w)$ and $(y, a) \in \I(\N, w)$.
  But then it follows from the definition of $a^{\bul}$ that
  $a^{\bul}(\langle w, x \rangle) = a^{\bul}(\langle w, y \rangle)$,
  so $\fomo{M}^{\bul}$ is N-Cartesian.
\end{proof}

  The remainder of this section focusses on showing that every
  coherent and Cartesian intuitionistic neighbourhood model is isomorphic to an
  intuitionistic neighbourhood model of the form $\fomo{M}^{\bul}$,
  for some $\IFOM$-structure $\fomo{M}$. To this end, we fix a
  coherent and Cartesian intuitionistic neighbourhood model $\mo{M} = (W, \leq, N, V)$.
  Denote the equivalence class of $w \in W$ with respect to $R^{\sim}$
  by $\bar{w}$, and its equivalence class with respect to $\leq^{\sim}$
  by $\tilde{w}$.
  
\begin{definition}
  Let $N_{triv} = \{ a \in N \mid \dom(a) = \emptyset \}$ be the collection
  of intuitionistic neighbourhoods of $\mo{M}$ whose domain is empty.
  Define
  \begin{align*}
    \ov{W} &= W/{R^{\sim}} = \{ \bar{w} \mid w \in W \} &
    \I(\s, \bar{w}) &= \{ \tilde{v} \mid w R^{\sim} w' \leq^{\sim} v \text{ for some } w' \} \\
    && \I(\n, \bar{w}) &= \{ a \in N \mid w R^{\sim} x \text{ and } x \in \dom(a) \text{ for some } x \in W \} \cup N_{triv}
  \end{align*}
  Define a relation ${\qel} \subseteq \ov{W} \times \ov{W}$, and
  interpretations of the predicates by:
  \begin{align*}
    \bar{w} \qel \bar{w}' &\iff \text{there exists $v'$ such that } w \leq v' R^{\sim} w' \\
    (\tilde{x}, a) \in \I(\N, \bar{u})
      &\iff \text{for all $w$ we have: }
            u R^{\sim} w \leq^{\sim} x \text{ implies } w \in \dom(a) \\
    (a, \tilde{x}) \in \I(\E, \bar{u})
      &\iff \text{there exist } y, z
            \text{ such that } u R^{\sim} y
            \text{ and } z \in a(y) \text{ and } z \leq^{\sim} x \\
    \tilde{x} \in \I(\P_i, \bar{u})
      &\iff w R^{\sim} y \leq^{\sim} x \text{ implies } y \in V(p_i)
  \end{align*}
  Lastly, let $\mo{M}^{\circ} := (\ov{W}, \qel, \I)$.
\end{definition}

  We verify that this definition gives rise to a $\IFOM$-structure.

\begin{lemma}
  The tuple
  $\mo{M}^{\circ} = (\ov{W}, \qel, \I)$
  is an $\IFOM$-structure.
\end{lemma}
\begin{proof}
  \textit{$\qel$ is a partial order on $\ov{W}$.}
    The relation $\qel$ is reflexive because $\leq$ and $R^{\sim}$ are.
    For transitivity, suppose $\bar{w} \qel \bar{v} \qel \bar{u}$.
    Then there are $v', u'$ such that $w \leq v' R^{\sim} v \leq u' R^{\sim} u$.
    Now we can use~Lemma~\ref{lem:properties}\eqref{it:properties-1}
    to find $u''$ such that $v' \leq u' R^{\sim} u'$.
    This implies $w \leq u'' R^{\sim} u$, so that $\bar{w} \qel \bar{u}$.
    In a picture:
    
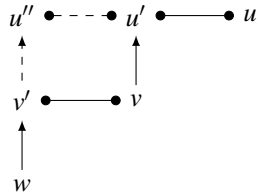
\begin{equation*}
      \begin{tikzpicture}[scale=.75]
        %% nodes
          \node (w)   at (0,0)   {$w$};
          \node (vp)  at (0,1.5) {$v'$};
          \node (v)   at (2,1.5) {$v$};
          \node (up)  at (2,3)   {$u'$};
          \node (u)   at (4,3)   {$u$};
          \node (upp) at (0,3)   {$u''$};
        %% arrows
          \draw[-latex]         (w)   to (vp);
          \draw[Circle-Circle]  (vp)  to (v);
          \draw[-latex]         (v)   to (up);
          \draw[Circle-Circle]  (up)  to (u);
          \draw[dashed, -latex] (vp)  to (upp);
          \draw[dashed,Circle-Circle]  (upp) to (up);
      \end{tikzpicture}
    \end{equation*}
    For antisymmetry, suppose $\bar{w} \qel \bar{v} \qel \bar{w}$.
    Then there exist $v'$ and $w'$ such that $w \leq v' R^{\sim} v \leq w' R^{\sim} w$.
    Using~Lemma~\ref{lem:properties}\eqref{it:properties-1}
    again we get $w''$ such that $v' \leq w'' R^{\sim} w'$.
    But this means that $w \leq w'' R^{\sim} w$, so we can use the Cartesian
    property of $\mo{M}$ to obtain $w'' = w$.
    We then get $v' = w$ because $\leq$ is antisymmetric.
    This implies $w R^{\sim} v$, hence $\bar{w} = \bar{v}$.
    
  \medskip\noindent
  \textit{The state domains are increasing.}
    Suppose $\bar{w} \qel \bar{v}$ and $\tilde{u} \in \I(\s, \bar{w})$.
    Then there exist $v', s \in W$ such that $w \leq v' R^{\sim} v$
    and $w R^{\sim} s \leq^{\sim} u$. We can use 
    Lemma~\ref{lem:properties}\eqref{it:properties-1} to find
    some $t$ such that $s \leq t$ and $v' R^{\sim} t$. Pictorially:
    
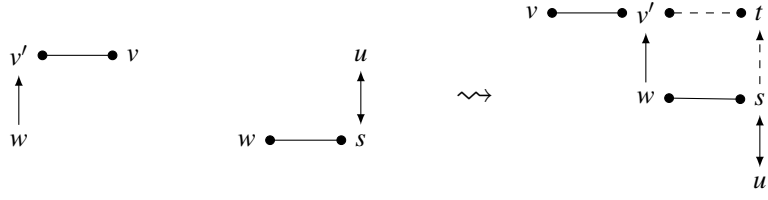
\begin{equation*}
      \begin{tikzpicture}[scale=.75]
        %% Fig 1
        %% nodes
          \node (w) at (0,0) {$w$};
          \node (vp) at (0,1.5) {$v'$};
          \node (v) at (2,1.5) {$v$};
        %% arrows
          \draw[-latex] (w) to (vp);
          \draw[Circle-Circle] (vp) to (v);
        %% Fig 2
        %% nodes
          \node (w) at (4,0) {$w$};
          \node (s) at (6,0) {$s$};
          \node (u) at (6,1.5) {$u$};
        %% arrows
          \draw[Circle-Circle] (w) to (s);
          \draw[latex-latex] (s) to (u);
        %% Implies
          \node (imp) at (8,.75) {\Large{$\rightsquigarrow$}};
        %% Fig 3
        %% nodes
          \node (v)  at (9 ,2.25) {$v$};
          \node (w)  at (11, .75) {$w$};
          \node (vp) at (11,2.25) {$v'$};
          \node (s)  at (13, .72) {$s$};
          \node (t)  at (13,2.25) {$t$};
          \node (u)  at (13,-.75) {$u$};
        %% arrows
          \draw[-latex] (w) to (vp);
          \draw[Circle-Circle] (w) to (s);
          \draw[dashed,-latex] (s) to (t);
          \draw[dashed, Circle-Circle] (vp) to (t);
          \draw[Circle-Circle] (v) to (vp);
          \draw[latex-latex] (s) to (u);
      \end{tikzpicture}
    \end{equation*}
    Combining these we get $v R^{\sim} w' R^{\sim} t$
    and $t \leq^{\sim} s \leq^{\sim} u$.
    Therefore $v R^{\sim} t \leq^{\sim} u$, so that $\tilde{u} \in \I(\s, \bar{v})$.
  
  \medskip\noindent
  \textit{The neighbourhood domains are increasing.}
    Suppose $\bar{w} \qel \bar{v}$ and $a \in \I(\n, \bar{w})$.
    If $a \in N_{triv}$ then $a \in \I(\n, \bar{v})$ by definition.
    If not, then there exist $v'$ and $x$ such that $w \leq v' R^{\sim} v$
    and $w R^{\sim} x$ and $x \in \dom(a)$.
    Using Lemma~\ref{lem:properties}\eqref{it:properties-1} we can find
    some $x' \geq x$ such that $v' R^{\sim} x'$, hence $v R^{\sim} x'$.
    Since $x \leq x'$ we have $x' \in \dom(a)$, hence $a \in \I(\n, \bar{v})$.
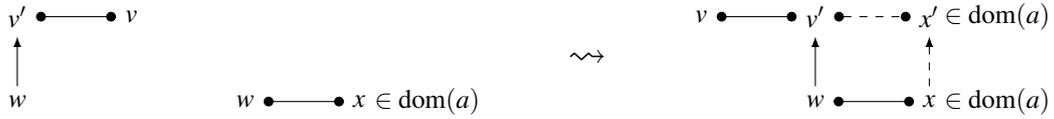
    \begin{equation*}
      \begin{tikzpicture}[scale=.75]
        %% Fig 1
        %% nodes
          \node (w) at (0,0) {$w$};
          \node (vp) at (0,1.5) {$v'$};
          \node (v) at (2,1.5) {$v$};
        %% arrows
          \draw[-latex] (w) to (vp);
          \draw[Circle-Circle] (vp) to (v);
        %% Fig 2
        %% nodes
          \node (w) at (4,0) {$w$};
          \node (s) at (6,0) {$x$};
        %% arrows
          \draw[Circle-Circle] (w) to (s);
          \draw (s) node[right]{$\; \in \dom(a)$};
        %% Implies
          \node (imp) at (10,.75) {\Large{$\rightsquigarrow$}};
        %% Fig 3
        %% nodes
          \node (v)  at (12,1.5) {$v$};
          \node (w)  at (14,0) {$w$};
          \node (vp) at (14,1.5) {$v'$};
          \node (x)  at (16,0) {$x$};
          \node (xp) at (16,1.5) {$x'$};
%        %% arrows
          \draw[-latex] (w) to (vp);
          \draw[Circle-Circle] (w) to (x);
          \draw[dashed,-latex] (x) to (xp);
          \draw[dashed, Circle-Circle] (vp) to (xp);
          \draw[Circle-Circle] (v) to (vp);
          \draw (x) node[right]{$\; \in \dom(a)$};
          \draw (xp) node[right]{$\; \in \dom(a)$};
      \end{tikzpicture}
    \end{equation*}
    
  \medskip\noindent
  \textit{The interpretation of $\N$ is increasing.}
    Suppose $\bar{w} \qel \bar{v}$, $\tilde{x} \in \I(\s, \bar{w})$,
    $a \in \I(\n, \bar{w})$ and $(\tilde{x}, a) \in \mc{I}(\N, \bar{w})$.
    Then there exist $v', s$ such that $w \leq v' R^{\sim} v$
    and $w R^{\sim} s \leq^{\sim} x$, and because $(\tilde{x}, a) \in \I(\N, \bar{w})$
    we als have $t \in \dom(a)$
    for all $t$ such that $w R^{\sim} t \leq^{\sim} x$.
    So $s \in \dom(a)$. Given $w \leq v' R^{\sim} v$ and $w R^{\sim} s$,
    we can use Lemma~\ref{lem:properties}\eqref{it:properties-1} to find $s'$ such that
    $s \leq s'$ and $v R^{\sim} s'$. So $s' \in \dom(a)$.
    By Lemma~\ref{lem:properties}\eqref{it:properties-2}, for any world $t'$
    such that $v R^{\sim} t' \leq^{\sim} x$ we have $t' = s'$.
    This proves that $(\tilde{x}, a) \in \mc{I}(\N, \bar{v})$.
    In a picture:
    
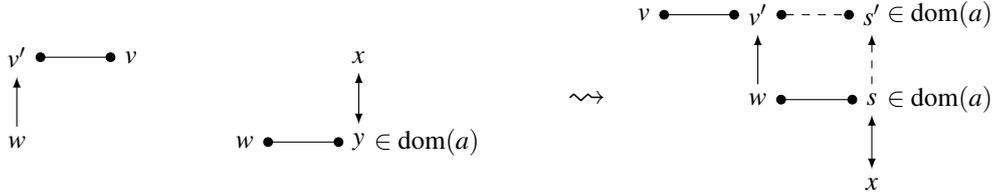
\begin{equation*}
      \begin{tikzpicture}[scale=.75]
        %% Fig 1
        %% nodes
          \node (w) at (0,0) {$w$};
          \node (vp) at (0,1.5) {$v'$};
          \node (v) at (2,1.5) {$v$};
        %% arrows
          \draw[-latex] (w) to (vp);
          \draw[Circle-Circle] (vp) to (v);
        %% Fig 2
        %% nodes
          \node (w) at (4,0) {$w$};
          \node (s) at (6,0) {$y$};
          \node (x) at (6,1.5) {$x$};
        %% arrows
          \draw[Circle-Circle] (w) to (s);
          \draw[latex-latex] (s) to (x);
          \draw (s) node[right]{$\; \in \dom(a)$};
        %% Implies
          \node (imp) at (10,.75) {\Large{$\rightsquigarrow$}};
        %% Fig 3
        %% nodes
          \node (v)  at (11,2.25) {$v$};
          \node (w)  at (13, .75) {$w$};
          \node (vp) at (13,2.25) {$v'$};
          \node (s)  at (15, .75) {$s$};
          \node (sp) at (15,2.25) {$s'$};
          \node (x)  at (15,-.75) {$x$};
        %% arrows
          \draw[-latex] (w) to (vp);
          \draw[Circle-Circle] (vp) to (v);
          \draw[Circle-Circle] (w) to (s);
          \draw[dashed, Circle-Circle] (vp) to (sp);
          \draw[dashed, -latex] (s) to (sp);
          \draw[latex-latex] (s) to (x);
          \draw (s) node[right]{$\; \in \dom(a)$};
          \draw (sp) node[right]{$\; \in \dom(a)$};
      \end{tikzpicture}
    \end{equation*}
    
  \medskip\noindent
  \textit{The interpretation of $\E$ is increasing.}
    Suppose $\bar{w} \qel \bar{v}$, $\tilde{x} \in \I(\s, \bar{w})$,
    $a \in \I(\n, \bar{w})$ and $(a, \tilde{x}) \in \mc{I}(\E, \bar{w})$.
    Then there exist $y, z$ such that $w R^{\sim} y$, $z \in a(y)$
    and $z \leq^{\sim} x$. Using Lemma~\ref{lem:properties}\eqref{it:properties-1}
    and~\eqref{it:in-2} we can find $y', z'$ such that
    $v' R^{\sim} y'$, $z' \in a(y')$ and $y \leq y'$ and $z \leq z'$.
    This implies that $v R^{\sim} y'$ and $z' \leq^{\sim} x$,
    so that we have $(a, \tilde{x}) \in \I(\E, \bar{v})$.
    Pictorially:
    
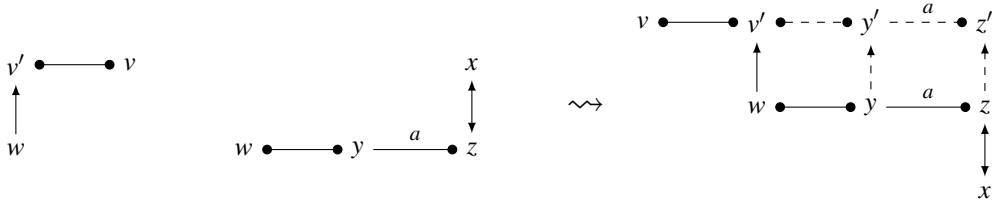
\begin{equation*}
      \begin{tikzpicture}[scale=.75]
        %% Fig 1
        %% nodes
          \node (w) at (0,0) {$w$};
          \node (vp) at (0,1.5) {$v'$};
          \node (v) at (2,1.5) {$v$};
        %% arrows
          \draw[-latex] (w) to (vp);
          \draw[Circle-Circle] (vp) to (v);
        %% Fig 2
        %% nodes
          \node (w) at (4,0) {$w$};
          \node (y) at (6,0) {$y$};
          \node (z) at (8,0) {$z$};
          \node (x) at (8,1.5) {$x$};
        %% arrows
          \draw[Circle-Circle] (w) to (y);
          \draw[-Circle] (y) to node[above]{\footnotesize{$a$}} (z);
          \draw[latex-latex] (z) to (x);
        %% Implies
          \node (imp) at (10,.75) {\Large{$\rightsquigarrow$}};
        %% Fig 3
        %% nodes
          \node (v)  at (11,2.25) {$v$};
          \node (w)  at (13, .75) {$w$};
          \node (vp) at (13,2.25) {$v'$};
          \node (y)  at (15, .75) {$y$};
          \node (yp) at (15,2.25) {$y'$};
          \node (z)  at (17,.75)  {$z$};
          \node (zp) at (17,2.25) {$z'$};
          \node (x)  at (17,-.75) {$x$};
        %% arrows
          \draw[-latex] (w) to (vp);
          \draw[Circle-Circle] (vp) to (v);
          \draw[Circle-Circle] (w) to (y);
          \draw[-Circle] (y) to node[above]{\footnotesize{$a$}} (z);
          \draw[dashed, Circle-Circle] (vp) to (yp);
          \draw[dashed, -latex] (y) to (yp);
          \draw[dashed, -Circle] (yp) to node[above]{\footnotesize{$a$}} (zp);
          \draw[dashed, -latex] (z) to (zp);
          \draw[latex-latex] (z) to (x);
      \end{tikzpicture}
    \end{equation*}

  \medskip\noindent
  \textit{The interpretation of the $\P_i$ is increasing.}
    Suppose $\bar{w} \qel \bar{v}$ and $\tilde{x} \in \I(\s, \bar{w})$
    and $\tilde{x} \in \I(\P_i, \bar{w})$.
    Let $s$ be such that $w R^{\sim} s \leq^{\sim} x$.
    Then $s \in V(p_i)$. Similar to above we find some $s'$ such that
    $v R^{\sim} s' \leq^{\sim} x$ and $s \leq s'$. This implies $s' \in V(p)$,
    so that (since such an $s'$ between $v$ and $x$ is unique)
    we have $\tilde{x} \in \I(\P_i, \bar{v})$.
\end{proof}

\begin{proposition}\label{prop:Cartesian}
  Let $\mo{M} = (W, \leq, N, V)$ be a coherent and Cartesian intuitionistic
  neighbourhood model
  and $(\mo{M}^{\circ})^{\bul} = (\ov{W}^{\bul}, \leqq, N^{\bul}, V^{\bul})$
  the intuitionistic neighbourhood model induced by $\mo{M}^{\circ}$.
  Then $\mo{M} \cong (\mo{M}^{\circ})^{\bul}$.
\end{proposition}
\begin{proof}
  We first establish a bijection between $W$ and $\ov{W}^{\bul}$.
  For each $w \in W$ we have $w R^{\sim} w \leq^{\sim} w$,
  so $\langle \bar{w}, \tilde{w} \rangle \in \ov{W}^{\bul}$.
  Define $\alpha(w) := \langle \bar{w}, \tilde{w} \rangle$.
  If $\langle \bar{w}, \tilde{x} \rangle \in \overline{W}^{\bul}$ then
  by definition we can find some $v \in W$ such that $w R^{\sim} v \leq^{\sim} x$,
  and by Lemma~\ref{lem:properties}\eqref{it:properties-2} this $v$ is unique.
  So we can define $\beta(\langle \bar{w}, \tilde{x} \rangle)$ as the unique
  $v$ such that $w R^{\sim} v \leq^{\sim} x$.
  Since $w$ satisfies $w R^{\sim} w \leq^{\sim} w$ we have
  $\beta(\alpha(w)) = \beta(\langle \bar{w}, \tilde{w} \rangle) = w$.
  Conversely, if $\langle \bar{w}, \tilde{x} \rangle \in \ov{W}^{\bul}$
  and $v = \beta(\langle \bar{w}, \tilde{x} \rangle)$, then
  $w R^{\sim} v \leq^{\sim} x$ so $\bar{w} = \bar{v}$ and $\tilde{x} = \tilde{v}$,
  so that $\alpha(\beta(\langle \bar{w}, \tilde{x} \rangle)) = \alpha(v) = \langle \bar{v}, \tilde{v} \rangle = \langle \bar{w}, \tilde{x} \rangle$.
  So $\alpha : W \to \ov{W}^{\bul}$ is a bijection with inverse $\beta$.

  Next, let us show that $\alpha : (W, \leq) \to (\ov{W}^{\bul}, \leqq)$ is
  an order isomorphism.
  If $w \leq v$ then $\bar{w} \qel \bar{v}$ and $\tilde{w} = \tilde{v}$,
  so $\langle \bar{w}, \tilde{w} \rangle = \langle \bar{w}, \tilde{v} \rangle \leqq \langle \bar{v}, \tilde{v} \rangle$.
  Conversely, suppose $\langle \bar{w}, \tilde{w} \rangle \leqq \langle \bar{v}, \tilde{v} \rangle$. Then $\bar{w} \qel \bar{v}$ and $\tilde{w} = \tilde{v}$.
  This means that we can find a $v'$ such that $w \leq v' R^{\sim} v$.
  Since $w \leq^{\sim} v$ we then get $v \leq^{\sim} v' R^{\sim} v$,
  so since $\mo{M}$ is cartesian $v = v'$. This implies $w \leq v$.

  If $a \in N$ and $\dom(a) = \emptyset$ then $a$ is in the neighbourhood
  domain of all worlds in $\ov{W}$, hence $a^{\bul} \in N^{\bul}$.
  If $\dom(a) \neq \emptyset$ then we can find some $w \in W$ such that
  $w \in \dom(a)$, which implies $a \in \I(\n, \bar{w})$ and hence
  $a^{\bul} \in N^{\bul}$.
  This yields a map $N \to N^{\bul} : a \mapsto a^{\bul}$,
  which is a bijection by construction.
  To see that $a$ and $a^{\bul}$ correspond on their domains, compute
  $$
    w \in \dom(a)
      \iff a \in \I(\n, \bar{w}) \text{ and } (\tilde{w}, a) \in \I(\N, \bar{w})
      \iff \langle \bar{w}, \tilde{w} \rangle \in \dom(a^{\bul}).
  $$
  So we have $w \in \dom(a)$ if and only if $\alpha(w) \in \dom(a^{\bul})$.
  
  Next we prove that $v \in a(w)$ if and only if $\alpha(v) \in a^{\bul}(\alpha(w))$.
  Suppose $v \in a(w)$.
  Then $(\tilde{w}, a) \in \I(\N, \bar{w})$, because $w R^{\sim} w \leq^{\sim} w$ and $w \in \dom(a)$.
  Furthermore, we have $w R^{\sim} w$ and $v \in a(w)$ and $v \leq^{\sim} v$,
  so by definition $(a, \tilde{v}) \in \I(\E, \bar{w})$.
  This then implies $\langle \bar{w}, \tilde{v} \rangle \in a^{\bul}(\langle \bar{w}, \tilde{w} \rangle)$. Finally, since $v \in a(w)$ we have $w R^{\sim} v$, so
  $\bar{w} = \bar{v}$, hence $\langle \bar{v}, \tilde{v} \rangle \in a^{\bul}(\langle \bar{w}, \tilde{w} \rangle)$.

  For the converse, assume $\langle \bar{v}, \tilde{v} \rangle \in a^{\bul}(\langle \bar{w}, \tilde{w} \rangle)$.
  Then $\bar{w} = \bar{v}$, and $(\tilde{w}, a) \in \I(\N, \bar{w})$
  and $(a, \tilde{v}) \in \I(\E, \bar{w})$.
  Since $(\tilde{w}, a) \in \I(\N, \bar{w})$ we have $w \in \dom(a)$,
  and $(a, \tilde{v}) \in \I(\E, \bar{w})$ implies the existence of $y, z \in \I(\s, \bar{w})$
  such that $w R^{\sim} y$ and $z \in a(y)$ and $z \leq^{\sim} v$.
  In a diagram:
  
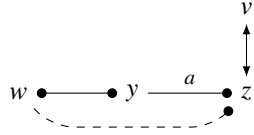
\begin{equation*}
    \begin{tikzpicture}[scale=.75]
      %% nodes
        \node (w) at (0,0) {$w$};
        \node (y) at (2,0) {$y$};
        \node (z) at (4,0) {$z$};
        \node (v) at (4,1.5) {$v$};
      %% arrows
        \draw[Circle-Circle] (w) to (y);
        \draw[-Circle] (y) to node[above]{\footnotesize{$a$}} (z);
        \draw[latex-latex] (z) to (v);
        \draw[dashed, -Circle, rounded corners=.8em] (w) to (.6,-.6) to (3.4,-.6) to (z);
    \end{tikzpicture}
  \end{equation*}
  Then $y \in \dom(a)$, so since $w R^{\sim} y$ and $\mo{M}$ is N-Cartesian
  we find $z \in a(w)$.
  Since $\bar{w} = \bar{v}$, we find $z = \beta(\bar{v}, \tilde{v}) = v$,
  hence $v \in a(w)$.
  
  Lastly, it follows immediately from the definitions of $\I(\P_i)$
  and $V^{\bul}$ that $w \in V(p_i)$ if and only if
  $\langle \bar{w}, \tilde{w} \rangle \in V^{\bul}(p_i)$.
  We conclude that $\mo{M} \cong (\mo{M}^{\circ})^{\bul}$.
\end{proof}

  Now we come to the main theorem of the section.

\begin{theorem}\label{thm:Cartesian}
  An intuitionistic neighbourhood model $\mo{M}$ is isomorphic to
  an intuitionistic neighbourhood model of the form $\fomo{M}^{\bul}$,
  for some $\IFOM$-structure $\fomo{M}$, if and only if
  it is coherent and Cartesian.
\end{theorem}
\begin{proof}
  The direction from left to right follows from the fact that every
  intuitionistic neighbourhood model of the form $\fomo{M}^{\bul}$ is
  coherent and Cartesian.
  Conversely, Proposition~\ref{prop:Cartesian} allows us to find a
  $\IFOM$-structure $\fomo{M}$ such that
  $\fomo{M}^{\bul}$ is isomorphic to $\mo{M}$ for each coherent and Cartesian model $\mo{M}$.
\end{proof}

\begin{corollary}\label{cor:Cartesian}
  Let $\mo{M} = (W, \leq, N, V)$ be a coherent and Cartesian intuitionistic
  neighbourhood model. For every formula $\phi$ and every $w \in W$,
  \begin{equation*}
    \mo{M}, w \Vdash \phi
      \iff \mo{M}^{\circ}, \bar{w}, \tilde{w} \Vdash \phi.
  \end{equation*}
\end{corollary}
\begin{proof}
  We have $\mo{M}, w \Vdash \phi$ if and only if
  $(\mo{M}^{\circ})^{\bul}, \langle \bar{w}, \tilde{w} \rangle \Vdash \phi$
  because of the isomorphism from Proposition~\ref{prop:Cartesian},
  and $(\mo{M}^{\circ})^{\bul}, \langle \bar{w}, \tilde{w} \rangle \Vdash \phi$
  if and only if $\mo{M}^{\circ}, \bar{w}, \tilde{w} \models \phi$
  by Proposition~\ref{prop:truth-str-to-inm}.
\end{proof}

%================================================================================
\subsection{Unravelling}\label{subsec:unravelling}

  In the previous section, we have seen that the class of $\IFOM$-structures
  and the class of coherent Cartesian intuitionistic neighbourhood models satisfy
  the same consecutions. Therefore $\Gamma \Vdash^{\inm} \phi$ implies $\Gamma \models \phi$.
  We prove that the reverse implication also holds by showing that for any
  intuitionistic neighbourhood model $\mo{M}$ and world $w$,
  we can construct a coherent and Cartesian intuitionistic neighbourhood  model
  $\mo{M}^{ur}_w$ in which there is a world $w'$ that satisfies precisely
  the same $\Lbd$-formulas as $w$.
  As a consequence, any consecution that is falsifiable in an intuitionistic
  neighbourhood model is falsifiable in a coherent and Cartesian one.
  We prove this using an adaptation of the unravelling technique
  for Kripke frames, see e.g.~\cite[Section~2.1]{BRV01} or~\cite[Section~3.3]{ChaZak97}.
  
  We start by defining paths through an intuitionistic neighbourhood frame or model.

\begin{definition}
  Let $\mo{F} = (W, \leq, N)$ be an intuitionistic neighbourhood frame.
  A \emph{path} through $\mo{F}$ is a sequence of the form
  $(w_0, r_0, w_1, r_1, \ldots, r_{n-1}, w_n)$ such that $w_j \in W$ for all
  $j \in \{ 0, \ldots, n \}$, and for all $i \in \{ 0, \ldots, n-1 \}$ either
  \begin{itemize}
    \item $r_i = {\leq}$ and $w_i \leq w_{i+1}$; or
    \item $r_i \in N$ and $w_i \in \dom(r_i)$ and $w_{i+1} \in r_i(w_i)$.
  \end{itemize}
  If $\vec{w} = (w_0, r_0, w_1, r_1, \ldots, r_{n-1}, w_n)$ is a path
  in $\mo{F}$ then we let $\first(\vec{w}) = w_0$ and $\last(\vec{w}) = w_n$,
  and we define $\length(\vec{w}) = n$.
  If $a \in N$ and $x \in W$ are such that $w_n \in \dom(a)$ and $x \in a(w_n)$,
  then we write $\vec{w} : a : x$ for the
  path obtained from appending $a$ and $x$ to $\vec{w}$.
  If $\vec{v} = (v_0, s_0, v_1, s_1, \ldots, s_{m-1}, v_m)$ is another path
  and $w_n = v_0$, then we denote by
  $\vec{w} \glue \vec{v}$ the path obtained from composing $\vec{w}$ and
  $\vec{v}$ and identifying $w_n$ with $v_0$.
  
  The path $\vec{w}$ is called an \emph{order path} if $r_i = {\leq}$
  for all $i \in \{ 0, \ldots, n-1 \}$,
  and a \emph{neighbourhood path} if $r_i \in N$ for all $i \in \{ 0, \ldots, n-1 \}$.
  Finally, we call $\vec{w}$ an \emph{unravelling path} if there exist
  an order path $\vec{v}$ and a neighbourhood path $\vec{u}$ such that
  $\vec{w} = \vec{v} \glue \vec{u}$.
  In this case, we define $\vec{w}_{ord} := \vec{v}$ and $\vec{w}_{nbd} := \vec{u}$.
\end{definition}

\begin{definition}
  Let $\mo{F} = (W, \leq, N)$ be a coherent intuitionistic neighbourhood frame
  and let $s \in W$ be a world. Then we let $W_s^{ur}$ denote the set of
  unravelling paths starting at $s$. 

  Let $\vec{w}, \vec{v} \in W^{ur}_s$ and suppose
  $\vec{w}_{nbd} = (w_0, a_0, w_1, a_1, \ldots, a_{n-1}, w_n)$
  and $\vec{v}_{nbd} = (v_0, b_0, v_1, b_1, \ldots, b_{m-1}, v_m)$.
  Then we let $\vec{w} \leq^{ur} \vec{v}$ if $n = m$, 
  $a_i = b_i$ for each $i \in \{ 0, \ldots, n-1 \}$, and there exist
  order-paths $\vec{t}_0, \ldots, \vec{t}_n$ such that
  \begin{itemize}
    \item $\length(\vec{t}_0) = \ldots = \length(\vec{t}_n)$; and
    \item for each $j \in \{ 0, \ldots, n \}$,
          $\first(\vec{t}_j) = w_j$ and
          $\last(\vec{t}_j) = v_j$
    \item $\vec{v}_{ord} = \vec{w}_{ord} \glue \vec{t}_0$
  \end{itemize}
\end{definition}

\newcommand{\exRed}{\textcolor{red}}
\newcommand{\exGreen}{\textcolor{ForestGreen}}
\newcommand{\exBlue}{\textcolor{blue}}

\begin{example}\label{exm:unrav-1}
  Consider the intuitionistic neighbourhood frame $(W, \leq, N)$ with
  $W = \{ s, t, u, v, w, x \}$, ordered by the reflexive closure of
  $w \leq v, u \leq t$ and $s \leq x$, and with $N = \{ a \}$ given by
  $a(v) = \{ u, t \}$, $a(u) = \{ s \}$ and $a(t) = \{ x \}$.
  This is depicted in Figure~\ref{fig:unrav-1}.
  Consider the unravelling paths
  \begin{equation*}
    \exRed{\vec{p}_1 = (w, {\leq}, v, a, u, a, s)}, \qquad
    \exGreen{\vec{p}_2 = (w, {\leq}, v, a, t, a, x)}, \qquad
    \exBlue{\vec{p}_3 = (w, {\leq}, v, {\leq}, v, a, u, a, s)}.
  \end{equation*}
  Then their order parts are
  $(\exRed{\vec{p}_1})_{ord} = (\exGreen{\vec{p}_2})_{ord} = (w, {\leq}, v)$
  and $(\exBlue{\vec{p}_3})_{ord} = (w, {\leq}, v, {\leq}, v)$.
  Their neighbourhood paths are
  $(\exRed{\vec{p}_1})_{nbd} = (v, a, u, a, s)$
  and $(\exGreen{\vec{p}_2})_{nbd} = (\exBlue{\vec{p}_3})_{nbd} = (v, a, t, a, x)$.
  
  \begin{figure}[h!]
    \centering
    \begin{tikzpicture}[scale=.75]
      %% nodes
        \node (w) at (0,0)   {$w$};
        \node (v) at (0,1.5) {$v$};
        \node (u) at (2,1.5) {$u$};
        \node (t) at (2,3)   {$t$};
        \node (s) at (4,1.5) {$s$};
        \node (x) at (4,3)   {$x$};
      %% edges
        \draw[-latex] (w) to (v);
        \draw[-latex] (u) to (t);
        \draw[-latex] (s) to (x);
        \draw[-Circle] (v) to (u);
        \draw[-Circle] (v) to (t);
        \draw[-Circle] (u) to (s);
        \draw[-Circle] (t) to (x);
      %% pahts
        \draw[thick, red, rounded corners=1mm, cap=round]
              (.25,-.1) -- (.25,1.25) -- node[below]{$\vec{p}_1$} 
              (4.1,1.25);
        \draw[thick, ForestGreen, rounded corners=1mm, cap=round]
              (-.25,-.1) -- (-.25,1.6)
                        -- (1.9,3.25) -- (4.1,3.25) node[right]{$\vec{p}_2$};
        \draw[thick, blue, rounded corners=.5mm, cap=round]
              (-.4,-.1) -- (-.4,1.6) arc(320:-20:.5)
                        -- node[xshift=-2mm,yshift=2mm]{$\vec{p}_3$} (1.8,3.4)
                        -- (4.1,3.4);
    \end{tikzpicture}
    \caption{The intuitionistic neighbourhood frame from Example~\ref{exm:unrav-1}.
             The arrows denote the order $\leq$, and the lines with a dot at the
             end denote the neighbourhood relation,
             i.e.%
             $\protect\tikz[baseline=-.8mm]
               {\protect\node (w) at (0,0) {$v$};
                \protect\node (v) at (1,0) {$u$};
                \protect\draw[-Circle] (w) to (v);}$%
             means $u \in a(v)$.}
    \label{fig:unrav-1}
  \end{figure}
  
  In this example we have $\exRed{\vec{p}_1} \nleq^{ur} \exGreen{\vec{p}_2}$:
  while we can find paths $\vec{t}_0, \vec{t}_1$ and $\vec{t}_2$
  connecting $(\exRed{\vec{p}_1})_{nbd}$ and $(\exGreen{\vec{p}_2})_{nbd}$ pointwise,
  these need to have length $> 1$ because for example
  $\vec{t}_2$ needs to be such that $\first(\vec{t}_2) = s$ and $\last(\vec{t}_2) = x$.
  In order to make $\vec{t}_0$ have length $> 1$ we need to take
  $\vec{t}_0 = (v, {\leq}, v)$ (and similar if we want to make the length of $\vec{t}_0$ even longer).
  But this implies
  \begin{equation*}
    (\exGreen{\vec{p}_2})_{ord}
      = (w, {\leq}, v)
      \neq (w, {\leq}, v, {\leq}, v)
      = (w, {\leq}, v) \glue (v, {\leq}, v)
      = (\exRed{\vec{p}_1})_{ord} \glue \vec{t}_0
  \end{equation*}
  On the other hand, taking $\vec{t}_0 = (v, {\leq}, v)$,
  $\vec{t}_1 = (u, {\leq}, t)$ and $\vec{t}_2 = (s, {\leq}, x)$
  shows that $\exRed{\vec{p}_1} \leq^{ur} \exBlue{\vec{p}_3}$.

  If we would allow the $\vec{t}_i$ to have different lengths,
  then we would get $\exRed{\vec{p}_1} \leq^{ur} \exGreen{\vec{p}_2}$.
  The reason we do not do this is that it would falsify the $R^{\sim}$-Cartesian
  property in the unravelling.
  We will see in Lemma~\ref{lem:unrav}\eqref{it:unrav-4} that two paths are equivalent
  modulo $R^{\sim}$ if their order paths coincide, so
  $\exRed{\vec{p}_1} R^{\sim} \exGreen{\vec{p}_2}$.
  Then, if we were to loosen our definition of $\leq^{ur}$, we would
  get $\exRed{\vec{p}_1} \leq^{\sim} \exGreen{\vec{p}_2} R^{\sim} \exRed{\vec{p}_1}$.
  Since $\exRed{\vec{p}_1} \neq \exGreen{\vec{p}_2}$ this violates the $R^{\sim}$-Cartesian property.
\end{example}

\begin{lemma}
  The pair $(W_s^{ur}, \leq^{ur})$ is an intuitionistic Kripke frame.
\end{lemma}
\begin{proof}
  Reflexivity of $\leq^{ur}$ follows from reflexivity of $\leq$,
  transitivity of $\leq^{ur}$ follows from the definition.
  For antisymmetry, suppose $\vec{w} \leq^{ur} \vec{v} \leq^{ur} \vec{w}$.
  Then $\vec{w}_{ord}$ extends $\vec{v}_{ord}$ and $\vec{v}_{ord}$ extends
  $\vec{w}_{ord}$ so we must have $\vec{w}_{ord} = \vec{v}_{ord}$.
  This implies that the paths $\vec{t}_0, \ldots, \vec{t}_n$ connecting
  $\vec{w}_{nbd}$ and $\vec{v}_{nbd}$, where $n = \length(\vec{w}_{nbd})$,
  all have length 1. Then it follows from the definition of $\leq^{ur}$ that
  $\vec{w}_{nbd} = \vec{v}_{nbd}$. Therefore $\vec{w} = \vec{v}$.
\end{proof}

  We now define coherent neighbourhoods for $(W_s^{ur}, \leq^{ur})$.
  
\begin{definition}
  Let $\vec{w} \in W_s^{ur}$ and $a \in N$ such that $\last(\vec{w}) \in \dom(a)$.
  Define the neighbourhood $a_{\vec{w}} : W_s^{ur} \rightharpoonup \fun{P}(W_s^{ur})$
  to be the partial function with domain
  $\dom(a_{\vec{w}}) = \{ \vec{v} \in W_s^{ur} \mid \vec{w} \leq^{ur} \vec{v} \}$
  given by
  $$
    a_{\vec{w}}(\vec{v}) = \{ \vec{v} : a : x \mid x \in a(\last(\vec{v})) \}.
  $$
\end{definition}

  We verify that these are indeed coherent intuitionistic neighbourhoods.

\begin{lemma}\label{lem:unrav-coh}
  Let $\vec{w} \in W_s^{ur}$ and $a \in N$ such that $\last(\vec{w}) \in \dom(a)$.
  Then the partial function $a_{\vec{w}}$ is a coherent intuitionistic neighbourhood
  for $(W_s^{ur}, \leq^{ur})$.
\end{lemma}
\begin{proof}
  The domain of $a_{\vec{w}}$ is an upset by definition,
  so it is an intuitionistic neighbourhood.
  
  For~\eqref{it:in-2}, suppose $\vec{v} \leq^{ur} \vec{u}$
  and $(\vec{v} : a : x) \in a_{\vec{w}}(\vec{v})$.
  (Note that any element of $a_{\vec{w}}(\vec{v})$ must be of this shape.)
  By definition of $\leq^{ur}$, there exist order paths
  $\vec{t}_0, \ldots, \vec{t}_n$ between $\vec{v}_{nbd}$ and $\vec{u}_{nbd}$,
  where $n = \length(\vec{v}_{nbd}) = \length(\vec{u}_{nbd})$.
  The definition of $a_{\vec{w}}$ yields $x \in a(\last(\vec{v}))$.
  Since $a$ is an intuitionistic neighbourhood of $\mo{F}$,
  we can use~\eqref{it:in-2} repeatedly to find some path $\vec{s}$ of the
  same length as $\vec{t}_n$ such that $\last(\vec{s}) \in a(\last(\vec{u}))$.
  Let $y = \last(\vec{s})$. Then $(\vec{u} : a : y) \in a_{\vec{w}}(\vec{u})$,
  and $\vec{t}_0, \ldots, \vec{t}_n, \vec{s}$ witness that
  $(\vec{v} : a : x) \leq^{ur} (\vec{u} : a : y)$.
  This proves that~\eqref{it:in-2} holds for $a_{\vec{w}}$.

  To see that~\eqref{it:in-3} holds, suppose
  $(\vec{v} : a : x) \in a_{\vec{w}}(\vec{v})$
  and $(\vec{v} : a : x) \leq^{ur} \vec{u}$.
  Then there exists paths $\vec{t}_0, \ldots, \vec{t}_n$ between
  $(\vec{v} : a : x)_{nbd}$ and $\vec{u}_{nbd}$ witnessing
  $(\vec{v} : a : x) \leq^{ur} \vec{u}$.
  Furthermore, $\vec{u}$ must be of the form $\vec{s} : a : y$.
  The order-paths $\vec{t}_0, \ldots, \vec{t}_{n-1}$ then show that
  $\vec{v} \leq^{ur} \vec{s}$, and by definition of
  $a_{\vec{w}}$ we have $\vec{u} \in a_{\vec{w}}(\vec{s})$.
  So $a_{\vec{w}}$ satisfies~\eqref{it:in-3}.
\end{proof}

\begin{definition}
  Let $\mo{M} = (W, \leq, N, V)$ be a coherent intuitionistic neighbourhood model.
  Let $N^{ur} = \{ a_{\vec{w}} \mid \vec{w} \in W_s^{ur}, a \in N \text{ and } \last(\vec{w}) \in \dom(a) \}$.
  Define a valuation $V^{ur}$ for $(W_s^{ur}, \leq^{ur})$ by
  \begin{equation*}
    V^{ur}(p_i) = \{ \vec{w} \in W_s^{ur} \mid \last(\vec{w}) \in V(p_i) \}.
  \end{equation*}
  Then the \emph{unravelling of $\mo{M}$ from $s$} is the
  intuitionistic neighbourhood model
  \begin{equation*}
    \mo{M}_s^{ur} := (W_s^{ur}, \leq^{ur}, N^{ur}, V^{ur}).
  \end{equation*}
\end{definition}

  We know that $\mo{M}_s^{ur}$ is coherent.
  Before proving that an intuitionistic neighbourhood model $\mo{M}$ and its
  unravelling are related in a truth-preserving way, we verify that the
  unravelling is Cartesian.

\begin{lemma}\label{lem:unrav}
  Let $\mo{M} = (W, \leq, N, V)$ be a coherent intuitionistic neighbourhood
  model and $\mo{M}_s^{ur} = (W_s^{ur}, \leq^{ur}, N^{ur}, V^{ur})$ its
  unravelling from $s$.
  Define $\vec{w} R \vec{v}$ if $\vec{v} \in a(\vec{w})$ for some $a \in N^{ur}$.
  \begin{enumerate}
    \item \label{it:unrav-1}
          If $\vec{w} \leq^{ur, \sim} \vec{v}$
          then $\length(\vec{w}_{nbd}) = \length(\vec{v}_{nbd})$.
    \item \label{it:unrav-2}
          If $\vec{w} \leq^{ur} \vec{v}$ and $\length(\vec{w}) = \length(\vec{v})$
          then $\vec{w} = \vec{v}$.
    \item \label{it:unrav-3}
          If $\vec{w} \leq^{ur,\sim} \vec{v}$ and $\length(\vec{w}) = \length(\vec{v})$
          then $\vec{w} = \vec{v}$.
    \item \label{it:unrav-4}
          $\vec{w} R^{\sim} \vec{v}$ if and only if
          $\vec{w}_{ord} = \vec{v}_{ord}$.
  \end{enumerate}
\end{lemma}
\begin{proof}
  The first item follows from the fact that $\vec{w} \leq^{ur} \vec{v}$
  implies $\length(\vec{w}_{nbd}) = \length(\vec{v}_{nbd})$, which holds
  by definition of $\leq^{ur}$.
  For the second item, suppose $\vec{w} \leq^{ur} \vec{v}$
  and $\length(\vec{w}) = \length(\vec{v})$. Then
  $\length(\vec{w}_{nbd}) = \length(\vec{v}_{nbd})$, hence
  $\length(\vec{w}_{ord}) = \length(\vec{v}_{ord})$. The latter implies
  that the order paths $\vec{t}_0, \ldots, \vec{t}_{\length(\vec{w}_{nbd})}$
  have length 1, which forces $\vec{w} = \vec{v}$.
  The third item follows from repeated application of the second.
  
  Lastly, note that $\vec{w} R \vec{v}$ if and only if
  $\vec{w} = (\vec{v} : a : x)$ or $\vec{v} = (\vec{w} : a : x)$ for some
  suitable $a \in N$ and $x \in W$. This entails that $\vec{w} R \vec{v}$
  implies $\vec{w}_{ord} = \vec{v}_{ord}$. Therefore
  $\vec{w} R^{\sim} \vec{v}$ implies $\vec{w}_{ord} = \vec{v}_{ord}$.
  Furthermore, we get that $\vec{w}_{ord} R^{\sim} \vec{w}$ for any
  $\vec{w} \in W_s^{ur}$. Therefore $\vec{w}_{ord} = \vec{v}_{ord}$
  implies $\vec{w} R^{\sim} \vec{v}$.
\end{proof}

\begin{proposition}
  Let $\mo{M} = (W, \leq, N, V)$ be a coherent intuitionistic neighbourhood model
  and $\mo{M}_s^{ur} = (W_s^{ur}, \leq^{ur}, N^{ur}, V^{ur})$ its unravelling
  from $s$. Then $\mo{M}_s^{ur}$ is Cartesian.
\end{proposition}
\begin{proof}
  To see that $\mo{M}_s^{ur}$ is $R^{\sim}$-Cartesian,
  suppose $\vec{w} \leq^{ur, \sim} \vec{v} R^{\sim} \vec{w}$.
  Then $\vec{w} \leq^{ur, \sim} \vec{v}$ implies $\length(\vec{w}_{nbd}) = \length(\vec{v}_{nbd})$ and $\vec{v} R^{\sim} \vec{w}$ implies
  $\vec{v}_{ord} = \vec{w}_{ord}$.
  Therefore $\length(\vec{w}) = \length(\vec{v})$,
  so that Lemma~\ref{lem:unrav}\eqref{it:unrav-3} gives $\vec{w} = \vec{v}$.
  
  Second, suppose $\vec{v} R^{\sim} \vec{u}$, $a_{\vec{w}} \in N^{ur}$
  and $\vec{v}, \vec{u} \in \dom(a_{\vec{w}})$.
  Then $\vec{w} \leq^{ur} \vec{v}$ and $\vec{w} \leq^{ur} \vec{u}$,
  so $\vec{v} \leq^{ur, \sim} \vec{u}$.
  This implies $\length(\vec{v}_{nbd}) = \length(\vec{u}_{nbd})$.
  Further, $\vec{v} R^{\sim} \vec{u}$ implies that
  $\vec{v}_{ord} = \vec{u}_{ord}$.
  It follows that $\length(\vec{v}) = \length(\vec{u})$,
  so we can use Lemma~\ref{lem:unrav}\eqref{it:unrav-3}
  to see that $\vec{v} = \vec{u}$.
  Clearly this implies $a_{\vec{w}}(\vec{v}) = a_{\vec{w}}(\vec{u})$,
  so $\mo{M}_s^{ur}$ is N-Cartesian.
\end{proof}

  Finally, we prove that the unravelling preserves truth,
  in the following sense:

\begin{proposition}\label{prop:unrav}
  Let $\mo{M} = (W, \leq, N, V)$ be an intuitionistic neighbourhood
  model, $s \in W$ and $\mo{M}_s^{ur} = (W_s^{ur}, \leq^{ur}, N^{ur}, V^{ur})$
  its unravelling from $s$.
  Then for all $\vec{w} \in W_s^{ur}$ and all $\phi \in \Lbd$,
  $$
    \mo{M}_s^{ur}, \vec{w} \Vdash \phi \iff \mo{M}, \last(\vec{w}) \Vdash \phi.
  $$
\end{proposition}
\begin{proof}
  We use induction on the structure of $\phi$.
  If $\phi = \bot$ then this is clearly true, and the case $\phi = p_i \in \Prop$
  follows from the definition of $V^{ur}$.
  The inductive cases for $\wedge$ and $\vee$ are routine.

  \medskip\noindent
  \textit{Induction step for $\phi = \psi \to \chi$.}
  Suppose $\mo{M}_s^{ur}, \vec{w} \Vdash \psi \to \chi$.
  Suppose $x$ is such that $\last(\vec{v}) \leq x$ and $\mo{M}, x \Vdash \psi$.
  Then we can use~\eqref{it:in-3} repeatedly to find an unravelling
  path $\vec{v}$ such that $\vec{w} \leq^{ur} \vec{v}$ and $\last(\vec{v}) = x$,
  see Figure~\ref{fig:use-N3}.
\begin{figure}[h!]
  \centering
    \begin{tikzpicture}[scale=.75]
      %%%% SECOND DIAGRAM
      %% nodes
        \node (00) at (0   ,.7) {$s$};
        \node (02) at (0   ,2) {$\circ$};
        \node (12) at (1.25,2) {$\circ$};
        \node (32) at (3.25,2) {$\circ$};
        \node (42) at (4.75 ,2) {$\circ$};
        \node (43) at (4.75 ,3.2) {$x$};
        \node (33) at (3.25,3.2) {$\circ$};
      %% edges
        \draw[-] (00) to (0,1.05);
        \draw[dotted,-, semithick] (0,1.05) to (0,1.5);
        \draw[-latex] (0,1.4) to (02);
        \draw[-Circle] (02) to (12);
        \draw[-] (12) to (1.9,2);
        \draw[dotted,-, semithick] (1.9,2) to (2.5,2);
        \draw[-Circle] (2.5,2) to (32);
        \draw[-Circle] (32) to node[above]{\footnotesize{$a$}} (42);
        \draw[-latex] (42) to (43);
        \draw[dashed,-latex] (32) to (33);
        \draw[dashed,-Circle] (33) to node[above]{\footnotesize{$a$}} (43);
      %% paths
        \draw[thick, red, rounded corners=1mm, cap=round]
              (.2,.65) -- (.2,1.8) -- node[below]{$\vec{w}$} 
              (4.8,1.8);
      %%%% ARROW
        \node at (6.4,2) {\Large{$\rightsquigarrow$}};
      %%%% SECOND DIAGRAM
      %% nodes
        \node (00) at (8   ,.7) {$s$};
        \node (02) at (8   ,2) {$\circ$};
        \node (12) at (9.25,2) {$\circ$};
        \node (32) at (11.25,2) {$\circ$};
        \node (42) at (12.75 ,2) {$\circ$};
        \node (43) at (12.75 ,3.2) {$x$};
        \node (33) at (11.25,3.2) {$\circ$};
        \node (13) at (9.25,3.2) {$\circ$};
        \node (03) at (8,3.2) {$\circ$};
      %% edges
        \draw[-] (00) to (8,1.05);
        \draw[dotted,-, semithick] (8,1.05) to (8,1.5);
        \draw[-latex] (8,1.4) to (02);
        \draw[-Circle] (02) to (12);
        \draw[-] (12) to (9.9,2);
        \draw[dotted,-, semithick] (9.9,2) to (10.5,2);
        \draw[-Circle] (10.5,2) to (32);
        \draw[-Circle] (32) to node[above]{\footnotesize{$a$}} (42);
        \draw[-latex] (42) to (43);
        \draw[dashed,-latex] (32) to (33);
        \draw[dashed,-Circle] (33) to (43);
        \draw[dashed,-Circle] (13) to (33);
        \draw[dashed,-Circle] (03) to (13);
        \draw[dashed,-latex] (12) to (13);
        \draw[dashed,-latex] (02) to (03);
      %% paths
        \draw[thick, red, rounded corners=1mm, cap=round]
              (8.2,.65) -- (8.2,1.8) -- node[below]{$\vec{w}$} (12.8,1.8);
        \draw[thick, blue, rounded corners=1mm, cap=round]
              (7.8,.65) -- (7.8,3.4) -- node[above]{$\vec{v}$} (12.8,3.4);
    \end{tikzpicture}
  \caption{From $\vec{w} \in W_s^{ur}$ and $\last(\vec{w}) \leq x$,
           we can construct $\vec{v} \in W_s^{ur}$ such that
           $\vec{w} \leq^{ur} \vec{v}$ and $\last(\vec{v}) = x$.}
  \label{fig:use-N3}
\end{figure}
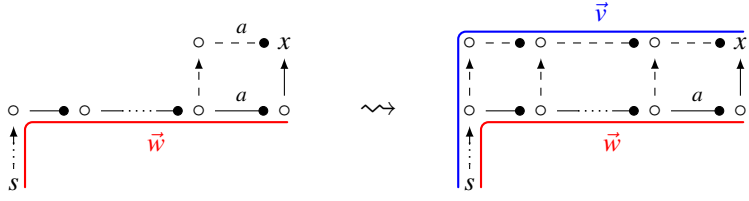
  By the induction hypothesis we have $\mo{M}_s^{ur}, \vec{v} \Vdash \psi$,
  and since $\vec{w} \leq^{ur} \vec{v}$ we must have $\mo{M}_s^{ur}, \vec{v} \Vdash \chi$.
  Using the induction hypothesis again gives $\mo{M}, x \Vdash \chi$.
  Since $x$ was any successor of $\last(\vec{w})$, it follows that
  $\mo{M}, \last(\vec{w}) \Vdash \psi \to \chi$.
  Conversely, if $\mo{M}, \last(\vec{w}) \Vdash \psi \to \chi$
  and $\vec{v} \in W_s^{ur}$ is such that $\vec{w} \leq^{ur} \vec{v}$
  and $\mo{M}_s^{ur}, \vec{v} \Vdash \psi$, then we have
  $\last(\vec{w}) \leq \last(\vec{v})$ and we can use induction to find
  $\mo{M}_s^{ur}, \vec{v} \Vdash \chi$. Therefore $\mo{M}_s^{ur}, \vec{w} \Vdash \psi \to \chi$.
  
  \medskip\noindent
  \textit{Induction step for $\phi = \Box\psi$.}
  Suppose $\mo{M}_s^{ur}, \vec{w} \Vdash \Box\psi$.
  Then there exists an intuitionistic neighbourhood $b \in N^{ur}$ such that
  $\vec{w} \in \dom(b)$ and such that $\vec{w} \leq^{ur} \vec{v}$
  implies $b(\vec{v}) \subseteq \llb \psi \rrb^{\mo{M}_s^{ur}}$.
  By construction of $\mo{M}_s^{ur}$, the neighbourhood $b$ must be of the form
  $a_{\vec{u}}$ for some $a \in N$ and $\vec{u} \in W_s^{ur}$ such that
  $\vec{u} \leq^{ur} \vec{w}$. We claim that $a$ witnesses $\mo{M}, w \Vdash \Box\psi$.
  To this end, let $x,y  \in W$ be such that $w \leq x$ and $y \in a(x)$.
  Then using repeated application of~\eqref{it:in-3} (see Figure~\ref{fig:use-N3})
  we can find some $\vec{v}$ such that $\vec{w} \leq \vec{v}$ and
  $\last(\vec{v}) = x$. By definition of $a_{\vec{u}}$, we then find
  $(\vec{v} : a : y) \in a_{\vec{u}}(\vec{v})$,
  so we must have $\mo{M}_s^{ur}, (\vec{v} : a : y) \Vdash \psi$.
  The induction hypothesis then implies $\mo{M}, y \Vdash \psi$.
  It follows that $\mo{M}, \last(\vec{w}) \Vdash \Box\psi$.
  
  Conversely, suppose $\mo{M}, \last(\vec{w}) \Vdash \Box\psi$.
  Then there exists an intuitionistic neighbourhood $a \in N$ with
  $\last(\vec{w}) \in \dom(a)$ such that for all $x \geq \last(\vec{w})$
  and all $y \in a(x)$ we have $\mo{M}, y \Vdash \psi$.
  Consider the intuitionistic neighbourhood $a_{\vec{w}}$ of $\mo{M}_s^{ur}$.
  Let $\vec{v}, \vec{u} \in W_s^{ur}$ be such that $\vec{w} \leq^{ur} \vec{v}$
  and $\vec{u} \in a_{\vec{w}}(\vec{v})$. Then $\vec{u}$ is of the form
  $\vec{v} : a : y$ and we have $y \in a(\last(\vec{v}))$ and $\last(\vec{w}) \leq \last(\vec{v})$. By assumption, this implies $\mo{M}, y \Vdash \psi$,
  so using the induction hypothesis we find $\mo{M}_s^{ur}, \vec{u} \Vdash \psi$,
  hence $\mo{M}_s^{ur}, \vec{w} \Vdash \Box\psi$.
  
  \medskip\noindent
  \textit{Induction step for $\phi = \Diamond\psi$.}
    Assume $\mo{M}_s^{ur}, \vec{w} \Vdash \Diamond\psi$.
    Let $x \geq \last(\vec{w})$ and $a \in N$ such that $x \in \dom(a)$.
    As above, we can find some $\vec{v} \in W_s^{ur}$ such that
    $\vec{w} \leq^{ur} \vec{v}$ and $\last(\vec{v}) = x$.
    By construction, we then have $\vec{v} \in \dom(a_{\vec{v}})$,
    so by assumption there exists a $\vec{u} \in a_{\vec{v}}(\vec{v})$
    such that $\mo{M}_s^{ur}, \vec{u} \Vdash \psi$.
    Unfolding the definitions shows that $\vec{u}$ must be of the form
    $\vec{v} : a : y$ for some $y \in W$ such that $y \in a(\last(\vec{v}))$.
    Therefore $y \in a(x)$, and by induction we have $\mo{M}, y \Vdash \psi$.
    We conclude that $\mo{M}, \last(\vec{w}) \Vdash \Diamond\psi$.
    
    Now suppose $\mo{M}, \last(\vec{w}) \Vdash \Diamond\psi$.
    Let $\vec{v} \in W_s^{ur}$ and $b \in N^{ur}$ be such that
    $\vec{w} \leq^{ur} \vec{v}$ and $\vec{v} \in \dom(b)$.
    Then $b$ is of the form $a_{\vec{u}}$ for some $\vec{u} \leq^{ur} \vec{v}$.
    By definition we have $\last(\vec{w}) \leq \last(\vec{v})$ and
    $\last(\vec{v}) \in \dom(a)$.
    The induction hypothesis then gives us some $y \in W$ such that
    $y \in a(\last(\vec{v}))$ and $\mo{M}, y \Vdash \psi$.
    But this means that $(\vec{v} : a : y) \in a_{\vec{u}}(\vec{v})$
    and $\mo{M}_s^{ur}, (\vec{v} : a : y) \Vdash \psi$.
    Therefore $\mo{M}_s^{ur}, \vec{w} \Vdash \Diamond\psi$.
\end{proof}

\begin{theorem}\label{thm:INF-vs-FOLM}
  We have $\Gamma \Vdash^{\inm} \phi$ iff $\Gamma \models \phi$.
\end{theorem}
\begin{proof}
  If $\Gamma \Vdash^{\inm} \phi$ then it follows from
  Proposition~\ref{prop:truth-str-to-inm} that $\Gamma \models \phi$.
  If $\Gamma \not\Vdash^{\inm} \phi$ then there exists an intuitionistic neighbourhood
  model $\mo{M} = (W, \leq, N, V)$ and a world $w \in W$ such that
  $\mo{M}, w \Vdash \Gamma$ and $\mo{M}, w \not\Vdash \phi$.
  By Theorem~\ref{thm:compl-coh} we may assume that $\mo{M}$ is coherent.
  Let $\mo{M}_w^{ur}$ be the unravelling of $\mo{M}$ from $w$,
  and write $\vec{w}$ for the unravelling path consisting just of $w$.
  Then Proposition~\ref{prop:unrav} tells us $\mo{M}_w^{ur}, \vec{w} \Vdash \Gamma$
  while $\mo{M}_w^{ur}, \vec{w} \not\Vdash \phi$.
  Now it follows from Corollary~\ref{cor:Cartesian} that
  $(\mo{M}_w^{ur})^{\circ}$ is a $\IFOM$-structure which shows that
  $\Gamma \not\models \phi$.
\end{proof}

%================================================================================
\subsection{Canonical model construction}\label{subsec:IM-complete}

  We construct a canonical model to derive strong completeness of $\IMfC$
  with respect to the class intuitionistic monotone models.
  This then proves that the calculus defining $\IMfC$ axiomatises $\IMf$.
  As usual in the intuitionistic setting, the canonical model is based on
  the set of prime theories ordered by inclusion.

\begin{definition}\label{def:prime}
  A \emph{prime theory} is
  a set $\Gamma \subseteq \Lbd$ such that,
  \begin{itemize}
    \item $\Gamma \vdash_{\IMfC} \phi$ implies $\phi \in \Delta$
          ($\Gamma$ is deductively closed);
    \item $\phi \vee \psi \in \Gamma$ implies $\phi \in \Gamma$ or $\psi \in \Gamma$
          ($\Gamma$ has disjunctive property);
    \item $\bot \notin \Gamma$ ($\Gamma$ is consistent).
  \end{itemize}
  A prime theory $\Gamma$ is called \emph{maximal} if it is not properly contained
  in any other prime theory.
\end{definition}

  We can prove a Lindenbaum lemma as usual, see for example~\cite[Lemma~11]{BezJon05}.

\begin{lemma}\label{lem:lindenbaum}
  Let $\Gamma \cup \{ \phi \} \subseteq \Lbd$ be a set of formulas such that
  $\Gamma \not\vdash \phi$.
  Then we can extend $\Gamma$ to a prime theory $\Gamma'$ such that
  $\Gamma \subseteq \Gamma'$ and $\phi \notin \Gamma'$.
\end{lemma}

  The intuitionistic Kripke frame underlying our canonical model is standard
  except for a small change: we unravel each maximal prime theory into
  a chain of copies. This is necessary to make the diamond-case of the
  truth lemma go through. To this end, the worlds of the canonical model
  are given by pairs $(\Gamma, n)$, where $\Gamma$ is a prime theory and
  $n$ is a natural number. If $\Gamma$ is not maximal then we only allow $n = 0$,
  but if $\Gamma$ is maximal then $n$ ranges over $\mathbb{N}$.
  The intuitionistic neighbourhoods on the resulting intuitionistic Kripke
  frame are designed such that the model satisfies and falsifies formulas
  of the form $\Box\phi$ and $\Diamond\phi$ at the right worlds.
  For formulas of the form $\Diamond\phi$ we use many different neighbourhoods.
  While we could do with fewer intuitionistic neighbourhoods, the current
  setup makes the proof of the truth lemma go through more smoothly.

\begin{definition}\label{def:canon}
  The domain of our canonical model is given by
  \begin{align*}
    W_{\IMfC} := \{ (\Gamma, 0) \mid \Gamma \text{ is a non-maximal prime theory} \}
    \cup \{ (\Gamma, n) \mid \Gamma \text{ is a maximal prime theory and } n \in \mb{N} \}.
  \end{align*}
  Order $W_{\IMfC}$ pointwise, that is, let $(\Gamma, n) \cleq (\Gamma', n')$
  if $\Gamma \subseteq \Gamma'$ and $n \leq n'$.
  This gives an intuitionistic Kripke frame $(W_{\IMfC}, \cleq)$.
  For each $\phi \in \Lbd$, we write $\tilde{\phi} := \{ (\Gamma, n) \in W_{\IMfC} \mid \phi \in \Gamma \}$ and define the neighbourhood $a_{\phi}$ by
  \begin{align*}
    a_{\phi}
      : W_{\IMfC} \rightharpoonup \fun{P}(W_{\IMfC})
      : (\Gamma, n) &\mapsto
        \begin{cases}
          \tilde{\phi} &\text{if } \Box\phi \in \Gamma \\
          \text{undefined} &\text{otherwise}
        \end{cases}
  \intertext{Furthermore, for each $\psi \in \Lbd$ and
  $(\Delta, m) \in W_{\IMfC}$ such that $\Diamond\psi \notin \Delta$
  and $\Box\top \in \Delta$, define}
    b_{\psi, \Delta, m}
      : W_{\IMfC} \rightharpoonup \fun{P}(W_{\IMfC})
      : (\Gamma, n) &\mapsto
        \begin{cases}
          W_{\IMfC} \setminus \tilde{\psi} &\text{if } (\Delta, m) = (\Gamma, n) \\
          W_{\IMfC} &\text{if } (\Delta, m) \cleq (\Gamma, n) \text{ and } (\Delta, m) \neq (\Gamma, n) \\
          \text{undefined} &\text{otherwise}
        \end{cases}
  \end{align*}
  Let $N_{\IMfC} = \{ a_{\phi} \mid \phi \in \Lbd \} \cup \{ b_{\psi, \Delta, m} \mid \psi \in \Lbd, \Delta \in \widetilde{\Box\top} \}$.
  Finally, define a valuation $V_{\IMfC}$ of the proposition letters by
  $V(p_i) = \tilde{p_i}$, and define
  $$
    \mo{M}_{\IMfC} := (W_{\IMfC}, \cleq, N_{\IMfC}, V_{\IMfC}).
  $$
\end{definition}

\begin{lemma}
  The structure $\mo{M}_{\IMfC}$ is an intuitionistic neighbourhood model.
\end{lemma}
\begin{proof}
  It is easy to see that $(W_{\IMfC}, \cleq)$ is an intuitionistic Kripke frame.
  By construction, the domain of each of the partial functions in $N_{\IMfC}$ is
  upward closed, so the $N_{\IMfC}$ consists of intuitionistic neighbourhoods.
  Besides, it follows immediately from the definition of $V_{\IMfC}$ that it
  maps proposition letters to upsets of $(W_{\IMfC}, \cleq)$.
  So $\mo{M}_{\IMfC}$ is an intuitionistic neighbourhood model.
\end{proof}

  We now show that the truth lemma goes through without problems.

\begin{lemma}[Truth lemma]\label{lem:truth}
  For all $(\Gamma, n) \in W_{\IMfC}$ and $\phi \in \Lbd$ we have
  $$
    \phi \in \Gamma \iff \mo{M}_{\IMfC}, (\Gamma, n) \Vdash \phi.
  $$
\end{lemma}
\begin{proof}
  We prove the lemma by induction on the structure of $\phi$.
  If $\phi = \bot$ or $\phi = p_i \in \Prop$ then the lemma is obvious.
  The induction steps for $\phi = \psi \wedge \chi$ and $\phi = \psi \vee \chi$
  follow from the fact that they are interpreted locally.
  The induction step for implication also proceeds as usual.
  
  \medskip\noindent
  \textit{Induction step for $\phi = \Box\psi$.}
  Suppose $\Box\psi \in \Gamma$. Then $\Gamma \in \dom(a_{\psi})$,
  and by induction we have that $(\Delta, m) \in \tilde{\psi}$
  implies $(\Delta, m) \Vdash \psi$.
  Since $a_{\psi}(\Gamma', k) = \tilde{\psi}$ whenever $(\Gamma, n) \cleq (\Gamma', k)$,
  we find that $(\Gamma, n) \Vdash \Box\psi$.
  
  Conversely, if $(\Gamma, n) \Vdash \Box\psi$ then there exists a neighbourhood
  $c \in N_{\IMfC}$ such that $c(\Gamma', n') \subseteq \llb \psi \rrb$ for all
  $(\Gamma', n')$ above $(\Gamma, n)$.
  If $c = a_{\chi}$ then $\Box\chi \in \Gamma$.
  Further, then we have $\tilde{\chi} \subseteq \llb \psi \rrb^{\mo{M}_{\IMfC}}$.
  By induction we then have $\llb \psi \rrb^{\mo{M}_{\IMfC}} = \tilde{\psi}$,
  so that $\tilde{\chi} \subseteq \tilde{\psi}$, which in turn
  implies that $\vdash_{\IMfC} \chi \to \psi$.
  Therefore, by monotonicity $\vdash \Box\chi \to \Box\psi$,
  hence $\Box\psi \in \Gamma$.
  If $c = b_{\chi, \Delta, m}$ then we must have $W_{\IMfC} \subseteq \tilde{\psi}$,
  so $\vdash_{\IMfC} \top \to \psi$ (because $b_{\chi, \Delta, m}$ becomes $W_{\IMfC}$ on
  successors of $(\Delta, m)$ and every element of our frame has a successor).
  Also by definition this means $\Box\top \in \Delta$ and $(\Delta, m) \cleq (\Gamma, n)$,
  so that $\top \to \psi$ implies $\Box\top \to \Box\psi$ hence $\Box\psi \in \Gamma$.
  
  \medskip\noindent
  \textit{Induction step for $\phi = \Diamond\psi$.}
  Assume $\Diamond\psi \in \Gamma$ and let $(\Gamma', n')$ be any successor
  of $(\Gamma, n)$.
  If $a_{\chi}$ is a neighbourhood of $(\Gamma', n')$, then by definition
  $\Box\chi \in \Gamma'$.
  This means $\Gamma' \not\vdash (\Box\chi \wedge \Diamond\psi) \to \bot$,
  so that it follows from ($\ANega$) that $\not\vdash (\chi \wedge \psi) \to \bot$.
  This implies that there exists a prime theory containing both
  $\chi$ and $\psi$, hence it is an element in $a_{\chi}(\Gamma', n') = \tilde{\chi}$
  that satisfies $\psi$.
  If $b_{\chi, \Delta, m}$ is a neighbourhood of $(\Gamma', n')$
  then either it maps to $W_{\IMfC}$, which clearly has a world in it satisfying $\chi$,
  or it maps to $W_{\IMfC} \setminus \tilde{\chi}$ and $(\Gamma', n') = (\Delta, m)$
  and $\Diamond\chi \notin \Delta$.
  If the latter does not have a world in it that satisfies $\psi$,
  then we have $\tilde{\psi} \subseteq \tilde{\chi}$.
  This implies $\vdash \psi \to \chi$, hence by monotonicity
  $\Diamond\psi \to \Diamond\chi$.
  By assumption $\Delta\psi \in \Gamma \subseteq \Gamma'$,
  but that means $\Diamond\chi \in \Gamma' = \Delta$, a contradiction.
  
  For the converse, assume $\Diamond\psi \notin \Gamma$.
  Then $\Gamma \cup \{ \Box\top \} \not\vdash \Diamond\psi$,
  for otherwise we would have $\Gamma \vdash \Box\top \to \Diamond\psi$,
  which by ($\AInt$) and ($\RMP$) implies $\Gamma \vdash \Diamond\psi$,
  hence by deductive closure of theories $\Diamond\psi \in \Gamma$.
  So we can extend $\Gamma$ to a prime theory $\Gamma'$ containing $\Box\top$
  but not $\Diamond\psi$, which gives an element $(\Gamma', n) \in W_{\IMfC}$
  such that $(\Gamma, n) \cleq (\Gamma', n)$.
  Now by definition $b_{\psi, \Gamma', n}$ is such that
  $b_{\psi, \Gamma', n}(\Gamma', n) = (W_{\IMfC} \setminus \tilde{\psi})$
  and this proves $(\Gamma, n) \not\Vdash \Diamond\psi$.
\end{proof}

\begin{theorem}\label{thm:IM-IMC}
  We have $\Gamma \vdash_{\IMfC} \phi$ if and only if
  $\Gamma \Vdash^{\inm} \phi$ if and only if $\Gamma \models \phi$ if and only if $\Gamma \vdash_{\IMf} \phi$.
\end{theorem}
\begin{proof}
  Theorem~\ref{thm:INF-vs-FOLM} states that $\Gamma \Vdash^{\inm} \phi$
  if and only if $\Gamma \models \phi$, and by Theorem~\ref{thm:sc-triv}
  we have $\Gamma \models \phi$ if and only if $\Gamma \vdash_{\IMf} \phi$,
  so we only have to prove the first ``iff.''
  If $\Gamma \vdash_{\IMfC} \phi$, then Theorem~\ref{thm:sound}
  gives $\Gamma \Vdash^{\inm} \phi$. We prove the converse by contraposition.
  Suppose $\Gamma \not\vdash_{\IMfC} \phi$.
  Then we can use Lemma~\ref{lem:lindenbaum} to find a prime theory
  $\Delta$ extending $\Gamma$ that does not contain $\phi$.
  Then we have $\mo{M}_{\IMfC}, (\Delta, 0) \not\Vdash \Gamma$
  and $\mo{M}_{\IMfC}, (\Delta, 0) \Vdash \phi$
  by Lemma~\ref{lem:truth}, so that $\Gamma \not\Vdash^{\inm} \phi$.
\end{proof}

%%%%%%%%%%%%%%%%%%%%%%%%%%%%%%%%%%%%%%%%%%%%%%%%%%%%%%%%%%%%%%%%%%%%%%%%%%%%%%%%%
\section{Relation to other logics}\label{sec:relation}

  In order to position $\IMf$ into the landscape of existing intuitionistic
  modal logics, we investigate the relation of $\IMf$ to various other
  intuitionistic monotone modal logics, as well as to its normal
  modal counterpart $\IK$.

%================================================================================
\subsection{Comparing intuitionistic monotone modal logics}\label{subsec:comparison}

  There are two intuitionistic counterparts of monotone modal logic that often
  appear in the literature: constructive monotone modal logic
  $\WM$~\cite{DalGreOli20,Dal22},
  and the extension of intuitionistic logic with a single monotone
  modality $\Mon$, which we denote by $\iM$.
  We write $\lan{L}_{\Mon}$ for the language over which $\iM$ is defined.
  A generalised Hilbert calculus for $\iM$ can be obtained by reading $\Mon$ as $\Box$
  and deleting ($\RMonD$) in Definition~\ref{def:WM}.

  In this section we consider the connection between $\iM$, $\WM$ and $\IMf$.
  First, we recall \emph{constructive neighbourhood semantics} for $\WM$ 
  and use this to show that $\WM$ proves strictly fewer consecutions than $\IMf$.
  We then show that $\IMf$ is sound and complete with respect to the class
  of so-called \emph{full} constructive neighbourhood models,
  thereby providing yet another frame semantics for $\IMf$.
  
  Next, we observe that there are two obvious ways to map the language of
  $\iM$ to $\Lbd$, namely by viewing $\Mon$ either as $\Box$ or as $\Diamond$.
  This gives rise to a comparison of $\iM$ with the box-free and
  diamond-free fragments of $\WM$ and $\IMf$.
  This is analogous to the study of $\Diamond$-free fragments of various
  intuitionistic normal modal logics, see~\cite[Section~9.2]{Gre99}
  and~\cite{DasMar23,GroShiClo24}.
  
\begin{definition}
  A \emph{constructive neighbourhood model} is a tuple
  $\mo{M} = (W, \preceq, \gamma, V)$
  consisting of a preordered set $(W, \preceq)$, a neighbourhood function
  $\gamma : W \to \pow\pow(W)$, and a valuation $V : \Prop \to \up(W, \preceq)$.
  The interpretation of an $\Lbd$-formula at a world $w \in W$
  is defined via the usual propositional clauses and the following modal clauses:
  \begin{align*}
    \mo{M}, w \Vdash \Box\phi
      &\iff \text{for all } w' \geq w
            \text{ there exists } a \in \gamma(w')
            \text{ such that for all } v \in a,
            \text{ we have } \mo{M}, v \Vdash \phi \\
    \mo{M}, w \Vdash \Diamond\phi
      &\iff \text{for all } w' \geq w
            \text{ and all } a \in \gamma(w')
            \text{ there exists } v \in a,
            \text{ such that } \mo{M}, v \Vdash \phi
  \end{align*}
\end{definition}

  It was proven in \cite[Section~4]{Dal22} that $\WM$ is sound and complete
  with respect to the class of constructive neighbourhood models,
  and the proof in fact entails strong completeness.
  The same class of models can be used to interpret the language $\lan{L}_{\Mon}$,
  interpreting $\Mon$ in the same way as $\Box$. This clearly provides a
  sound semantics, and a straightforward canonical model construction proves
  (strong) completeness as well. So we have:

\begin{theorem}\label{thm:compl-iM-WM}
  Both $\WM$ and $\iM$ are sound and strongly complete with respect to the class
  of constructive neighbourhood models.
\end{theorem}

  The next example shows that $(\Box\top \to \Diamond p) \to \Diamond p$ is
  not derivable in $\WM$.
  Since this is an axiom of $\IMf$, it shows that $\WM$ proves fewer formulas
  than $\IMf$.

\begin{example}
  Let $W = \{ w, v \}$ ordered by $w \preceq w \preceq v \preceq v$.
  Define $\gamma : W \to \fun{PP}W$ by $\gamma(w) = \{ \{ w \} \}$ and
  $\gamma(v) = \emptyset$.
  Define a valuation $V$ of the proposition letter $P$ by $V(p) = \{ v \}$.
  Then $\mo{M} = (W, \preceq, \gamma, V)$ is a constructive neighbourhood model.
  
  Both $w$ and $v$ do not satisfy $\Box\top$, because $\gamma(v) = \emptyset$.
  Therefore $w \Vdash \Box\top \to \Diamond p$. But $w \not\Vdash \Diamond p$,
  because $\{ w \} \in \gamma(w)$ and $w \not\Vdash p$.
  This shows that $w \not\Vdash (\Box\top \to \Diamond p) \to \Diamond p$,
  hence $\WM$ does not prove $(\Box\top \to \Diamond p) \to \Diamond p$.
\end{example}

\begin{remark}
  Various other semantics were used for $\iM$.
  For example, in~\cite[Section~6]{DalGreOli20} the extension of intuitionistic
  logic with a monotone (box) modality was interpreted in 
  constructive neighbourhood models $(W, \preceq, \gamma, V)$
  such that (i) $\preceq$ is a partial order, (ii) $\gamma(w)$ is upwards closed under
  inclusion, and (iii) $w \preceq w'$ implies $\gamma(w) \subseteq \gamma(w')$.
  The interpretation of the modality can then be simplified to
  $w \Vdash \Mon\phi$ iff $\llb \phi \rrb \in \gamma(w)$.
 
  If moreover we allow only upsets as neighbourhoods, i.e.~we view $\gamma$ as a
  map from $W$ to $\fun{P}(\up(W, \preceq))$, then we obtain the semantics
  for $\iM$ used by Goldblatt~\cite[Section~4]{Gol81},
  and later in~\cite[Section~8]{GroPat20} and~\cite[Sections~4.1 and~4.2]{Gro22gt}.
\end{remark}

  Since $\WM$ is weaker than $\IMf$, we can try to use (certain)
  constructive neighbourhood models as semantics for $\IMf$.
  To this end, first we note that a constructive neighbourhood frame
  $(W, \preceq, \gamma, V)$ satisfies $(\Box\top \to \Diamond\phi) \to \Diamond\phi$ if it
  satisfies: if $\gamma(w) \neq \emptyset$ and $w \leq w'$ then $\gamma(w') \neq \emptyset$.
  Let us call such constructive neighbourhood models \emph{full}.
  In passing, we note: 

\begin{proposition}\label{prop:compl-iM}
  $\iM$ is sound and strongly complete with respect to
  the class of full constructive neighbourhood models.
\end{proposition}
\begin{proof}
  Soundness follows from Theorem~\ref{thm:compl-iM-WM}.
  For completeness, suppose $\Gamma \not\vdash_{\iM} \phi$.
  Then by Theorem~\ref{thm:compl-iM-WM} there exists a constructive
  neighbourhood model $(W, \preceq, \gamma, V)$ falsifying the consecution.
  Defining $\gamma'$ by
  \begin{equation*}
    \gamma'(w)
      = \begin{cases}
          \gamma(w)
            &\text{if } \gamma(v) \neq \emptyset \text{ for all } v \in W 
                                                 \text{ such that } w \leq v \\
          \emptyset
            &\text{otherwise}
        \end{cases}
  \end{equation*}
  yields a full constructive neighbourhood model $(W, \preceq, \gamma', V)$
  that leaves the interpretation of $\lan{L}_{\Mon}$-formulas unchanged,
  hence gives a full model falsifying the consecution.
\end{proof}

  We can transform intuitionistic neighbourhood models into full
  constructive neighbourhood models as follows.

\begin{definition}\label{def:inm-cnm}
  Let $\mo{M} = (W, \leq, N, V)$ be an intuitionistic neighbourhood frame.
  For each $w \in W$, let $N_w = \{ a \in N \mid w \in \dom(a) \}$
  and ${\uparrow}w = \{ v \in W \mid w \leq v \}$, and define
  \begin{equation*}
    \ms{C}(w) := \big\{ \{ a(f(a)) \mid a \in N_w \} \mid f : N_w \to {\uparrow}w \text{ is a function} \big\}.
  \end{equation*}
  Let $\widehat{W} = \{ (w, \Sigma) \mid \Sigma \in \ms{C}(w) \}$
  and define $(w, \Sigma) \preceq (w', \Sigma')$ iff $w \leq w'$.
  Then $(\widehat{W}, \preceq)$ is a pre-ordered set.
  Define the neighbourhood function $\gamma$ by $\gamma(w, \Sigma) := \Sigma$,
  and the valuation $\widehat{V}$ by $\widehat{V}(p_i) = \{ (w, \Sigma) \in \widehat{W} \mid w \in V(p) \}$.
  Then
  \begin{equation*}
    \widehat{\mo{M}} := (\widehat{W}, \preceq, \gamma, \widehat{V})
  \end{equation*}
  is a constructive neighbourhood model.
\end{definition}

  By construction, $\widehat{\mo{M}}$ is a full constructive neighbourhood
  model. Moreover, we have:

\begin{lemma}\label{lem:inm-cnm}
  Let $\mo{M} = (W, \leq, N, V)$ be an intuitionistic neighbourhood model
  and $\widehat{\mo{M}} = (\widehat{W}, \preceq, \gamma, \widehat{V})$ its
  corresponding constructive neighbourhood model.
  Then for all $(w, \Sigma) \in \widehat{W}$ and $\phi \in \Lbd$,
  \begin{equation*}
    \widehat{\mo{M}}, (w, \Sigma) \Vdash \phi
      \iff \mo{M}, w \Vdash \phi.
  \end{equation*}
\end{lemma}
\begin{proof}
  We use induction on the structure of $\phi$.
  The cases $\phi = \bot$ and $\phi = p_i \in \Prop$ hold by definition.
  The inductive cases for intuitionistic connectives are routine.
  
  \medskip\noindent
  \textit{Inductive step for $\phi = \Box\psi$.}
    If $\mo{M}, w \Vdash \Box\psi$ then there exists $a \in N$ such that
    $a(w') \subseteq \llb \psi \rrb^{\mo{M}}$ for all $w' \geq w$.
    For each $(w', \Sigma') \in \widehat{W}$ such that
    $(w, \Sigma) \preceq (w', \Sigma')$, we have $w \leq w'$
    hence $a \in N_{w'}$, so that $a(v) \in \Sigma' =: \gamma(w', \Sigma')$ for some
    $v \geq w'$. Then $w \leq v$, so by induction
    $a(v) \subseteq \llb \psi \rrb^{\widehat{\mo{M}}}$.
    This proves $\widehat{\mo{M}}, (w, \Sigma) \Vdash \Box\psi$.
    
    Conversely, if $\mo{M}, w \not\Vdash \Box\psi$ then for each
    $a \in N_w$ there exists some $v_a \geq w$ such that
    $a(v_a) \not\subseteq \llb \psi \rrb^{\mo{M}}$.
    But then $\Delta := \{ a(v_a) \mid a \in N_w \} \in \ms{C}(w)$,
    so $(w, \Sigma) \preceq (w, \Delta)$ and (by induction) there is no neighbourhood in
    $\gamma(w, \Delta) := \Delta$ that is contained in $\llb \psi \rrb^{\widehat{\mo{M}}}$.
    Therefore $\widehat{\mo{M}}, (w, \Sigma) \not\Vdash \Box\psi$.

  \medskip\noindent
  \textit{Inductive step for $\phi = \Diamond\psi$.}
    This follows immediately from the construction.
\end{proof}

\begin{theorem}\label{thm:IMf-cnm}
  The logic $\IMf$ is sound and strongly complete with respect to the class
  of full constructive neighbourhood models.
\end{theorem}
\begin{proof}
  For soundness, we only have to verify that full constructive neighbourhood
  models satisfy ($\AInt$). To this end, let
  $\mo{M} = (W, \preceq, \gamma, V)$ be a constructive neighbourhood model
  and suppose $w \in W$ satisfies $\Box\top \to \Diamond\phi$.
  To see that $w \Vdash \Diamond\phi$, let $w' \geq w$ and $a \in \gamma(w')$.
  Then $\gamma(w') \neq \emptyset$, hence by fullness $v \geq w'$ implies
  $\gamma(v) \neq \emptyset$, which in turn implies $w' \Vdash \Box\top$.
  Since $w \Vdash \Box\top \to \Diamond\phi$ and $w \leq w'$,
  we find $w' \Vdash \Diamond\phi$, so by definition there must be some
  $u \in a$ such that $u \Vdash \phi$, as desired.
  
  For completeness, suppose $\Gamma \not\vdash_{\IMf} \phi$.
  Then by Theorem~\ref{thm:IM-IMC} there exists an intuitionistic
  neighbourhood model falsifying the consecution, so that
  Definition~\ref{def:inm-cnm} and Lemma~\ref{lem:inm-cnm}
  yield a full constructive neighbourhood model falsifying it.
\end{proof}

  We can use this theorem to investigate the embedding of
  $\lan{L}_{\Mon}$ into $\Lbd$ that sends $\Mon$ to $\Box$.
  For $\phi \in \lan{L}_{\Mon}$, write $\phi^{\Box}$ for the formula in $\Lbd$
  obtained by replacing each occurrence of $\Mon$ with $\Box$.
  Extend this to $\Gamma \subseteq \lan{L}_{\Mon}$ by defining
  $\Gamma^{\Box} = \{ \phi^{\Box} \mid \phi \in \Gamma \}$.
  Note that every diamond-free consecution of formulas in $\Lbd$ can
  be viewed as a consecution of the form $\Gamma^{\Box} \vdash \phi^{\Box}$.
  Therefore the following theorem can be read as saying that the logics
  $\iM$, $\WM$ and $\IMf$ prove precisely the same diamond-free consecutions.

\begin{theorem}\label{thm:diamond-free}
  Let $\Gamma \cup \{ \phi \} \subseteq \lan{L}_{\Mon}$ be a collection
  of formulas. Then
  \begin{equation*}
    \Gamma \vdash_{\iM} \phi
      \iff \Gamma_{\Box} \vdash_{\WM} \phi^{\Box}
      \iff \Gamma_{\Box} \vdash_{\IMf} \phi^{\Box}.
  \end{equation*}
\end{theorem}
\begin{proof}
  Any proof $\delta$ in the generalised Hilbert calculus for $\iM$
  for $\Gamma\vdash_{\iM} \phi$ gives rise to a proof for
  $\Gamma^{\Box} \vdash_{\WM} \phi^{\Box}$
  and for $\Gamma^{\Box} \vdash_{\IMf} \phi^{\Box}$
  by simply replacing the use of $(\RMonM)$ by $(\RMonB)$.
  Conversely,
  if $\Gamma \not\vdash_{\iM} \phi$ then by Proposition~\ref{prop:compl-iM}
  there exist a full constructive neighbourhood model $\mo{M} = (W, \preceq, \gamma, V)$
  falsifying the consecution, so
  Theorems~\ref{thm:compl-iM-WM} and~\ref{thm:IMf-cnm} entail
  $\Gamma \not\vdash_{\WM} \phi$ and $\Gamma \not\vdash_{\IMf} \phi$.
\end{proof}

  We complete this section by investigating the relation between $\iM$ and
  the box-free fragments of $\WM$ and $\IMf$.
  For $\phi \in \lan{L}_{\Mon}$ and $\Gamma \subseteq \lan{L}_{\Mon}$,
  define $\phi^{\Diamond}$ and $\Gamma^{\Diamond}$ as expected.
  Then the next example and lemma show that there are consecutions $\Gamma \vdash \phi$
  in $\lan{L}_{\Mon}$ that are not derivable in $\iM$, but whose
  translation is derivable in $\IMf$.
  
\begin{example}\label{exm:iM-IM}
  Let $W = \{ w, v \}$ be a two-element set, and order it
  by $w \preceq w \preceq v \preceq v$. Define the neighbourhood function
  $\gamma : W \to \fun{PP}(W)$ by $\gamma(w) = \emptyset$ and $\gamma(v) = \fun{P}(W)$,
  and let $V$ be any valuation.
  Then $\mo{M} = (W, \preceq, \gamma, V)$ is a constructive neighbourhood model and we have
  $\mo{M}, w \not\Vdash \Mon\top$ because $\gamma(w)$ contains no neighbourhoods.
  Moreover, we have $\mo{M}, v \Vdash \Mon\bot$ because
  $\llb \bot \rrb^{\mo{M}} = \emptyset \in \gamma(v)$.
  The latter implies $\mo{M}, w \not\Vdash \neg\Mon\bot$ and
  $\mo{M}, v \not\Vdash \neg\Mon\bot$, and hence we trivially find
  $\mo{M}, w \Vdash \neg\Mon\bot \to \Mon\top$.
  Ultimately, this gives
  \begin{equation*}
    \mo{M}, w \not\Vdash (\neg\Mon\bot \to \Mon\top) \to \Mon\top.
  \end{equation*}
  It now follows from Theorem~\ref{thm:compl-iM-WM} that
  $\not\vdash_{\iM} (\neg\Mon\bot \to \Mon\top) \to \Mon\top$.
  On the other hand, it is not hard to see that $\IMf$ proves
  $(\neg\Diamond\bot \to \Diamond\top) \to \Diamond\top$.
  So we obtain the following theorem.
\end{example}

\begin{theorem}
  Let $\Gamma \cup \{ \phi \} \subseteq \lan{L}_{\Mon}$. Then
  \begin{alignat*}{2}
    \Gamma \vdash_{\iM} \phi
      &\quad\iff
      &&\Gamma^{\Diamond} \vdash_{\WM} \phi^{\Diamond}, \\
    \Gamma \vdash_{\iM} \phi
      &\quad\text{implies}\quad
      &&\Gamma^{\Diamond} \vdash_{\IMf} \phi^{\Diamond},
  \end{alignat*}
  and the implication is strict.
\end{theorem}
\begin{proof}
  Due to the symmetry of the proof system for $\WM$, we have
  $\Gamma^{\Diamond} \vdash_{\WM} \phi^{\Diamond}$ if and only if
  $\Gamma^{\Box} \vdash_{\WM} \phi^{\Box}$.
  Together with Theorem~\ref{thm:diamond-free}, this entails the ``iff.''
  The implication follows from the fact that the proof system for
  $\IMf$ extends that for $\WM$, and the strictness is
  witnessed by Example~\ref{exm:iM-IM}.
\end{proof}

%================================================================================
\subsection{Translating $\IMf$ into multimodal $\IK$}

  The way we defined $\IMf$ via first-order logic is closely aligned to the
  procedure used by Simpson to obtain $\IK$ via a suitable
  first-order logic~\cite{Sim94}.
  As a consequence, $\IMf$ and $\IK$ are strongly related:
  intuitionistic neighbourhood frames give rise to frames
  for a multimodal version $\IK_2$ of $\IK$, and we can embed $\IMf$ into $\IK_2$.
  We prove that this embedding preserves and reflects derivability,
  in the sense that an $\Lbd$-consecution $\Gamma \vdash \phi$ is in $\IMf$
  if and only if its translation is in $\IK_2$.
  This mirrors the relation between the classical counterparts $\M$ and $\K$
  of $\IMf$ and $\IK$, investigated in~\cite[Section~5.2]{Han03} and in more
  generality in~\cite{GasHer96,KraWol99}.

  Let $\lan{L}_2$ be the language generated by the grammar
  \begin{equation*}
    \phi ::= p_i \mid \bot \mid \phi \wedge \phi \mid \phi \vee \phi \mid \phi \to \phi
      \mid \Box_N\phi \mid \Diamond_N\phi \mid \Box_{\ni}\phi \mid \Diamond_{\ni}\phi,
  \end{equation*}
  where $p_i$ ranges over $\Prop$.
  Let $\IK_2$ be the multimodal counterpart of $\IK$ over this language
  (see Appendix~\ref{app:IK} for details).
  Then we can translate $\Lbd$ into $\lan{L}_2$ as follows.

\begin{definition}\label{def:trans}
  The modal translation $(-)^t : \Lbd \to \lan{L}_2$ is defined recursively by
  \begin{align*}
    p^t &= p
      &\text{(for $p \in \Prop \cup \{ \bot \}$)} \\
    (\phi \star \psi)^t &= \phi^t \star \psi^t
      &\text{(for $\star \in \{ \wedge, \vee, \to \}$)} \\
    (\Box\phi)^t &= \DiamondN\BoxI \phi^t \\
    (\Diamond\phi)^t &= \BoxN\DiamondI \phi^t
  \end{align*}
\end{definition}

  The idea behind this translation is best understood semantically.
  Roughly speaking, the two pairs of modalities in $\IK_2$ are interpreted
  using relations $R_N$ and $R_{\ni}$, and these mimic the interpretations
  of the predicates $\N$ and $\E$ from the logic $\IFOM$.
  We make this intuition precise in Definition~\ref{def:IM-IK2-sem} and the
  subsequent lemma.
  First, we show that if $\Gamma \vdash \phi$ is derivable in $\IMf$,
  then its translation $\Gamma^t \vdash \phi^t$,
  where $\Gamma^t = \{ \psi^t \mid \psi \in \Gamma \}$,
  is derivable in $\IK_2$.

\begin{lemma}\label{lem:IKJ-mon}
  For each $j \in \{ N, \ni \}$, the rules
  \begin{equation*}
    \dfrac{\emptyset \vdash \phi \wedge \psi \to \bot}
          {\Gamma \vdash \Box_j\phi \wedge \Diamond_j\psi \to \bot} \; (\RStr_{j})
    \qquad
    \dfrac{\emptyset \vdash \phi \to \psi}
          {\Gamma \vdash \Box_j\phi \to \Box_j\psi} \; (\RMonB_{,j})
    \qquad\text{and}\qquad
    \dfrac{\emptyset \vdash \phi \to \psi}
          {\Gamma \vdash \Diamond_j\phi \to \Diamond_j\psi} \; (\RMonD_{,j}).
  \end{equation*}
  are derivable in $\IK_2$.
\end{lemma}
\begin{proof}
  Suppose $\emptyset \vdash \phi \wedge \psi \to \bot$.
  Then $\vdash \phi \to (\psi \to \bot)$.
  Necessitation gives $\Gamma \vdash \Box_j(\phi \to (\psi \to \bot))$,
  and using the axiom $(\AKbox)$ and modus ponens we find
  $\Gamma \vdash \Box_j\phi \to \Box_j(\psi \to \bot)$.
  Now, using ($\AKdia)$, modus ponens and intuitionistic propositional reasoning
  yields $\Gamma \vdash \Box_j\phi \to (\Diamond_j\psi \to \Diamond_j\bot)$.
  The axiom $(\ANdia)$ and propositional reasoning then yields
  $\Gamma \vdash \Box_j\phi \to (\Diamond_j\psi \to \bot)$
  so that currying gives the required conclusion of $(\RStr_j)$.
  
  For the monotonicity rules, combine necessitation, modus ponens,
  ($\AKbox$) and ($\AKdia)$.
\end{proof}

\begin{lemma}\label{lem:trans-ax}
  The axioms $(\ANega^t)$ and $(\AInt^t)$ are derivable in $\IK_2$.
\end{lemma}
\begin{proof}
  The former can be proven using $(\RStr_{\ni})$ and $(\RStr_N)$, which we have in $\IK_2$ by
  Lemma~\ref{lem:IKJ-mon}.
  For the second, instantiating $(\ADualI)$ for $\Box_N$ and $\Diamond_N$ with
  $p = \top$ and $q = \Diamond_{\ni}\psi$ gives
  \begin{equation*}
    (\Diamond_N \top \to \Box_N\Diamond_{\ni}\psi) \to \Box_N(\top \to \Diamond_{\ni}\psi).
  \end{equation*}
  Now we can use $\Box_{\ni}\top \leftrightarrow \top$
  and $(\top \to \Diamond_{\ni}\psi) \leftrightarrow \Diamond_{\ni}\psi$
  to transform this to
  $(\Diamond_N\Box_{\ni}\top \to \Box_N\Diamond_{\ni}\psi) \to \Box_N\Diamond_{\ni}\psi$,
  which is the translation of $(\AInt)$.
\end{proof}

\begin{proposition}\label{prop:IM-IK2}
  Let $\Gamma \cup \{ \phi \} \subseteq \Lbd$ be a set of formulas.
  Then $\Gamma \vdash_{\IMf} \phi$ implies $\Gamma^t \vdash_{\IK_2} \phi^t$.
\end{proposition}
\begin{proof}
  We use induction on the structure of a proof $\delta$ for $\Gamma \vdash_{\IMf} \phi$.
  If the last rule is ($\REl$) then $\phi \in \Gamma$,
  so $\phi^t \in \Gamma^t$ hence $\Gamma^t \vdash_{\IK_2} \phi^t$ using ($\REl$).
  If the last rule in $\delta$ is ($\RAx$) then $\phi$ is an instance of
  ($\ANega$) or ($\AInt$). By Lemma~\ref{lem:trans-ax} their translations
  are derivable in $\IK_2$, so we can use ($\RAx$) to find
  $\Gamma^t \vdash_{\IK_2} \phi^t$.
  
  If the last rule in $\delta$ is modus ponens, then we have
  $\Gamma \vdash_{\IMf} \psi$ and $\Gamma \vdash_{\IMf} \psi \to \phi$ for some
  $\psi \in \Lbd$. By induction, $\Gamma^t \vdash_{\IK_2} \psi^t$
  and $\Gamma^t \vdash_{\IK_2} \vdash (\psi \to \phi)^t$.
  By definition $(\psi \to \phi)^t = \psi^t \to \phi^t$, so we can use
  modus ponens to derive $\Gamma^t \vdash_{\IK_2} \phi^t$.
  Finally, if the last rule in $\delta$ is $(\RMonB)$ then we can replace this
  by $(\RMonB_{\ni})$ and $(\RMonD_N)$ (in that order),
  which are derivable in $\IK_2$ by Lemma~\ref{lem:IKJ-mon}.
  The case for $(\RMonD)$ can be dealt with similarly.
\end{proof}

  We use a semantic argument to prove the converse of Proposition~\ref{prop:IM-IK2}.
  This makes use of the following transformation of an intuitionistic neighbourhood
  model into an $\IK_2$-model.

\begin{definition}\label{def:IM-IK2-sem}
  Let $\mo{M} = (W, \leq, N, V)$ be an intuitionistic neighbourhood model.
  Let $N(W) = \{ (a, w) \mid a \in N, w \in \dom(a) \}$ and order $N(W)$ by
  $(a, w) \sqsubset (b, v)$ iff ($a = b$ and $w \leq v$).
  Define $W^* = W \cup N(W)$ and let ${\leq^*}$ be the union of $\leq$ and
  $\sqsubset$. Then $(W^*, \leq^*)$ is the disjoint union of
  $(W, \leq)$ and $(N(W), \sqsubset)$, hence a partially ordered set.
  Define binary relations $R_N$ and $R_{\ni}$ on $W^*$ by
  \begin{equation*}
    wR_N(a, v) \iff w = v, \qquad
    (a, w) R_{\ni}v \iff v \in a(w)
  \end{equation*}
  Write $\mo{M}^* = (W^*, \leq^*, R_N, R_{\ni}, V)$.
\end{definition}

\begin{lemma}
  Let $\mo{M} = (W, \leq, N, V)$ be a coherent intuitionistic neighbourhood
  model. Then $\mo{M}^*$ is an $\IK_2$-model.
\end{lemma}
\begin{proof}
  We need to verify the following frame conditions:
  
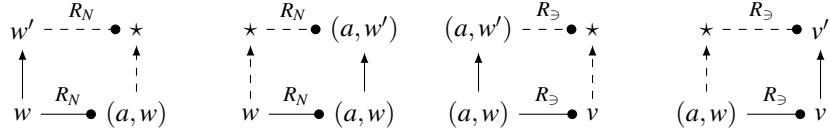
\begin{equation*}
        \begin{tikzpicture}[scale=.75]
      %% nodes
        \node (w)   at (0,0)   {$w$};
        \node (aw)  at (2,0)   {$(a, w)$};
        \node (wp)  at (0,1.5) {$w'$};
        \node (awp) at (2,1.5) {$\star$};
      %% edges
        \draw[-latex] (w) to (wp);
        \draw[-Circle] (w) to node[above]{\footnotesize{$R_N$}} (aw);
        \draw[dashed, -latex] (aw) to (awp);
        \draw[dashed, -Circle] (wp) to node[above]{\footnotesize{$R_N$}} (awp);
      %% nodes
        \node (w)   at (4,0)   {$w$};
        \node (aw)  at (6,0)   {$(a, w)$};
        \node (wp)  at (4,1.5) {$\star$};
        \node (awp) at (6,1.5) {$(a, w')$};
      %% edges
        \draw[dashed, -latex] (w) to (wp);
        \draw[-Circle] (w) to node[above]{\footnotesize{$R_N$}} (aw);
        \draw[-latex] (aw) to (awp);
        \draw[dashed, -Circle] (wp) to node[above]{\footnotesize{$R_N$}} (awp);
      %% nodes
        \node (aw)  at ( 8,0)   {$(a, w)$};
        \node (v)   at (10,0)   {$v$};
        \node (awp) at ( 8,1.5) {$(a, w')$};
        \node (vp)  at (10,1.5) {$\star$};
      %% edges
        \draw[-latex] (aw) to (awp);
        \draw[-Circle] (aw) to node[above]{\footnotesize{$R_{\ni}$}} (v);
        \draw[-latex, dashed] (v) to (vp);
        \draw[-Circle, dashed] (awp) to node[above]{\footnotesize{$R_{\ni}$}} (vp);
      %% nodes
        \node (aw)  at (12,0)   {$(a, w)$};
        \node (v)   at (14,0)   {$v$};
        \node (awp) at (12,1.5) {$\star$};
        \node (vp)  at (14,1.5) {$v'$};
      %% edges
        \draw[-latex, dashed] (aw) to (awp);
        \draw[-Circle] (aw) to node[above]{\footnotesize{$R_{\ni}$}} (v);
        \draw[-latex] (v) to (vp);
        \draw[-Circle, dashed] (awp) to node[above]{\footnotesize{$R_{\ni}$}} (vp);
    \end{tikzpicture}
  \end{equation*}
  That is, we need to find a suitable element for $\star$ making the dashed lines true.
  In the first two diagrams we can simply use $(a, w')$ and $w'$.
  In the third diagram, we have $v \in a(w)$ and $w \leq w'$, so we can use
  \eqref{it:in-2} to find some $v' \in a(w')$ such that $v \leq v'$.
  Since $v' \in a(w')$ we have $(a, w') R_{\ni} v'$, so $v'$ completes the
  diagram. We can use~\eqref{it:in-3} in a similar way to complete the last diagram.
\end{proof}
  
  This transformation of intuitionistic neighbourhood models to
  $\IK_2$-models works well with the translation of $\Lbd$-formulas.

\begin{proposition}\label{prop:IM-IK2-sem}
  Let $\mo{M} = (W, \leq, N, V)$ be a coherent intuitionistic neighbourhood model.
  Then for all $\phi \in \Lbd$ and $w \in W$,
  $$
    \mo{M}, w \Vdash \phi \iff \mo{M}^*, w \Vdash \phi^t.
  $$
\end{proposition}
\begin{proof}
  The proof proceeds by induction on the structure of $\phi$.
  The propositional cases are routine. The induction steps for
  $\phi = \Box\psi$ and $\phi = \Diamond\psi$ essentially hold by
  design of $\mo{M}^*$. We showcase the former.

  Suppose $\phi = \Box\psi$ and $\mo{M}, w \Vdash \Box\psi$.
  Then there exists a neighbourhood $a$ such that $w \in \dom(a)$ and
  for all $w' \geq w$ we have that $v \in a(w')$ implies $\mo{M}, v \Vdash \psi$.
  Then by induction we have $\mo{M}^*, v \Vdash \psi^t$ for all $v \in a(w')$ for
  some $w' \geq w$. By definition of $\sqsubset$ we have
  $(a, w) \sqsubset (a, w')$ if and only if $w \leq w'$.
  Therefore we find $\mo{M}^*, (a, w) \Vdash \Box_{\ni}\psi^t$.
  Since $w R_N (a, w)$ we find $\mo{M}^*, w \Vdash \Diamond_N\Box_{\ni}\psi^t$.
  
  For the converse, suppose $\mo{M}^*, w \Vdash \Diamond_N\Box_{\ni}\psi^t$.
  Then there exists some $a \in N$ such that $w R_N (a, w)$ and
  such that $(a, w) \leq^* (a, w') R_{\ni} v$ implies $\mo{M}^*, v \Vdash \psi^t$.
  Unravelling the definitions and using the induction hypothesis
  shows that $a \in N$ witnesses $\mo{M}, w \Vdash \Box\psi$.
\end{proof}

  We now have all ingredients to prove that the translation from
  Definition~\ref{def:trans} both preserves and reflects derivability
  of consecutions.

\begin{theorem}\label{thm:trans}
  Let $\Gamma \vdash \phi$ be a $\Lbd$-consecution.
  Then we have
  \begin{equation*}
    \Gamma \vdash_{\IMf} \phi \iff \Gamma^t \vdash_{\IK_2} \phi^t.
  \end{equation*}
\end{theorem}
\begin{proof}
  The direction from left to right is Proposition~\ref{prop:IM-IK2}.
  For the converse we reason by contraposition.
  Suppose $\Gamma \not\vdash_{\IMf} \phi$.
  Then by Theorems~\ref{thm:compl-coh} and~\ref{thm:IM-IMC} there exists a
  coherent intuitionistic
  neighbourhood model $\mo{M}$ and a world $w$ such that $\mo{M}, w \Vdash \Gamma$
  while $\mo{M}, w \not\Vdash \phi$.
  Using Proposition~\ref{prop:IM-IK2-sem} we find that $\mo{M}^*$ is
  an intuitionistic $\IK_2$-model such that $\mo{M}^*, w \Vdash \Gamma^t$
  and $\mo{M}^*, w \not\Vdash \phi^t$.
  Finally, completeness of $\IK_2$ with respect to $\IK_2$-frames
  (Theorem~\ref{thm:ik-compl}) yields $\Gamma^t \not\vdash_{\IK_2} \phi^t$.
\end{proof}

  As an example, we use Theorem~\ref{thm:trans} to prove that
  the formula $(\Box\bot \to \Diamond\top) \to \Diamond\top$ is not
  derivable in $\IMf$.

\begin{example}
  We show that the formula $(\Box\bot \to \Diamond\top) \to \Diamond\top$
  is not derivable in $\IMf$ by showing that its translation into bimodal formula
  is not in $\IK_2$.
  The bimodal translation of $(\Box\bot \to \Diamond\top) \to \Diamond\top$ is
  \begin{equation}\label{eq:exm-trans}
    (\DiamondN\BoxI\bot \to \BoxN\DiamondI\top) \to \BoxN\DiamondI\top.
  \end{equation}
  Consider the following $\IK_2$-frame, where $\leq$ is given by the reflexive
  closure of the arrows, and the relations $R_N$ and $R_{\in}$ are as indicated:
  
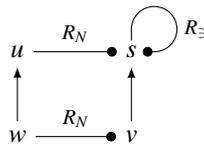
\begin{equation*}
        \begin{tikzpicture}[scale=.75]
      %% nodes
        \node (w) at (0,0)   {$w$};
        \node (v) at (2,0)   {$v$};
        \node (u) at (0,1.5) {$u$};
        \node (s) at (2,1.5) {$s$};
      %% edges
        \draw[-latex] (w) to (u);
        \draw[-Circle] (w) to node[above]{\footnotesize{$R_N$}} (v);
        \draw[-latex] (v) to (s);
        \draw[-Circle] (u) to node[above]{\footnotesize{$R_N$}} (s);
        \draw[-Circle] (2,1.7) -- (2,1.9) arc(180:-90:.4) -- (2.2,1.5);
        \node at (3.15,1.9) {\footnotesize{$R_{\ni}$}};
    \end{tikzpicture}
  \end{equation*}
  Let $V$ be any valuation.
  Then $v \not\Vdash \BoxI\bot$ because $v \leq s R_{\ni} s$,
  and $s \not\Vdash \bot$. Since the only world accessible from $w$ via $R_N$
  is $v$, this implies $w \not\Vdash \DiamondN\BoxI\bot$.
  Furthermore, we have $v \Vdash \BoxN\DiamondI\top$ because $s$ is the only
  world accessible from $v$ via $({\leq} \circ R_N)$ and $s \Vdash \DiamondI\top$.
  Therefore we find $w \Vdash \DiamondN\BoxI\bot \to \BoxN\DiamondI\top$.
  Finally, $w \not\Vdash \BoxN\DiamondI\top$ because $w R_N v$ and $v \not\Vdash \Diamond\top$.
  So this model falsifies the formula from~\eqref{eq:exm-trans},
  and hence $\IMf$ does not validate $(\Box\bot \to \Diamond\top) \to \Diamond\top$.
\end{example}

%%%%%%%%%%%%%%%%%%%%%%%%%%%%%%%%%%%%%%%%%%%%%%%%%%%%%%%%%%%%%%%%%%%%%%%%%%%%%%%%%
\section{Conclusion}\label{sec:conc}

  We introduced an intuitionistic monotone modal logic motivated by intuitionistic
  first-order logic. Specifically, we observed how classical monotone modal logic
  can be translated into a suitable first-order logic $\FOM$, changed
  this first-order logic to an intuitionistic one of the same signature,
  and then used the same translation to define the intuitionistic monotone modal
  logic $\IMf$.
  This mirrors the way the intuitionistic modal logic $\IK$ can be obtained.
  
  Guided by the interpreting structures of $\IFOM$, we provided various types
  of semantics for $\IMf$. This led to the concept of an
  \emph{intuitionistic neighbourhood}: a counterpart of a classical neighbourhood
  that may change when moving along the intuitionistic accessibility relation.
  We then compared our logic to various other intuitionistic modal
  logics, so as to situate it in the intuitionistic modal logic landscape.

  While (as shown in Section~\ref{subsec:comparison}) we could have used
  a more traditional notion of neighbourhoods to define semantics for $\IMf$,
  we opted to use the novel concept of intuitionistic neighbourhoods because we feel
  that it may pave the way for new mathematical as well as philosophical considerations.
  For example, we could ask if it gives rise to different correspondence results
  than constructive neighbourhood semantics, and whether it is possible to obtain
  Sahlqvist-style correspondence results.
  Philosophically, intuitionistic neighbourhoods embody the idea that the same
  piece of evidence may change (e.g.~become more specific) when moving among
  possible worlds. It would be interesting to see how these can be used for
  intuitionistic reasoning about (multi)agent systems.

%................................................................................
\paragraph{Acknowledgements}
  I am grateful to Ian Shillito for his advice on generalised Hilbert calculi and
  his comments on an earlier version of this paper.
  I would also like to thank the reviewers who commented on an earlier version
  of this paper. Their comments and questions led me to completely rethink the
  paper, resulting in a better motivated paper and a more thorough investigation
  of $\IMf$.

%%%%%%%%%%%%%%%%%%%%%%%%%%%%%%%%%%%%%%%%%%%%%%%%%%%%%%%%%%%%%%%%%%%%%%%%%%%%%%%%%
%\bibliographystyle{plain} 
%\bibliography{../../Latex/biblio.bib}
\bibliographystyle{plainnat}
{\footnotesize
\bibliography{../../../Latex/biblio.bib}{}}

\begin{thebibliography}{48}
\providecommand{\natexlab}[1]{#1}
\providecommand{\url}[1]{\texttt{#1}}
\expandafter\ifx\csname urlstyle\endcsname\relax
  \providecommand{\doi}[1]{doi: #1}\else
  \providecommand{\doi}{doi: \begingroup \urlstyle{rm}\Url}\fi

\bibitem[Areces and Figueira(2009)]{AreFig09}
C.~Areces and D.~Figueira.
\newblock Which semantics for neighbourhood semantics?
\newblock In \emph{Proceedings {IJCAI} 2009}, pages 671--676, 2009.

\bibitem[Bellin et~al.(2001)Bellin, Paiva, and Ritter]{BelPaiRit01}
G.~Bellin, V.~de Paiva, and E.~Ritter.
\newblock Extended {C}urry-{H}oward correspondence for a basic constructive
  modal logic.
\newblock In \emph{Methods for Modalities {II}}, 2001.

\bibitem[Bezhanishvili and de~Jongh(2005)]{BezJon05}
N.~Bezhanishvili and D.~de~Jongh.
\newblock Intuitionistic logic, 2005.
\newblock URL
  \url{https://eprints.illc.uva.nl/id/eprint/200/1/PP-2006-25.text.pdf}.
\newblock ESSLLI course notes.

\bibitem[Blackburn et~al.(2001)Blackburn, Rijke, and Venema]{BRV01}
P.~Blackburn, M.~de Rijke, and Y.~Venema.
\newblock \emph{Modal Logic}.
\newblock Cambridge University Press, Cambridge, 2001.
\newblock \doi{10.1017/CBO9781107050884}.

\bibitem[Bo\v{z}i\'{c} and Do\v{s}en(1984)]{BozDos84}
M.~Bo\v{z}i\'{c} and K.~Do\v{s}en.
\newblock Models for normal intuitionistic modal logics.
\newblock \emph{Studia Logica}, 43:\penalty0 217--245, 1984.
\newblock \doi{10.1007/BF02429840}.

\bibitem[Bull(1965{\natexlab{a}})]{Bul65}
R.~A. Bull.
\newblock A modal extension of intuitionist logic.
\newblock \emph{Notre Dame Journal of Formal Logic}, 6\penalty0 (2):\penalty0
  142--146, 1965{\natexlab{a}}.
\newblock \doi{10.1305/ndjfl/1093958154}.

\bibitem[Bull(1965{\natexlab{b}})]{Bul65b}
R.~A. Bull.
\newblock Some modal calculi based on {IC}.
\newblock In J.N. Crossley and M.A.E. Dummett, editors, \emph{Formal Systems
  and Recursive Functions}, pages 3--7, Amsterdam, North Holland,
  1965{\natexlab{b}}. Elsevier.
\newblock \doi{10.1016/S0049-237X(08)71680-X}.

\bibitem[Bull(1966)]{Bul66}
R.~A. Bull.
\newblock {MIPC} as the formalisation of an intuitionist concept of modality.
\newblock \emph{The Journal of Symbolic Logic}, 31\penalty0 (4):\penalty0
  609--616, 1966.
\newblock \doi{10.2307/2269696}.

\bibitem[Cate et~al.(2009)Cate, Gabelaia, and Sustretov]{CatGabSus09}
B.~ten Cate, D.~Gabelaia, and D.~Sustretov.
\newblock Modal languages for topology: Expressivity and definability.
\newblock \emph{Annals of Pure and Applied Logic}, 159\penalty0 (1):\penalty0
  146--170, 2009.
\newblock \doi{10.1016/j.apal.2008.11.001}.

\bibitem[Chagrov and Zakharyaschev(1997)]{ChaZak97}
A.~Chagrov and M.~Zakharyaschev.
\newblock \emph{Modal Logic}.
\newblock Oxford University Press, Oxford, 1997.

\bibitem[Chellas(1980)]{Che80}
B.~F. Chellas.
\newblock \emph{Modal Logic: An Introduction}.
\newblock Cambridge University Press, Cambridge, 1980.

\bibitem[Dalmonte(2022)]{Dal22}
T.~Dalmonte.
\newblock Wijesekera-style constructive modal logics.
\newblock In D.~Fernández-Duque, A.~Palmigiano, and S.~Pinchinat, editors,
  \emph{Proc.~{AiML} 2022}, pages 281--303, England, 2022. College
  Publications.
\newblock
  aiml:\href{http://www.aiml.net/volumes/volume14/19-Dalmonte.pdf}{volumes/volume14/19-Dalmonte}.

\bibitem[Dalmonte et~al.(2020)Dalmonte, Grellois, and Olivetti]{DalGreOli20}
T.~Dalmonte, C.~Grellois, and N.~Olivetti.
\newblock Intuitionistic non-normal modal logics: A general framework.
\newblock \emph{Journal of Philosophical Logic}, 49:\penalty0 833--882, 2020.
\newblock \doi{10.1007/s10992-019-09539-3}.

\bibitem[Dalmonte et~al.(2021)Dalmonte, Grellois, and Olivetti]{DalGreOli21}
T.~Dalmonte, C.~Grellois, and N.~Olivetti.
\newblock Terminating calculi and countermodels for constructive modal logics.
\newblock In A.~Das and S.~Negri, editors, \emph{Proc.~{TABLEAUX}~2021}, pages
  391--408, Cham, 2021. Springer.
\newblock \doi{10.1007/978-3-030-86059-2_23}.

\bibitem[Das and Marin(2023)]{DasMar23}
Anupam Das and Sonia Marin.
\newblock On intuitionistic diamonds (and lack thereof).
\newblock In R.~Ramanayake and J.~Urban, editors, \emph{Automated Reasoning
  with Analytic Tableaux and Related Methods (TABLEAUX 2023)}, pages 283--301,
  2023.
\newblock \doi{10.1007/978-3-031-43513-3\_16}.

\bibitem[Davoren(2009)]{Dav09}
J.~M. Davoren.
\newblock On intuitionistic modal and tense logics and their classical
  companion logics: Topological semantics and bisimulations.
\newblock \emph{Annals of Pure and Applied Logic}, 161:\penalty0 349--367,
  2009.
\newblock \doi{10.1016/j.apal.2009.07.009}.

\bibitem[de~Groot et~al.(2024)de~Groot, Shillito, and Clouston]{GroShiClo24}
J.~de~Groot, I.~Shillito, and R.~Clouston.
\newblock Semantical analysis of intuitionistic modal logics between {CK} and
  {IK}, 2024.
\newblock arxiv:\href{https://arxiv.org/abs/2408.00262}{2408.00262}.

\bibitem[de~Paiva(2003)]{Pai03}
V.~de~Paiva.
\newblock Natural deduction and context as (constructive) modality.
\newblock In \emph{Proceedings {CONTEXT} 2003}, pages 116--129, 2003.
\newblock \doi{10.1007/3-540-44958-2_10}.

\bibitem[Ewald(1986)]{Ewa86}
W.~B. Ewald.
\newblock Intuitionistic tense and modal logic.
\newblock \emph{Journal of Symbolic Logic}, 51\penalty0 (1):\penalty0 166--179,
  1986.
\newblock \doi{10.2307/2273953}.

\bibitem[Fischer~Servi(1977)]{Fis77}
G.~Fischer~Servi.
\newblock On modal logic with an intuitionist base.
\newblock \emph{Studia Logica}, pages 141--149, 1977.
\newblock \doi{10.1007/BF02121259}.

\bibitem[Fischer~Servi(1980)]{Fis81}
G.~Fischer~Servi.
\newblock Semantics for a class of intuitionistic modal calculi.
\newblock In D.~Chiara and M.~Luisa, editors, \emph{Italian Studies in the
  Philosophy of Science}, pages 59--72, Dordrecht, Netherlands, 1980. Springer.
\newblock \doi{10.1007/978-94-009-8937-5_5}.

\bibitem[Fitch(1948)]{Fit48}
F.~B. Fitch.
\newblock Intuitionistic modal logic with quantifiers.
\newblock \emph{Portugaliae mathematica}, 7\penalty0 (2):\penalty0 113--118,
  1948.
\newblock URL \url{http://eudml.org/doc/114664}.

\bibitem[Flum and Ziegler(1980)]{FluZie80}
J.~Flum and M.~Ziegler.
\newblock \emph{Topological Model Theory}.
\newblock Springer, Berlin, Heidelberg, 1980.
\newblock \doi{10.1007/BFb0097006}.

\bibitem[Gasquet and Herzig(1996)]{GasHer96}
O.~Gasquet and A.~Herzig.
\newblock From classical to normal modal logics.
\newblock In H.~Wansing, editor, \emph{Proof Theory of Modal Logic}, volume~2
  of \emph{Applied Logic Series}, pages 293--311. Springer, Dordrecht, 1996.
\newblock \doi{10.1007/978-94-017-2798-3_15}.

\bibitem[Goldblatt(1981)]{Gol81}
R.~I. Goldblatt.
\newblock {G}rothendieck topology as geometric modality.
\newblock \emph{Mathematical Logic Quarterly}, 27:\penalty0 495--529, 1981.
\newblock \doi{10.1002/malq.19810273104}.

\bibitem[Goldblatt(1993)]{Gol93}
R.~I. Goldblatt.
\newblock \emph{Mathematics of Modality}.
\newblock CSLI publications, Stanford, California, 1993.

\bibitem[Gor\'{e} and Shillito(2020)]{GorShi20}
R.~P. Gor\'{e} and I.~Shillito.
\newblock Bi-intuitionistic logics: A new instance of an old problem.
\newblock In N.~Olivetti, R.~Verbrugge, S.~Negri, and G.~Sandu, editors,
  \emph{Proc.~{AIML} 2020}, pages 269--288, England, 2020. College
  Publications.
\newblock
  aiml:\href{http://www.aiml.net/volumes/volume13/Gore-Shillito.pdf}{volumes/volume13/Gore-Shillito}.

\bibitem[Grefe(1999)]{Gre99}
C.~Grefe.
\newblock \emph{{F}ischer {S}ervi's intuitionistic modal logic and its
  extensions}.
\newblock PhD thesis, Freie Universit{\"a}t Berlin, 1999.

\bibitem[Groot(2022{\natexlab{a}})]{Gro22-inl}
J.~de Groot.
\newblock {H}ennessy-{M}ilner and {V}an {B}enthem for instantial neighbourhood
  logic.
\newblock \emph{Studia Logica}, 110:\penalty0 717--743, 2022{\natexlab{a}}.
\newblock \doi{10.1007/s11225-021-09975-w}.

\bibitem[Groot(2022{\natexlab{b}})]{Gro22gt}
J.~de Groot.
\newblock {G}oldblatt-{T}homason theorems for modal intuitionistic logics.
\newblock In D.~Fernández-Duque, A.~Palmigiano, and S.~Pinchinat, editors,
  \emph{Proc.~{AiML}~2022}, pages 467--490, 2022{\natexlab{b}}.

\bibitem[Groot and Pattinson(2020)]{GroPat20}
J.~de Groot and D.~Pattinson.
\newblock Modal intuitionistic logics as dialgebraic logics.
\newblock In \emph{Proc.~{LICS} 2020}, pages 355--–369, New York, 2020.
  Association for Computing Machinery.
\newblock \doi{10.1145/3373718.3394807}.

\bibitem[Hansen(2003)]{Han03}
H.~H. Hansen.
\newblock Monotonic modal logics.
\newblock Master's thesis, Institute for Logic, Language and Computation,
  University of Amsterdam, 2003.
\newblock
  \href{https://www.illc.uva.nl/Research/Publications/Reports/MoL/}{illc-MoL}:\href{https://eprints.illc.uva.nl/id/eprint/108/2/PP-2003-24.text.pdf}{2003-24}.

\bibitem[Hansen and Kupke(2004)]{HanKup04}
H.~H. Hansen and C.~Kupke.
\newblock A coalgebraic perspective on monotone modal logic.
\newblock \emph{Electronic Notes in Theoretical Computer Science},
  106:\penalty0 121--143, 2004.
\newblock \doi{10.1016/j.entcs.2004.02.028}.

\bibitem[Hansen et~al.(2009)Hansen, Kupke, and Pacuit]{HanKupPac09}
H.~H. Hansen, C.~Kupke, and E.~Pacuit.
\newblock Neighbourhood structures: bisimilarity and basic model theory.
\newblock \emph{Logical Methods in Computer Science}, 5\penalty0 (2), 2009.
\newblock \doi{10.2168/LMCS-5(2:2)2009}.

\bibitem[Kracht and Wolter(1999)]{KraWol99}
M.~Kracht and F.~Wolter.
\newblock Normal monomodal logics can simulate all others.
\newblock \emph{Journal of Symbolic Logic}, 64\penalty0 (1):\penalty0 99--138,
  1999.
\newblock \doi{10.2307/2586754}.

\bibitem[Lewis(1973)]{Lew73}
D.~Lewis.
\newblock \emph{Counterfactuals}.
\newblock Harvard University Press, Harvard, 1973.

\bibitem[Mendler and Paiva(2005)]{MenPai05}
M.~Mendler and V.~de Paiva.
\newblock Constructive {CK} for contexts.
\newblock In L.~Serafini and P.~Bouquet, editors, \emph{Context Representation
  and Reasoning (CRR-2005)}, 2005.

\bibitem[Plotkin and Stirling(1986)]{PloSti86}
G.~Plotkin and C.~Stirling.
\newblock A framework for intuitionistic modal logics: Extended abstract.
\newblock In \emph{Proc.~{TARK} 1986}, pages 399--406, San Francisco, CA, USA,
  1986. Morgan Kaufmann Publishers Inc.

\bibitem[Shillito(2023)]{Shi23}
I.~Shillito.
\newblock \emph{New Foundations for the Proof Theory of Bi-Intuitionistic and
  Provability Logics Mechanized in Coq}.
\newblock PhD thesis, The Australian National University, 2023.
\newblock
  doi:\href{https://openresearch-repository.anu.edu.au/handle/1885/274408}{10.25911/8F2P-AA86}.

\bibitem[Simpson(1994)]{Sim94}
A.~K. Simpson.
\newblock \emph{The Proof Theory and Semantics of Intuitionistic Modal Logic}.
\newblock PhD thesis, University of Edinburgh, 1994.

\bibitem[Sotirov(1984)]{Sot84}
V.~Sotirov.
\newblock Modal theories with intuitionistic logic.
\newblock In \emph{Mathematical Logic, Proc.~Conf.~Math.~Logic Dedicated to the
  Memory of A. A. Markov (1903--1979), Sofia, September 22--23, 1980}, pages
  139--171, 1984.

\bibitem[Tabatabai et~al.(2022)Tabatabai, Iemhoff, and Jalali]{TabIemJal22}
A.~A. Tabatabai, R.~Iemhoff, and R.~Jalali.
\newblock Uniform {L}yndon interpolation for intuitionistic monotone modal
  logic, 2022.
\newblock arxiv:\href{https://arxiv.org/abs/2208.04607v1}{2208.04607v1}.

\bibitem[Troelstra and Schwichtenberg(2000)]{TroSch00}
A.~S. Troelstra and H.~Schwichtenberg.
\newblock \emph{Basic proof theory}.
\newblock Cambridge University Press, Cambridge, 2 edition, 2000.

\bibitem[Wijesekera(1990)]{Wij90}
D.~Wijesekera.
\newblock Constructive modal logics {I}.
\newblock \emph{Annals of Pure and Applied Logic}, 50:\penalty0 271--301, 1990.
\newblock \doi{10.1016/0168-0072(90)90059-B}.

\bibitem[Wijesekera and Nerode(2005)]{WijNer05}
D.~Wijesekera and A.~Nerode.
\newblock Tableaux for constructive concurrent dynamic logic.
\newblock \emph{Annals of Pure and Applied Logic}, 135:\penalty0 1--72, 2005.
\newblock \doi{10.1016/j.apal.2004.12.001}.

\bibitem[Wolter and Zakharyaschev(1997)]{WolZak97}
F.~Wolter and M.~Zakharyaschev.
\newblock The relation between intuitionistic and classical modal logics.
\newblock \emph{Algebra and Logic}, 36\penalty0 (2):\penalty0 73--92, 1997.
\newblock \doi{10.1007/BF02672476}.

\bibitem[Wolter and Zakharyaschev(1998)]{WolZak98}
F.~Wolter and M.~Zakharyaschev.
\newblock Intuitionistic modal logics as fragments of classical bimodal logics.
\newblock In E.~Orlowska, editor, \emph{Logic at Work, Essays in honour of
  {H}elena {R}asiowa}, pages 168--186. Springer--Verlag, 1998.

\bibitem[Wolter and Zakharyaschev(1999)]{WolZak99}
F.~Wolter and M.~Zakharyaschev.
\newblock Intuitionistic modal logic.
\newblock In A.~Cantini, E.~Casari, and P.~Minari, editors, \emph{Logic and
  Foundations of Mathematics: Selected Contributed Papers of the Tenth
  International Congress of Logic, Methodology and Philosophy of Science},
  pages 227--238, Dordrecht, Netherlands, 1999. Springer.
\newblock \doi{10.1007/978-94-017-2109-7_17}.

\end{thebibliography}

\appendix
%%%%%%%%%%%%%%%%%%%%%%%%%%%%%%%%%%%%%%%%%%%%%%%%%%%%%%%%%%%%%%%%%%%%%%%%%%%%%%%%%
\section{Multimodal IK}\label{app:IK}

  We sketch the syntax, proof system and semantics of a multimodal
  version of the intuitionistic modal logic $\log{IK}$.
  
\begin{definition}\label{def:lan}
  We fix a countable set $\Prop$ of propositional variables.
  For a set $J$ of modal indices, define
  $\mathcal{L}_J$ be the language generated by the grammar
  $$
    \phi ::= p_i
      \mid \bot
      \mid \phi \wedge \phi
      \mid \phi \vee \phi
      \mid \phi \to \phi
      \mid \Box_j\phi
      \mid \Diamond_j\phi,
  $$
  where $p_i$ ranges over $\Prop$ and $j \in J$.
\end{definition}

  We essentially define $\log{IK}_J$ via the axiomatisation used by
  Fischer Servi~\cite{Fis77},
  Plotkin and Sterling~\cite{PloSti86}
  and Simpson~\cite{Sim94}, but akin to Definition~\ref{def:ghc}
  we present it as a generalised Hilbert calculus.

\begin{definition}
  Let $\operatorname{Ax}$ be some axiomatisation of intuitionistic logic.
  Let $\Ax$ be the collection of substitution instances of axioms
  from $\operatorname{Ax}$ and of the following axioms:
  \begin{multicols}{2}
  \begin{enumerate}
  \setlength{\itemindent}{1em}\itemsep=0em
    \item[($\AKbox_j$)] \label{ax:int-1}
           $\Box_j(p \to q) \to (\Box_j p \to \Box_j q)$
    \item[($\ACdia_j$)] \label{ax:int-4}
           $\Diamond_j(p \vee q) \to (\Diamond_j p \vee \Diamond_j q)$
    \item[($\ANdia_j$)] \label{ax:int-3} $\neg\Diamond_j\bot$
    \item[($\AKdia_j$)] \label{ax:int-2}
           $\Box_j(p \to q) \to (\Diamond_j p \to \Diamond_j q)$
    \item[($\ADualI_j$)] \label{ax:int-5}
           $(\Diamond_j p \to \Box_j q) \to \Box_j(p \to q)$
    \item[] \phantom{.}
  \end{enumerate}
  \end{multicols}
  \noindent
  Define the logic $\log{IK}_J$ as the collection of sequents derivable
  from the following rules:
  the \emph{element rule} and the \emph{axiom rule}
  $$
    \dfrac{}{\Gamma \vdash \phi} \; (\REl),
    \qquad
    \dfrac{}{\Gamma \vdash \psi} \; (\RAx),
  $$
  where $\phi \in \Gamma$ and $\psi \in \Ax$,
  and \emph{modus ponens} and \emph{neccesitation}
  $$
    \dfrac{\Gamma \vdash \phi \quad \Gamma \vdash \phi \to \psi}
          {\Gamma \vdash \psi} \; (\RMP),
    \qquad
    \dfrac{\emptyset \vdash \phi}
          {\Gamma \vdash \Box_j\phi} \; (\RNec).
  $$
  We write $\Gamma \vdash_{\log{IK}_J} \phi$ if $\Gamma \vdash \phi$ is in
  $\log{IK}_J$.
\end{definition}

  If we pick $J$ to be a singleton set, then we simply obtain $\IK$.
  The Kripke semantics of $\log{IK}$, introduced
  independently by Ewald~\cite{Ewa86}, Fischer Servi~\cite{Fis81}
  and Plotkin and Stirling~\cite{PloSti86}, is given as follows.

\begin{definition}\label{def:kf}
  An \emph{intuitionistic multimodal frame} of \emph{$\IK_J$-frame}
  is a tuple $(W, \leq, \{ R_j \mid j \in J \})$ consisting
  of a set $W$, a preorder $\leq$ on $W$, and binary relation $R_j$ on $W$
  for each $j \in J$ satisfying:
  \begin{enumerate}
    \item \label{it:kf-1}
          If $w \leq v$ and $wR_ju$ then there exists a $x \in W$ such that
          $yRw$ and $z \leq w$;
    \item If $wR_ju$ and $u \leq x$ then there exists a $v \in W$ such that
          $w \leq v$ and $vR_jx$.
  \end{enumerate}
  These conditions are depicted in Figure~\ref{fig:IK-1}.
  \begin{figure}[h!]
    \centering
    \begin{tikzpicture}[scale=.75]
      %% nodes
        \node (w) at (0,0)   {$w$};
        \node (u) at (2,0)   {$u$};
        \node (v) at (0,1.5) {$v$};
        \node (x) at (2,1.5) {$x$};
      %% edges
        \draw[-latex] (w) to (v);
        \draw[-Circle] (w) to node[above]{\footnotesize{$R_j$}} (u);
        \draw[dashed, -latex] (u) to (x);
        \draw[dashed, -Circle] (v) to node[above]{\footnotesize{$R_j$}} (x);
      %% nodes
        \node (w) at (4,0)   {$w$};
        \node (u) at (6,0)   {$u$};
        \node (v) at (4,1.5) {$v$};
        \node (x) at (6,1.5) {$x$};
      %% edges
        \draw[dashed, -latex] (w) to (v);
        \draw[-Circle] (w) to node[above]{\footnotesize{$R_j$}} (u);
        \draw[-latex] (u) to (x);
        \draw[dashed, -Circle] (v) to node[above]{\footnotesize{$R_j$}} (x);
    \end{tikzpicture}
    \caption{Depiction of the conditions of an $\IK_J$-frame.}
    \label{fig:IK-1}
  \end{figure}
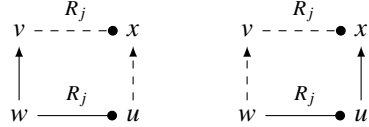
  An \emph{intuitionistic multimodal model} is an intuitionistic multimodal
  frame $(X, \leq, \{ R_j \mid j \in J \})$ together with a valuation $V$
  that assigns to each proposition letter $p_i$ and upset of $(X, \leq)$.
  The interpretation of a formula $\phi$ in a model
  $\mo{M} = (X, \leq, \{ R_j \mid j \in J \}, V)$ is
  defined recursively via
  \begin{align*}
    \mo{M}, x \Vdash p &\iff x \in V(p) \qquad \text{(where $p \in \Prop$)} \\
    \mo{M}, x \Vdash \bot &\phantom{\iff} \text{never} \\
    \mo{M}, x \Vdash \psi \wedge \chi
      &\iff \mo{M}, x \Vdash \psi \text{ and } \mo{M}, x \Vdash \chi \\
    \mo{M}, x \Vdash \psi \vee \chi
      &\iff \mo{M}, x \Vdash \psi \text{ or } \mo{M}, x \Vdash \chi \\
    \mo{M}, x \Vdash \psi \to \chi
      &\iff \forall y \in X(\text{if } x \leq y \text{ and } \mo{M}, x \Vdash \psi \text{ then } \mo{M}, y \Vdash \chi) \\
    \mo{M}, x \Vdash \Box_j\psi &\iff \forall y, z \in X(\text{if } x \leq y
      \text{ and } yR_jz \text{ then } \mo{M}, z \Vdash \psi) \\
    \mo{M}, x \Vdash \Diamond_j\psi &\iff \exists y \in X \text{ s.t. }
       xR_jy \text{ and }\mo{M}, y \Vdash \psi
  \end{align*}
  A world $x$ in a model $\mo{M}$ satisfies a sequent $\Gamma \vdash \phi$
  if $\mo{M}, x \Vdash \phi$ whenever $x$ satisfies all formulas in $\Gamma$.
\end{definition}
  
  It is well known that the intuitionistic modal logic $\IK$ with a single pair
  of modalities is sound and strongly complete with respect to the class of
  intuitionistic multimodal frames from Definition~\ref{def:kf} where $J$ is
  a singleton set~\cite[Section~3.3]{Sim94}.
  The following is a multimodal analogue of this,
  and can be proven using a canonical model construction.

\begin{theorem}\label{thm:ik-compl}
  The logic $\log{IK}_J$ is sound and strongly complete with respect to the class
  of $\IK_J$-frames.
\end{theorem}

\end{document}